\documentclass[11pt,a4paper,reqno]{amsart}

\usepackage{amsmath,amssymb}
\usepackage{amsbsy,amsfonts,amsthm}
\usepackage[T1]{fontenc}
\usepackage[dvips]{graphicx}
\usepackage{verbatim} 

\usepackage[square,numbers,comma,sort&compress]{natbib}
\bibpunct{[}{]}{,}{n}{}{;}

\DeclareFontFamily{OT1}{pzc}{}
\DeclareFontShape{OT1}{pzc}{m}{it}{<-> s * [1.15] pzcmi7t}{}
\DeclareMathAlphabet{\mathpzc}{OT1}{pzc}{m}{it}

\DeclareMathOperator{\nest}{\odot}
\DeclareMathOperator*{\Nest}{\bigodot}
\DeclareMathOperator{\len}{length}

\DeclareMathOperator*{\argmin}{arg\,min}

\newcommand{\key}[1]{\emph{#1}}
\newcommand{\mat}[1]{\mathsf{#1}}

\newcommand{\powset}[1]{2^{#1}}
\newcommand{\nepowset}[1]{\mathcal{P}_{\ge1}(#1)}

\let\emptyset\varnothing

\newcommand{\reduc}[1]{\mathpzc{r}^{#1}}
\newcommand{\Reduc}[1]{\mathpzc{R}^{\!#1}}
\newcommand{\dress}[2]{\delta^{#1}_{#2}}
\newcommand{\Dress}[2]{\Delta^{\!#1}_{#2}}
\newcommand{\sumdress}[2]{\mathpzc{d}^{#1}_{#2}}
\newcommand{\sumDress}[2]{\mathpzc{D}^{#1}_{#2}}
\newcommand{\fps}[2]{#1\langle\!\langle#2\rangle\!\rangle}
\newcommand{\dsig}{\mathpzc{K}}
\newcommand{\widesim}[2][1.7]{\mathrel{\overset{#2}{\scalebox{#1}[1]{$\sim$}}}}
\newcommand{\Kequiv}{\widesim{\dsig}}
\newcommand{\KLequiv}{\widesim{(\dsig,l)}}
\newcommand{\locdepth}[1]{\ell^{#1}}
\newcommand{\G}{\mathcal{G}}
\newcommand{\wt}{\mathrm{W}} 
\newcommand{\ew}{\mathsf{w}}
\newcommand{\V}{\mathcal{V}}
\newcommand{\E}{\mathcal{E}}
\newcommand{\gentree}{\textsl{T}}
\newcommand{\sigtree}{\textsl{S}}
\newcommand{\hedge}{\textsl{H}}
\newcommand{\rootnode}[1]{\textsl{#1.root}}

\newtheorem{theorem}{Theorem}[section]
\newtheorem{prop}[theorem]{Proposition}
\newtheorem{lem}[theorem]{Lemma}
\newtheorem{cor}[theorem]{Corollary}

\theoremstyle{remark}

\newtheorem{definition}{Definition}[section] 

\newtheorem{remark}[definition]{Remark}
\newtheorem{example}{Example}[section]

\graphicspath{{./images/}}

\begin{document}

\title{A family of partitions of the set \\of walks on a directed graph}
\author{Simon Thwaite}
     \address{Faculty of Physics, Ludwig Maximilian University of Munich, Theresienstrasse 37, 80333 Munich, Germany}
\email{simon.thwaite@physik.uni-muenchen.de}

\begin{abstract}
We present a family of partitions of $W_\G$, the set of walks on an arbitrary directed graph $\G$. Each partition in this family is identified by an integer sequence $\dsig$, which specifies a collection of cycles on $\G$ with a certain well-defined structure. We term such cycles resummable, and a walk that does not traverse any such cycles $\dsig$-irreducible. For a given value of $\dsig$, the corresponding partition of $W_\G$ consists of a collection of cells that each contain a single $\dsig$-irreducible walk $i$ plus all walks that can be formed from $i$ by attaching one or more resummable cycles to its vertices. Each cell of the partition is thus an equivalence class containing walks that can be transformed into one another by attaching or removing a collection of resummable cycles.

We characterise the entire family of partitions of $W_\G$ by giving explicit expressions for the structure of the $\dsig$-irreducible walks and the resummable cycles for arbitrary values of $\dsig$. We demonstrate how these results can be exploited to recast the sum over all walks on a directed graph as a sum over dressed $\dsig$-irreducible walks, and discuss the applications of this reformulation to matrix computations.
\end{abstract}
\maketitle

\section{Introduction}\label{sec:Introduction}
     Graphs are fundamental mathematical structures that find applications in virtually every discipline. Given a particular graph, a natural object of study is the family of walks that it supports. In addition to their intrinsic mathematical interest \cite{Lawler2010}, walks on graphs find applications in a wide range of fields, including biology \cite{Berg1993}, quantum physics \cite{Kempe2003,Venegas-Andraca2012}, and statistical inference \cite{Malioutov2006,Chandrasekaran2008}. In many applications, not only the number but also the structure of the walks on a graph is of interest: for example, self-avoiding walks and their variants play an important role in statistical mechanics and polymer science \cite{Madras1996}. A comprehensive understanding of the structure of walks on a graph, and how this depends on the connectivity structure of the graph itself, is thus an interesting topic of research with many possible applications.

     One recent example of such an application in the field of applied mathematics is the method of path-sums: a technique for evaluating matrix functions via resummations of walks on weighted directed graphs \cite{Giscard2013}. The method of path-sums is based around two main ideas. Firstly, by exploiting a mapping between matrix multiplication and walks on a weighted directed graph $\G$, the Taylor series for a given matrix function such as the matrix exponential or matrix inverse is recast as a sum over all walks on $\G$. Secondly, by identifying families of cycles (that is, walks that start and finish on the same vertex, but do not revisit that vertex in between) that repeatedly appear in the set of walks, certain infinite geometric series appearing in this sum are exactly resummed into closed form. The effect is to reduce the sum over all walks to a sum over a certain subset of `irreducible' walks, each of which is `dressed' so as to include contributions from the families of resummed cycles. This produces a new series for the original matrix function, which can be interpreted as the result of exactly resumming various infinite geometric series appearing within the Taylor series for the function.

     The exact form of the irreducible walks and the way in which they are dressed are determined by the families of cycles that are resummed. The method of path-sums as presented in \cite{Giscard2013} introduced two examples. The first involves resumming all self-loops (i.e.~all edges that connect a vertex to itself). In this case the irreducible walks are those that do not contain any self-loops, and the dressing procedure amounts to adding all possible combinations of self-loops to a loopless walk. Applying this result to the computation of the matrix exponential reproduces the Dyson series, leading to the interesting conclusion that the Dyson series can be interpreted as the Taylor series with all terms corresponding to self-loops having been resummed. The second example involves resumming all possible cycles. In this case the remaining summation runs over simple paths on $\G$, where a simple path (also known as a self-avoiding walk) is a walk that is forbidden from visiting any vertex more than once. Each simple path in this sum is dressed by all possible combinations of cycles off the vertices it visits. The corresponding expression for the matrix function takes the form of a finite sum of finite continued fractions, known as the path-sum expression for the function \cite{Giscard2013}.

\begin{figure}
\centering
\includegraphics[width=\textwidth]{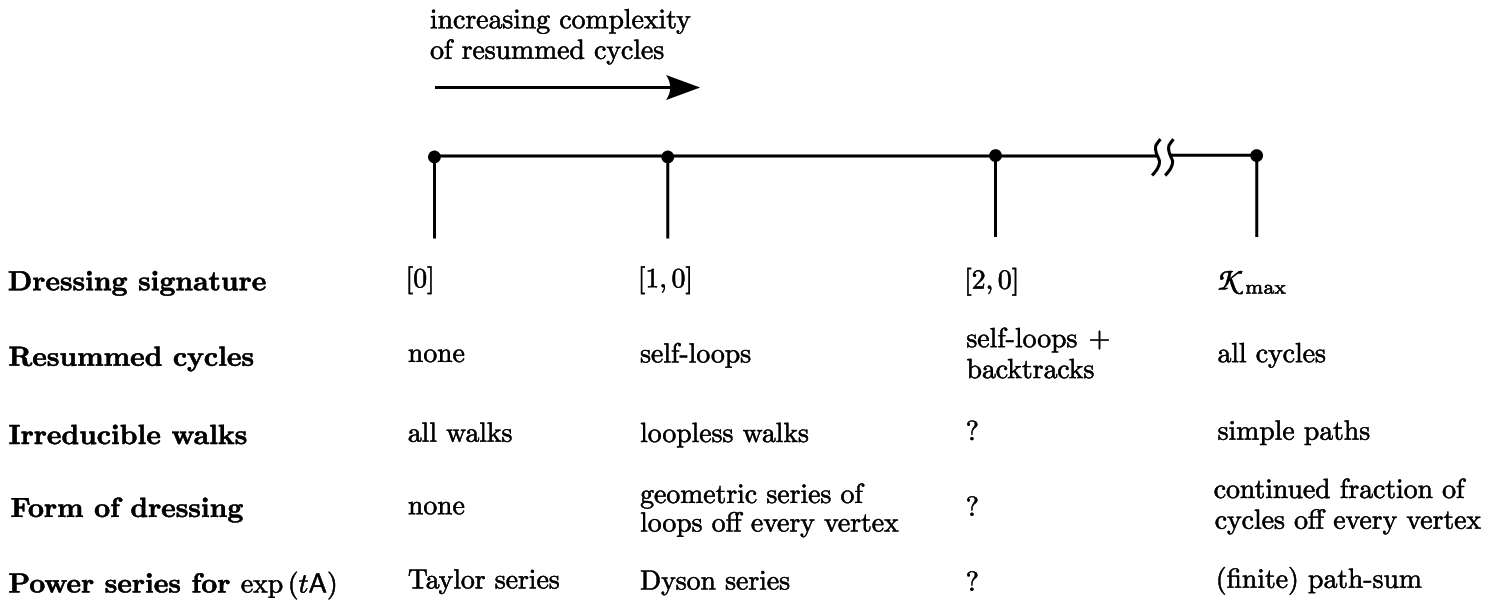}
       \caption{The hierarchy of cycle resummation schemes discussed in the text. Each resummation scheme is indexed by a dressing signature $\dsig$, which identifies the structure of the cycles that are exactly resummed within that dressing scheme. Associated with each resummation scheme is a set of irreducible walks and a dressing scheme, which are defined such that a sum over the irreducible walks, each dressed appropriately dressed by resummable cycles, reproduces the sum over all walks on $\G$. Obtaining explicit expressions for the irreducible walks and dressing scheme for an arbitrary dressing signature is the main goal of this article. The maximum dressing signature $\dsig_{\,\mathrm{max}}$ depends on the size and connectivity structure of the underlying digraph, and is discussed further in Section \ref{sec:DressingSignatures}.}
\label{fig:ResummationSpectrum}
\end{figure}

     Many other possible resummations exist besides these two examples. Indeed, we can imagine a hierarchy of resummations, ordered by increasing complexity of the cycles that are exactly resummed (see Figure \ref{fig:ResummationSpectrum}). Each member of this hierarchy is indexed by a \key{dressing signature}: a sequence $\dsig$ of non-negative integers that describes the structure of the cycles that are exactly resummed within that scheme. The lowest non-trivial member of this hierarchy corresponds to the resummation of self-loops ($\dsig = [1,0]$), while the highest member corresponds to the resummation of all possible cycles. Moving up the hierarchy corresponds to progressively including longer and more complex cycles in the resummation: for example backtracks (i.e.~cycles of the form $\alpha\rightarrow \beta\rightarrow \alpha$, corresponding to $\dsig = [2,0]$), triangles ($\dsig = [3,0]$), triangles with self-loops off their internal vertices ($\dsig = [3,1,0]$), etc. Associated with the resummation scheme indexed by a given dressing signature $\dsig$ is a family of $\dsig$-irreducible walks (i.e.~walks that do not contain any resummable cycles) and a dressing scheme. These are defined such that a sum over the $\dsig$-irreducible walks -- each dressed by combinations of resummable cycles as prescribed by the dressing scheme -- reproduces the sum over all walks without any over- or under-counting.

     When combined with the mapping between matrix functions and walks on graphs, this hierarchy of cycle resummations induces a family of series for a given matrix function. Each series in this family has the form of a sum over irreducible walks on the underlying graph, and corresponds to the Taylor series with certain infinite geometric series of terms having been exactly resummed. Apart from the Dyson series, which corresponds to the resummation of self-loops, the partially resummed series remain unknown. Their form, convergence behaviour, and mathematical properties are therefore interesting open questions with possible applications to matrix computations.

     In order to systematically study this hierarchy of resummations, a comprehensive understanding of the form of the irreducible walks and dressing scheme for a given dressing signature is required. Motivated by this goal, we study in this article the structure of the set of walks on an arbitrary directed graph $\G$ with respect to nesting, a product operation between walks that provides for unique factorisation \cite{Giscard2012}. We present two main results. We show firstly that for any dressing signature $\dsig$, which may be freely chosen, the family of walk sets produced by taking each $\dsig$-irreducible walk and dressing it by all combinations of resummable cycles form a partition of the set of all walks. This proves that the decomposition of the set of all walks into sets of dressed $\dsig$-irreducible walks does not lead to any walk being double-counted or missed out. Secondly, we present explicit expressions for the set of $\dsig$-irreducible walks and the dressing scheme for an arbitrary dressing signature $\dsig$, thus allowing the sum over all walks on $\G$ to be reformulated as a sum over dressed $\dsig$-irreducible walks.

\subsection{Article summary}

     The article is structured as follows. In Section \ref{sec:S2:RequiredConcepts} we summarize the required background material. In addition to standard concepts from graph theory -- in which case our treatment serves to establish our terminology and notation -- this section includes a definition of the nesting product (\S\ref{sec:NestingProduct}) and a discussion of a scheme for canonically representing walks on a graph via a structure we term a \key{syntax tree} (\S\ref{sec:SyntaxTree}).

     In Section \ref{sec:S3:FamilyOfPartitions} we develop our main results with a bare minimum of technical details, with a view to providing an overview for readers who wish to avoid being swamped by
     minutiae on a first reading. We begin in \S\ref{sec:DressingSignatures} by formalising the concept of a dressing signature. This paves the way for the introduction of the two operators which are of central importance to our approach: the \key{walk reduction operator} $\Reduc{\dsig}$, which projects any walk to its $\dsig$-irreducible core walk by removing all resummable cycles from it, and the \key{walk dressing operator} $\Dress{\dsig}{\G}$, which maps a $\dsig$-irreducible walk $i$ to the set of walks formed by adding all possible configurations of resummable cycles on $\G$ to $i$. In Section \ref{sec:32:FamilyOfPartitions} we state the required properties of $\Reduc{\dsig}$ and $\Dress{\dsig}{\G}$, and show that the family of walk sets produced by applying $\Dress{\dsig}{\G}$ to each of the $\dsig$-irreducible walks forms a partition of the set of all walks. In Section \ref{sec:EquivRelation} we discuss the equivalence relation associated with this partition.

     Section \ref{sec:S4:WalkReductionWalkDressing} contains the main technical content of this article. This section has two main aims: firstly, to present explicit definitions for the operators $\Reduc{\dsig}$ and $\Dress{\dsig}{\G}$ (found in \S\ref{sec:WalkReductionOperators} and \S\ref{sec:WalkDressingOperator} respectively) and secondly, to prove that these definitions satisfy the properties assumed of them in Section \ref{sec:S3:FamilyOfPartitions} (shown in \S\ref{sec:WalkReductionOperators} and \S\ref{subsec:RDeltaInverses}). A further key result is the explicit expression for the set of $\dsig$-irreducible walks for an arbitrary dressing signature $\dsig$, which we provide in \S\ref{subsec:StructureOfIK}.

     In Section \ref{sec:S5:DressedWalkSums} we present an application of the partitioning of the set of all walks into a collection of sets of dressed $\dsig$-irreducible walks. In \S\ref{sec:51:ResummedCharSeries} we show how this result can be exploited to recast the sum over all walks as a sum over dressed $\dsig$-irreducible walks for an arbitrary dressing signature $\dsig$. In \S\ref{sec:52:ResummedWeightSeries} we extend this result to the sum of all walk weights on a weighted digraph, and discuss the applications of this result to matrix computations.

     We conclude in Section \ref{sec:S6:Conclusion}. 
\section{Required concepts}\label{sec:S2:RequiredConcepts}
     This section provides an overview of the material required to develop the main results of this article. Some of the material included in this section is common knowledge in graph theory, in which case we follow standard references such as \cite{Biggs1993, Diestel2010}.
\subsection{Graphs}
     A directed graph or digraph $\G$ is a pair of sets $\left(\V,\E\right)$ such that the elements of the edge set $\E$ (typically referred to as directed edges or arrows) are ordered pairs of elements of the vertex set $\V$. A digraph is infinite if it has infinitely many edges, vertices, or both, or finite otherwise. An edge connecting a vertex to itself will be called a loop, and two or more edges that connect the same vertices in the same direction are called multiple edges. Throughout this article we consider finite directed graphs that may contain loops, but not multiple edges. We denote the vertices of $\G$ by Greek letters $\alpha,\beta,\ldots$, and an edge that leads from vertex $\mu$ to vertex $\nu$ by $(\mu\nu)$. We denote the digraph obtained by deleting from $\G$ the vertices $\alpha,\beta,\ldots,\omega$ and all edges incident on these vertices by $\G\backslash\alpha\beta\cdots\omega$.\\
\subsection{Walks, paths, and cycles}\label{subsec:WalksPathsCycles}
     \subsubsection*{Walks.} A walk $w$ of length $n\ge 1$ from $\alpha$ to $\omega$ on $\G$ is a sequence $(\alpha\mu_2), (\mu_2\mu_3), \ldots , (\mu_n\omega)$ of $n$ contiguous edges. If it happens that $\omega = \alpha$ then $w$ is a closed walk; walks that are not closed are open. We represent $w$ either by its edge sequence $(\alpha\mu_2)(\mu_2\mu_3)\cdots(\mu_n\omega)$ or by its vertex string $\alpha\mu_2\cdots\mu_{n}\omega$. The head and tail vertices of $w$ are $h(w) = \alpha\equiv \mu_1$ and $t(w) = \omega \equiv \mu_{n+1}$, respectively; the remaining vertices $\mu_2,\ldots,\mu_{n}$ are the \key{internal vertices} of $w$. The set of all walks from $\alpha$ to $\omega$ on $\G$ will be denoted by $W_{\G;\alpha\omega}$, while the set of all walks on $\G$ will be denoted by $W_\G = \cup_{\alpha,\omega} W_{\G;\alpha\omega}$.\\

     In addition to walks of positive length, we define for every vertex $\alpha$ of $\G$ a trivial walk of length zero, which we denote by $(\alpha)$. Then $h\left((\alpha)\right) = t\left((\alpha)\right) = \alpha$, so $(\alpha)$ is a closed walk off the vertex $\alpha$, and is thus a member of $W_{\G;\alpha\alpha}$. The set of all trivial walks on $\G$ will be denoted by $T_\G$.\\

     Finally, we define a zero walk $0$, which has undefined length and satisfies $h(0) = t(0) = 0$. The zero walk is a member of neither $W_\G$ nor $T_\G$.

\subsubsection*{Paths.}
     A \key{simple path} is a (possibly trivial) walk whose vertices are all distinct. The set of all simple paths from $\alpha$ to $\omega$ on $\G$ will be denoted by $P_{\G;\alpha\omega}$, and the set of all simple paths on $\G$ by $P_\G = \cup_{\alpha,\omega} P_{\G;\alpha\omega}$. When $\G$ is finite, both of these sets are finite. Note that $P_{\G;\alpha\alpha}$ is a set with a single member: the trivial walk $(\alpha)$.

\subsubsection*{Cycles.}
     A \key{cycle} off a vertex $\alpha$ is a non-trivial walk that starts and finishes on $\alpha$, but does not revisit $\alpha$ in between. Every cycle is a closed walk, but -- since a closed walk off $\alpha$ may revisit $\alpha$ any number of times -- not every closed walk is a cycle. The set of all cycles off $\alpha$ on $\G$ will be denoted by $C_{\G;\alpha}$, and the set of all cycles on $\G$ by $C_\G = \cup_\alpha C_{\G;\alpha}$. These sets are typically infinite. A cycle that visits each of its internal vertices only once will be termed a simple cycle. A simple cycle is a loop if it is of length 1, a backtrack if it is of length 2, or a triangle, square, etc.~if it is of length $3,4,\ldots$. (Simple cycles are thus also known as self-avoiding polygons.) The set of all simple cycles off a given vertex $\alpha$ on $\G$ will be denoted by $S_{\G;\alpha}$ and the set of all simple cycles on $\G$ by $S_\G = \cup_\alpha S_{\G;\alpha}$. Further, the set of all simple cycles off $\alpha$ with length $L$ will be denoted by $S_{\G;\alpha;L}$. These sets are finite when $\G$ is finite. Since we take the convention that the trivial walk $(\alpha)$ is not a cycle, neither $C_\G$ nor $S_\G$ contain any trivial walks.

\subsection{The nesting product}\label{sec:NestingProduct}
     In order to combine shorter walks into longer ones and decompose longer walks into shorter ones, it is useful to define a product operation between walks. A key guiding principle in constructing a product operation is that the decomposition of a given walk under the product should be unique. In this section we introduce nesting, a product operation between walks which provides for unique factorisation \cite{Giscard2012}.
     \begin{definition}[Nestable couples]\label{defn:NestableCouples} Let $w_1$ and $w_2$ be walks on a digraph $\G$, with vertex sequences $w_1 = \alpha_1 \alpha_2 \cdots \alpha_{m+1}$ and $w_2 = \beta_1 \beta_2\cdots \beta_{n+1}$ respectively. Then the pair $(w_1, w_2)$ is \key{nestable} if and only if both of the following conditions are true:
\begin{enumerate}
\item[(i)] $w_2$ is closed (i.e.~$\beta_{n+1} = \beta_1$),
     \item[(ii)] $w_1$ visits $\beta_1$ at least once, and $\{\alpha_1,\alpha_2,\ldots,\alpha_{j-1} \} \cap \{\beta_1,\beta_2,\ldots,\beta_{n}\} = \emptyset$, where $1\le j\le m+1$ is the index of the first appearance of $\beta_1$ in $w_1$. In other words, no vertex that $w_1$ visits before reaching $\beta_1$ for the first time is also visited by $w_2$.
\end{enumerate}
Note that we do not exclude the possibility that one or both of $w_1$ and $w_2$ may be trivial.
\end{definition}
\begin{remark}
     Let $(w_1, w_2)$ be a nestable pair of walks, with vertex sequences $w_1 = \alpha_1 \alpha_2 \cdots \alpha_{m+1}$ and $w_2 = \beta_1 \beta_2\cdots \beta_{1}$. Let $j$ be the index of the first appearance of $\beta_1$ in $w_1$. It follows from condition (ii) of Definition \ref{defn:NestableCouples} that $w_2$ is a walk on $\G \backslash \alpha_1\cdots\alpha_{j-1}$.
\end{remark}

\begin{definition}[Nesting product]\label{defn:NestingProduct}
     Let $w_1$ and $w_2$ be two walks on a digraph $\G$. If the couple $(w_1,w_2)$ is not nestable, we define the nesting product $w_1\nest w_2 = 0$, where $0$ is the zero walk. Otherwise, let $w_1$ and $w_2$ have vertex sequences $w_1 = \alpha_1 \alpha_2 \cdots \alpha_{m+1}$ and $w_2 = \beta_1 \beta_2 \cdots \beta_{n}\beta_1$, and let $k$ be the index of the final appearance of $\beta_1$ in $w_1$. Then we define the nesting product to be
\begin{align*}
w_1 \nest w_2 = \alpha_1\alpha_2\cdots\alpha_{k-1}\beta_1\beta_2\cdots\beta_{n}\beta_1
\alpha_{k+1}\cdots\alpha_{m+1}.
\end{align*}
     The walk $w_1\nest w_2$ has length $m + n$ and is said to consist of $w_2$ nested into $w_1$ (off vertex $\beta_1$). The vertex sequence of $w_1 \nest w_2$ is formed by replacing the final appearance of $\beta_1$ in $w_1$ by the entire vertex sequence of $w_2$.
\end{definition}

\begin{figure}
  \includegraphics[scale=1.15]{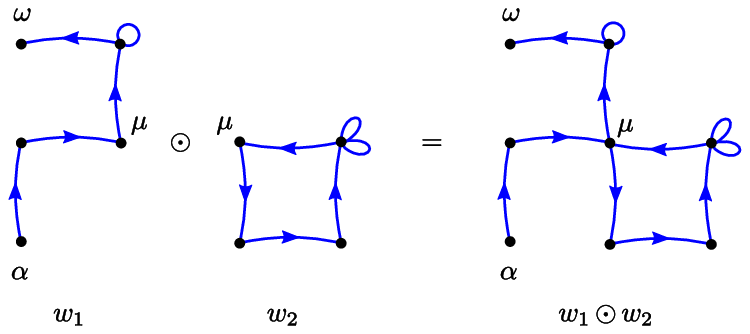}\hfill
  \includegraphics[scale=1.15]{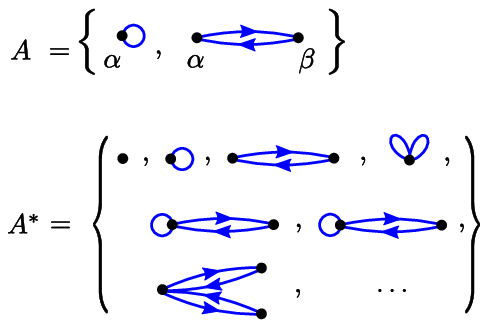}
  \caption{(Left) An example of nesting: $w_1$ is a walk from $\alpha$ to $\omega$, while $w_2$ is a closed walk off $\mu$, one of the vertices visited by $w_1$. Then $w_1 \nest w_2$ is the walk obtained by attaching $w_2$ to $\mu$. (Right) A graphical representation of the Kleene star of the set $A = \{\alpha\alpha,\alpha\beta\alpha\}$. The set $A^*$ contains all walks that can be formed by concatenating zero or more copies of $\alpha\alpha$ and $\alpha\beta\alpha$ together. The bare vertex represents the trivial walk $(\alpha)$. Note that the two walks in the second row of $A^*$ differ in the order in which they traverse the loop and the backtrack: one represents the walk $\alpha\alpha\nest \alpha\beta\alpha = \alpha\alpha\beta\alpha$, the other $\alpha\beta\alpha \nest \alpha\alpha = \alpha\beta\alpha\alpha$.}\label{fig:NestingAndKleeneStar}
\end{figure}

\noindent An example of the nesting operation is given in Figure~\ref{fig:NestingAndKleeneStar}.\\

\noindent In general, the nesting operation is neither commutative nor associative: for example, $11\nest 131 = 1131$ while $131 \nest 11 = 1311$, and $\left(12\nest 242\right)\nest 11 = 11242$ while $12\nest\left(242\nest 11\right) = 0$. For convenience we declare the nesting operator to be (notationally) left-associative: thus $a\nest b \nest c$ means $ (a\nest b) \nest c$, and explicit parentheses are required to indicate $a\nest(b\nest c)$.

In the case that $w_1$ and $w_2$ are closed walks off the same vertex, then nesting is equivalent to concatenation: for example, $1211 \nest 131 = 121131$. It is thus natural to define nesting powers of a closed walk $w$ as
\begin{align}
w^n = \begin{cases}
h(w) & \text{if $n = 0$,}\\
w^{n-1} \nest w & \text{otherwise.}
\end{cases}
\end{align}
     We further introduce the following shorthand notation for the nesting product of a sequence of closed walks off the same vertex:
\begin{align}
\Nest_{j=1}^N\, w_j = w_1 \nest w_2 \nest \cdots \nest w_N,
\end{align}
with the empty product defined to be the trivial walk off the relevant vertex.

\subsection{Nesting sets of walks and the Kleene star}
We extend the nesting product from individual walks to sets of walks as follows.
\begin{definition}[Nesting sets of walks]\label{defn:NestingSets}
     Given two sets of walks $A$ and $B$, we define $A \nest B$ to be the direct product of $A$ and $B$ with respect to the nesting operation:
\begin{align}
A \nest B = \big\{ a \nest b : a \in A \text{ and } b \in B\big\}.
\end{align}
It follows from this definition that if either $A$ or $B$ is empty, then $A\nest B$ is also empty. Further, nesting is distributive over set union: that is, $A \nest \left(B \cup C\right) = \left(A \nest B\right) \cup \left(A \nest C\right)$.
\end{definition}
We adopt the convention that -- in the absence of parentheses in an expression -- the nesting operator has a higher precedence than set union. The expression $A \nest B \, \cup \,C$ is thus to be interpreted as $(A \nest B) \,\cup \,C$, and explicit parentheses are required to indicate $A \nest\, (B\, \cup\, C)$.\\

Since nesting between closed walks off the same vertex is equivalent to concatenation, a set of closed walks may be nested with itself. This allows powers of a set of closed walks to be defined.

\begin{definition}[Nesting power of a set of walks]\label{defn:NestingPower}
     For a (possibly empty) set $A_{\alpha}$ of closed walks off a vertex $\alpha$, we define $A_{\alpha}^0 = \big\{(\alpha)\big\}$ and $A_{\alpha}^n = A_{\alpha}^{n-1} \nest A_{\alpha}$ for $n = 1,2,\ldots$. The set $A_{\alpha}^n$ contains all walks formed by nesting any $n$ elements of $A_{\alpha}$ together. Note that it follows from Definition \ref{defn:NestingSets} that if $A_\alpha = \emptyset$, then
\begin{align}
A_\alpha^n = \begin{cases}
\big\{(\alpha)\big\} & n = 0, \\
\emptyset & n \ge 1
\end{cases}
\end{align}
in analogy with a common convention for powers of the number 0.
\end{definition}

\begin{definition}[Kleene star of a set of walks \cite{Ebbinghaus1994}]\label{defn:KleeneStar}
     For a set $A_{\alpha}$ of closed walks off a vertex $\alpha$, the Kleene star of $A_{\alpha}$, denoted by $A_{\alpha}^*$, is the set of walks formed by nesting any number of elements of $A_{\alpha}$ together:
\begin{align}
A_{\alpha}^* = \bigcup_{n=0}^\infty A_{\alpha}^n.
\end{align}
It follows from the definition of nesting powers that if $A_{\alpha} = \emptyset$, then $A_{\alpha}^* = \big\{(\alpha)\big\}$.
\end{definition}
An example of the Kleene star operation on a set of cycles is given in Figure~\ref{fig:NestingAndKleeneStar}.

\subsection{Divisibility and prime walks}
     The nesting product of Definition \ref{defn:NestingProduct} induces a notion of divisibility, and with it a corresponding notion of prime walks.
\begin{definition}[Divisibility] Let $w$ and $w'$ be two walks. Then we say that $w'$ divides $w$, and write $w' \,|\, w$, if and only if one of the following conditions holds:
\begin{enumerate}
\item[(i)] there exist $n \ge 0$ walks $w_1, w_2, \ldots, w_n$ and an integer $0 \le i \le n$ such that $w = w_1 \nest\cdots \nest w_i \nest w' \nest w_{i+1} \nest \cdots \nest w_n$; or
\item[(ii)] there exists a walk $w''$ such that $w'\, |\, w''$ and $w'' \,| \,w$.
\end{enumerate}
A walk $w'$ that divides $w$ will be called a factor or divisor of $w$.
\end{definition}

\begin{definition}[Prime walks]\label{defn:PrimeWalks}
A walk $w$ is said to be prime with respect to the nesting operation if and only if for all nestable couples $(w',w'')$ such that $w \,| \,(w' \nest w'')$, then $w \,| \,w'$ or $w \,| \,w''$.
\end{definition}
As shown in \cite{Giscard2012}, the prime walks on a digraph $\G$ are precisely the simple paths and simple cycles of $\G$, of which there are a finite number if $\G$ is finite.

\subsection{The prime factorisation of a walk}\label{sec:PrimeFactorisation}
     The nesting operation permits the factorisation or decomposition of a walk $w$ on a digraph into two or more shorter walks, which, when nested together in the correct order, reproduce the original walk $w$. However, this factorisation is generally not unique: even very short walks typically admit multiple different factorisations. A simple example is the walk $12122$, which can be factorised either as $122 \nest 121$, or $1212 \nest 22$, or $12 \nest 22 \nest 121$.

     Of key interest among the factorisations of a given walk $w$ is the one that contains only prime walks. This factorisation, which we term the prime factorisation of $w$, corresponds to the decomposition of $w$ into a simple path plus a collection of simple cycles. Every walk admits a prime factorisation; further, this prime factorisation is unique \cite{Giscard2012}.

     An intuitive way of visualising the prime factorisation of $w$ is as follows. Consider traversing $w$ from start to finish, pausing at each vertex to check whether or not it has been visited previously. Upon arriving at a vertex $\mu$ that has previously been visited, remove from $w$ the sequence of edges traversed since the most recent encounter with $\mu$, and set this sequence aside. Proceed in this fashion until the end of $w$ is reached. At this point, the surviving vertex string contains no repeated vertices, and thus defines a simple path. Further, each of the extracted sections of $w$ contains no repeated internal vertices, and so defines a simple cycle off the relevant vertex $\mu$.\footnote{Viewed in this fashion, the process of prime factorisation corresponds closely to the `loop-erasure procedure' introduced by Lawler \cite{Lawler1980,Lawler1999} (note that his `loop' is a `closed walk' in our terminology). The underlying simple path of a walk $w$ is the `loop-erased random walk' $L(w)$ of that author. Obtaining the prime factorisation of $w$ via the loop-erasure procedure involves the trivial extra step of keeping track of each of the closed walks that are erased, and applying the loop-erasure procedure recursively to each closed walk until a collection of simple cycles is obtained.}

     The prime factorisation theorem for walks on a digraph can be summarised in the following Theorem, which we state without proof. For further details we refer to \cite{Giscard2012}.

\begin{theorem}[The prime factorisation theorem for walks on digraphs]\label{thm:PrimeFactorisationTheorem}
The set of all walks from $\alpha$ to $\omega$ on a digraph $\G$ can be written as
\begin{align*}
W_{\G;\alpha\omega} = P_{\G;\alpha\omega}
\nest C^{\,*}_{\G\backslash\alpha\cdots\mu_\ell;\omega}
\nest \cdots
\nest C^{\,*}_{\G\backslash\alpha;\mu_2}
\nest C^{\,*}_{\G;\alpha},
\end{align*}
     where $\ell$ is the length of the simple path $\alpha\mu_2\cdots\mu_\ell\omega \in P_{\G;\alpha\omega}$, and
\begin{align*}
C_{\G;\alpha} =
S_{\G;\alpha} \nest
       C^{\,*}_{\G\backslash\alpha\cdots\eta_{m-1};\eta_m}\nest \cdots \nest   C^{\,*}_{\G\backslash\alpha;\eta_2},
\end{align*}
     where $m$ is the length of the simple cycle $\alpha\eta_2\cdots\eta_m\alpha \in S_{\G;\alpha}$.
\end{theorem}

     Theorem \ref{thm:PrimeFactorisationTheorem} makes two statements. The first is that every walk can be uniquely decomposed into a (possibly trivial) simple path $p$ plus a (possibly empty) collection of cycles nested off the vertices of $p$. In other words, every walk $w$ from $\alpha$ to $\omega$ can be uniquely written for some $\ell \ge 0$ and $N_1,N_2,\ldots,N_{\ell+1} \ge 0$ as
\begin{align}
w = \alpha\mu_2\cdots\mu_\ell\omega
\nest \left[\Nest_{j=1}^{N_{\ell+1}}c_{\ell+1,j}\right]
\nest\cdots
\nest \left[\Nest_{j=1}^{N_2}c_{2,j}\right]
\nest \left[\Nest_{j=1}^{N_1}c_{1,j}\right],\label{eqn:WalkPrimeFactorisation}
\end{align}
where $\alpha\mu_2\cdots\mu_\ell\omega\in P_{\G;\alpha\omega}$ is a simple path (which we term the \key{base path} of $w$) and $c_{i,j} \in C_{\G\backslash\alpha\cdots\mu_{i-1};\mu_i}$ for $1\le i \le \ell + 1$ and $1 \le j \le N_i$ is a cycle off $\mu_i$. We refer to the cycles $c_{i,j}$ as the cycles nested off $\mu_i$ or the \key{child cycles} of $\mu_i$.

The second statement is that every cycle can be uniquely decomposed into a simple cycle $s$ plus a (possibly empty) collection of cycles nested off the internal vertices of $s$. That is, every cycle $c$ off a given vertex $\alpha$ can be written for some $m \ge 1$ and $M_2,\ldots, M_m \ge 0$ as
\begin{align}
c = \alpha\eta_2\cdots\eta_m\alpha
\nest \left[\Nest_{j=1}^{M_m}c_{m,j}\right]
\nest\cdots
\nest\left[\Nest_{j=1}^{M_2}c_{2,j}\right],\label{eqn:CyclePrimeFactorisation}
\end{align}
where $\alpha\eta_2\cdots\eta_m\alpha \in S_{\G;\alpha}$ is a simple cycle (which we term the \key{base cycle} of $c$) and $c_{i,j} \in C_{\G\backslash\alpha\cdots\eta_{i-1};\eta_i}$ for $1\le i \le m$ and $1 \le j \le M_i$ is a cycle off $\eta_i$. We again refer to the cycles $c_{i,j}$ for $1 \le j \le N_i$ as the child cycles of $\eta_i$.\\

\noindent An example of the prime factorisation is given in Figure \ref{fig:PrimeFactorisation}.
\begin{figure}\label{fig:PrimeFactorisation}
\includegraphics[scale=1.20]{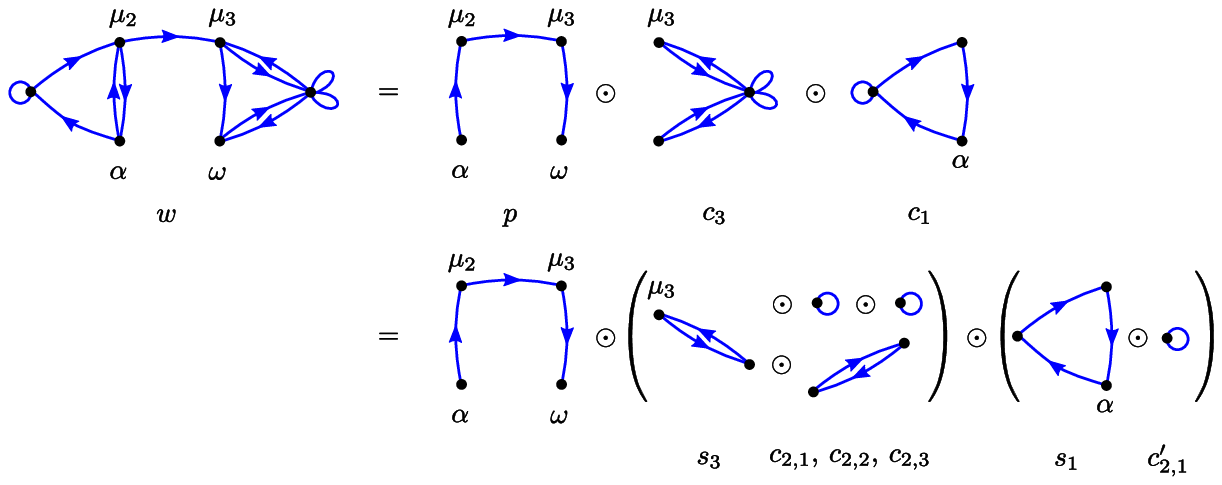}
  \caption{The process of factorising a walk into its prime factors. (Top) $w$ is factorised into a simple path $p$ plus a collection of cycles $c_i$ off the vertices of $p$ (cf.~Eq.~\eqref{eqn:WalkPrimeFactorisation}). The right-to-left ordering of the factors is a consequence of the nestable condition of Definition \ref{defn:NestableCouples}. (Bottom) Each cycle $c$ is then further factorised into a simple cycle $s$ plus a collection of cycles $c_{i,j}$ off the internal vertices of $s$ (cf.~Eq.~\eqref{eqn:CyclePrimeFactorisation}). }
\end{figure}

\subsection{The syntax tree of a walk}\label{sec:SyntaxTree}
     The factorisation of a walk $w$ into a collection of shorter walks via nesting is conveniently represented by a structure that we call a \key{syntax tree}. Given a particular factorisation of $w$, the corresponding syntax tree indicates both the factors and the order in which they must be nested together in order to reconstruct $w$. We begin by defining a generalised tree, which is an extension of the well-known tree data structure to include ordered sequences of child nodes instead of individual child nodes.

\begin{definition}[Generalised trees]
       A \key{generalised tree} \gentree{} consists of a root node \rootnode{T} plus a (possibly empty) ordered sequence of hedges, where a \key{hedge} is an ordered sequence of generalised trees:
\begin{align*}
       \gentree &= \big[\textsl{T.root}, [ \textsl{$H_{1}$}, \textsl{$H_{2}$}, \ldots, \textsl{$H_{n}$} ] \big], \\
       \hedge_i & = \big[\textsl{$T_{i,1}$},\textsl{$T_{i,2}$},\ldots, \textsl{$T_{i,k_i}$} \big].
\end{align*}
     Here $n \ge 0$ is the number of hedges proceeding from \rootnode{T}, and $k_i \ge 0$ is the number of generalised trees in the hedge $\hedge_i$. Adopting standard terminology, we say that each of the nodes \rootnode{$\gentree{}_{i,1}$},$\ldots$,\rootnode{$\gentree{}_{i,k_i}$} is a child node of $\rootnode{T}$, and has $\rootnode{T}$ as a parent. A node that has no children is a leaf node. Given a node $\textsl{n}$, nodes in the same hedge as $\textsl{n}$ are said to be siblings (specifically, right or left siblings, according to their position in the hedge relative to $\textsl{n}$).\\

     \noindent Every node in \gentree{} has a depth and a height: the depth of a node \textsl{n} is equal to the length of a path from \textsl{n} to \rootnode{T}, and the height of \textsl{n} is the length of the longest downward path from \textsl{n} to a leaf node. The height of the entire tree \gentree{} is defined to be the height of \textsl{T.root}.

     \noindent A subtree of $\gentree{}$ consists of a node $\textsl{n}$ and all its descendants (i.e.~all children, grandchildren, etc.) in $\gentree{}$. We denote the subtree rooted at node $\textsl{n}$ by $\textsl{t(n)}$: thus $t(\rootnode{\gentree{}}) = \gentree{}$, while for $\textsl{n} \neq \rootnode{\gentree{}}$, $t(\textsl{n})$ is a proper subtree of $\gentree{}$.

\end{definition}

\begin{definition}[Syntax trees]
     Given a digraph $\G$, a syntax tree over $\G$ is a generalised tree \gentree{} together with a collection of elements of $W_\G$, one associated with each node in \gentree{}. We denote the walk associated with node $\textsl{n}$ by $\textsl{n.contents}$, and define the contents of \gentree{} (as distinct from the contents of its root node) to be
\begin{subequations}
\begin{align*}
     \textsl{T.contents} = \textsl{T.root.contents} \nest \textsl{$H_{n}$.contents} \nest \cdots \nest \textsl{$H_{2}$.contents}
\nest \textsl{$H_{1}$.contents},
\end{align*}
where the contents of a hedge are given by
\begin{align*}
\textsl{$H_{i}$.contents} = \textsl{$T_{i,1}$.contents} \nest \textsl{$T_{i,2}$.contents}
\nest \cdots \nest \textsl{$T_{i,k_i}$.contents}.
\end{align*}
\end{subequations}
     The syntax tree \gentree{} is \key{valid} for a walk $w$ if and only if it satisfies the following conditions:
\begin{enumerate}
\item[(i)] $\textsl{T.contents} = w$,
       \item[(ii)] $h(H_i\textsl{.contents}) \neq h(H_j\textsl{.contents})$ for $ i \neq j$. In other words, no two hedges contain closed walks off the same vertex.
\end{enumerate}
     The purpose of the second condition is to disallow syntax trees where several closed walks off the same vertex are placed in different hedges, rather than into different generalised trees within the same hedge.
\end{definition}

\begin{figure}\label{fig:CanonicalSyntaxTrees}\centering
\includegraphics[scale=1.6]{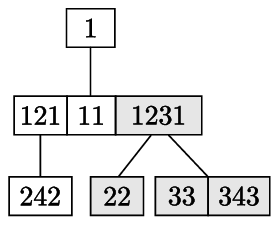}\hspace{2cm}
\includegraphics[scale=1.6]{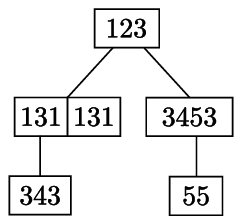}
\caption{Two examples of canonical syntax trees representing prime factorisations of walks. The root node contains a (possibly trivial) simple path, and every child node contains a simple cycle. (Left) the canonical syntax tree showing the factorisation of the closed walk $1242112233431 = 1 \nest (121 \nest 242) \nest 11 \nest (1231 \nest 33 \nest 343 \nest 22)$ into its prime factors. The subtree containing the cycle $12233431$ is shaded in grey. (Right) The canonical syntax tree of the walk $1343131234553 = 123 \nest (3453 \nest 55) \nest((131 \nest 343) \nest 131)$.}
\end{figure}

Each factorisation of a walk gives rise to a different syntax tree. Since many walks admit more than one factorisation, there are often multiple valid syntax trees for a given walk. Of particular interest among this family of syntax trees is the one that corresponds to the prime factorisation of $w$. We refer to this as the \key{canonical} syntax tree of $w$, or `the' syntax tree of $w$. If $\gentree{}$ is the canonical syntax tree for a walk $w$, then $\rootnode{\gentree{}}$ contains the base path of $w$ and every child node of $\gentree{}$ contains a simple cycle. Note that for every proper subtree $\textsl{t}$ of $\gentree{}$, $\textsl{t.contents}$ is a cycle. Figure \ref{fig:CanonicalSyntaxTrees} shows two examples of canonical syntax trees.

\begin{definition}[The canonical syntax tree of a walk]
     The \key{canonical syntax tree} of a walk $w$ on a digraph $\G$ is the unique valid syntax tree for $w$ that satisfies $\textsl{T.root.contents} \in P_\G$ and $\textsl{n.contents} \in S_\G$ for every child node $\textsl{n}$. The existence and uniqueness of this tree for any walk follow from the prime factorisation theorem discussed in Section \ref{sec:PrimeFactorisation}.
\end{definition} 
\section{A family of partitions of the walk set $W_{\G}$}\label{sec:S3:FamilyOfPartitions}
In this section we present the first main result of this article. We begin by introducing the concept of a \key{dressing signature}. A dressing signature $\dsig$ is a sequence of non-negative integers designed as a compact method of identifying a family of cycles by specifying the maximum length of their prime factors. Given a particular dressing signature and its associated family of cycles, a resummation scheme that exactly resums these cycles can be found. Thus, every dressing signature identifies a different resummation scheme; and every resummation scheme in the hierarchy depicted in Figure \ref{fig:ResummationSpectrum} corresponds to a different dressing signature. The dressing signature $[1,0]$ identifies a scheme that resums all loops, and $[2,0]$ to a scheme that resums both loops and backtracks (the trailing zero indicates that the backtracks have no cycles nested off their internal vertex). We term cycles that are exactly resummed within a given dressing scheme \key{resummable} cycles.

The main result of this section can then be stated as follows. For any dressing signature $\dsig$, the set of all walks on a digraph $\G$ can be partitioned into a collection of collectively exhaustive, mutually exclusive subsets. Each subset consists of a walk $i$ that does not contain any resummable cycles -- which we term a $\dsig$-irreducible walk -- plus all the walks that can be formed by nesting one or more resummable cycles into $i$.

The idea of the construction is straightforward. Given a dressing signature $\dsig$, we introduce two operators, which we term the walk reduction operator and the walk dressing operator for that dressing signature. We postpone giving the explicit form of these operators, defining them for now in terms of their actions. The walk reduction operator $\Reduc{\dsig} : W_\G \rightarrow W_\G$ is a projector that maps any walk $w$ to its $\dsig$-irreducible `core walk' by deleting all resummable cycles from it (in practice, this means mapping every resummable cycle $c$ to the corresponding trivial walk $(h(c))$, which is an identity element for the nesting operation). Conversely, the walk dressing operator $\Dress{\dsig}{\G}: \Reduc{\dsig}\big[W_\G\big] \rightarrow \powset{W_\G}$ maps a $\dsig$-irreducible walk $i$ to the set consisting of $i$ plus all walks that can be formed by nesting one or more resummable cycles off the vertices of $i$. Then the ensemble of walk sets obtained by applying $\Dress{\dsig}{\G}$ to each $\dsig$-irreducible walk on $\G$ forms a partition of $W_\G$. Figure \ref{fig:walk_set_partition} depicts the actions of $\Reduc{\dsig}$ and $\Dress{\dsig}{\G}$.

\begin{figure}\label{fig:walk_set_partition}
\centering
\includegraphics[scale=1.00]{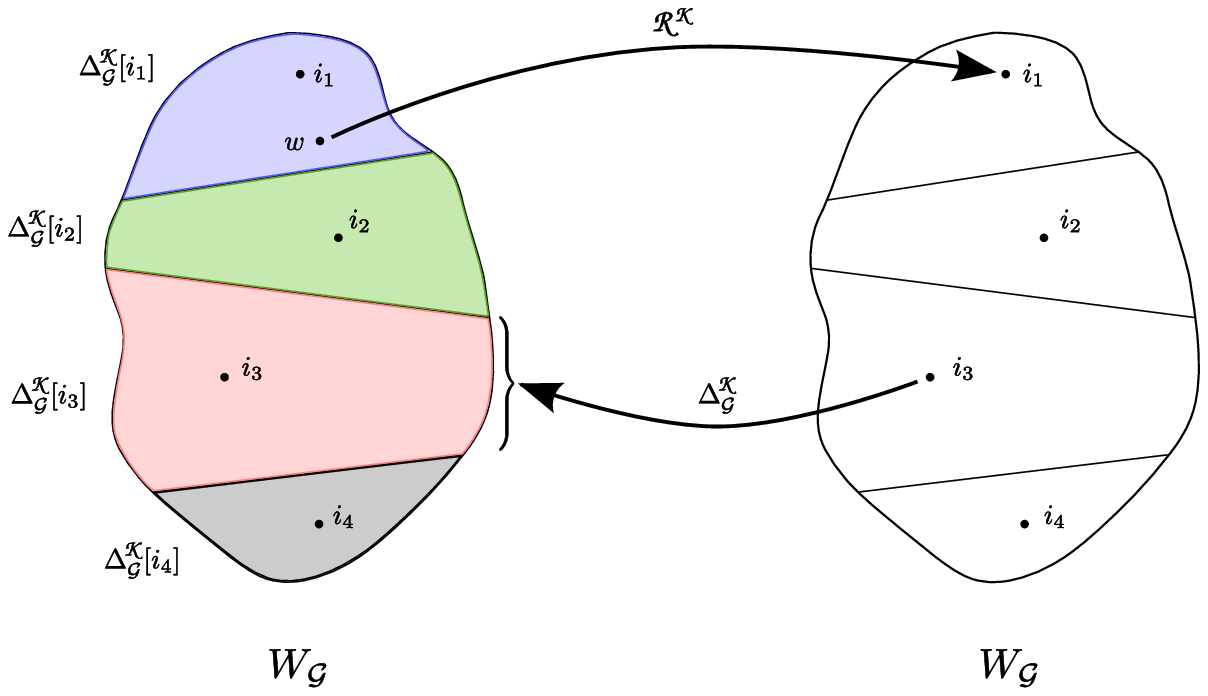}
  \caption{A schematic illustration of the actions of the walk reduction operator $\Reduc{\dsig}: W_\G \rightarrow W_\G$ and walk dressing operator $\Dress{\dsig}{\G}: \Reduc{\dsig}\big[W_\G\big]\rightarrow \powset{W_\G}$ for a particular dressing signature $\dsig$. The action of $\Reduc{\dsig}$ is to map any walk $w$ to its core $\dsig$-irreducible walk $i \in I^\dsig_\G$ by deleting all resummable cycles from $w$. Here the set of $\dsig$-irreducible walks is $I^{\dsig}_{\G} = \{i_1,i_2,i_3,i_4\}$. The walk dressing operator maps a $\dsig$-irreducible walk $i$ to the set of all walks that satisfy $\Reduc{\dsig}[w] = i$. The family of sets $\big\{\Dress{\dsig}{\G}[i] : i \in I^\dsig_\G \big \} = \big\{\Dress{\dsig}{\G}[i_1], \Dress{\dsig}{\G}[i_2], \Dress{\dsig}{\G}[i_3], \Dress{\dsig}{\G}[i_4]\big\}$ form a partition of $W_\G$. A different dressing signature $\dsig' \neq \dsig$ would induce a different set of irreducible walks and a different partition of $W_\G$.}
\end{figure}

As long as $\Reduc{\dsig}$ is idempotent and $\Dress{\dsig}{\G}$ acts as its inverse (in the sense of mapping a given $\dsig$-irreducible walk to the set of all walks that satisfy $\Reduc{\dsig}[w] = i$) this result can be proven independently of their explicit definitions. Therefore, in order to avoid getting sidetracked by the technical details of the definitions, in this section we simply state the properties required of $\Reduc{\dsig}$ and $\Dress{\dsig}{\G}$, assume that suitable definitions can be found, and show that it follows that the ensemble of walk sets described above partition $W_\G$. We justify this assumption in Section \ref{sec:S4:WalkReductionWalkDressing}, where we give explicit definitions for $\Reduc{\dsig}$ and $\Dress{\dsig}{\G}$ that satisfy the properties assumed here.

The remainder of the section is structured as follows. Section \S\ref{sec:DressingSignatures} is dedicated to defining and discussing dressing signatures. In \S\ref{sec:32:FamilyOfPartitions} we state the defining properties of the operators $\Reduc{\dsig}$ and $\Dress{\dsig}{\G}$, and show that the ensemble of walk sets described above form a partition of $W_\G$. In \S\ref{sec:EquivRelation} we discuss how $\Reduc{\dsig}$ can be interpreted as an equivalence relation on $W_\G$.

\subsection{Dressing signatures}\label{sec:DressingSignatures}
The syntax trees introduced in Section \ref{sec:SyntaxTree} form the basis for a method of classifying walks by the structure of their prime decompositions. When compared with a simple classification scheme based on comparing vertex sequences, such a syntax-tree-based taxonomy of walks has two main advantages. Firstly, it provides for a more precise description of walk structures than the elementary distinctions of closed vs.~open (walks) or simple vs.~compound (cycles). Secondly, it allows the structure of a walk to be discussed independently of the underlying graph. In particular, if the syntax tree of a given walk $w$ is modified so that only the length of the prime factor, rather than the prime factor itself, appears in each node, then the structure of $w$ can be precisely discussed without ever explicitly referring to the edges that it traverses. This abstraction is useful insofar as it allows a general definition of when two walks can be said to be equivalent -- even if they connect different vertices, traverse different edges, or exist on two different graphs.

The main goal of this article is to describe and generate irreducible walks for each of the resummation schemes depicted in Figure \ref{fig:ResummationSpectrum}. This goal requires a shorthand form of describing and identifying the families of cycles that are classed as `resummable' for each resummation scheme. To this end, we introduce a sequence $\dsig = [k_0,k_1,\ldots,k_D]$ of non-negative integers that we term a \key{dressing signature}. A dressing signature is a condensed way of prescribing the structure of a cycle $c$. The values of $k_i$ for $0\le i \le D$ set the maximum length of the prime factors (i.e.~the simple cycles) at depth $i+1$ in the syntax tree of $c$: thus the base cycle of $c$ has length less than or equal to $k_0$, its child cycles have length less than or equal to $k_1$, and so forth. In order to emphasize that the child cycles at depth $D-1$ have no children of their own, we fix the final entry of any dressing signature to be $k_D = 0$. Further, since only simple cycles of length 2 and greater can have child cycles, all elements of $\dsig$ apart from the last non-zero element must be at least 2. This leads to the following definition.

\begin{definition}[Dressing signatures]\label{defn:DressingSignatures}
     A dressing signature $\dsig = [k_0, k_1,\ldots,k_D]$ is a sequence of one or more integers such that $k_1,k_2,\ldots,k_{D-2} \ge 2$ , $k_{D-1} \ge 1$, and $k_D = 0$. The parameter $D \ge 0$ is the number of non-zero integers in $\dsig$ before the final zero, and will be referred to as the \key{depth} of $\dsig$.
\end{definition}


\begin{remark}
     Given a dressing signature $\dsig = [k_0, k_1,\ldots,0]$ with depth $D$, the sequences $[k_1,\ldots,k_{D-1},0]$, $\ldots$, $[k_{D-1},0], [0]$ formed by progressively deleting elements from the front of $\dsig$ are also dressing signatures.
\end{remark}

A cycle whose prime decomposition matches the form prescribed by a given dressing signature $\dsig$ is said to be $\dsig$-structured.

\begin{definition}[$\dsig$-structured cycles]\label{defn:KStructuredCycles}
     Let $\dsig = [k_0,k_1,\ldots,k_{D-1},0]$ be a dressing signature, and $\dsig' = [k_1,\ldots,k_{D-1},0]$ be the dressing signature obtained by deleting the first element of $\dsig$. Then a cycle $c$ is said to be $\dsig$-structured if and only if the base cycle of $c$ has length less than or equal to $k_0$, and all child cycles $c_{i,j}$ (if any exist) are $\dsig'$-structured. We denote the set of all $\dsig$-structured cycles off a given vertex $\alpha$ on a digraph $\G$ by $C_{\G;\alpha}^\dsig$, and the set of all $\dsig$-structured cycles on $\G$ by $C_\G^\dsig = \cup_\alpha \,C_{\G;\alpha}^\dsig$.
\end{definition}
     The set of $\dsig$-structured cycles is a subset of the set of all cycles: $C_\G^\dsig \subseteq C_\G$. We give an explicit expression for $C_\G^\dsig$ in Section \ref{sec:CycleReductionOperators} (see Definition \ref{defn:CkCycleSets}).

\begin{remark}
     Since every cycle has a length of at least 1, no cycle is $[0]$-structured. A loop is $\dsig$-structured for any value of $\dsig \neq [0]$. The cycles $121$ and $1231$ are both $[3,0]$-structured (indeed, they are $[k_0,0]$-structured for all $k_0 \ge 3$). The cycle $1221$ is $[2,1,0]$-structured, while $12234331 = 1231 \nest 343\nest 33\nest 22 $ is $[3,2,0]$-structured. Figure \ref{fig:k_structured_cycles} illustrates further examples of $[3,2,0]$-structured cycles that have the triangle 1231 as a base cycle.
\end{remark}

\begin{remark}
     The syntax tree \gentree{} of a $\dsig$-structured cycle $c$ has height at most $D$, where $D$ is the depth of $\dsig$. Since $c$ is a cycle, the root node $\rootnode{\gentree{}}$ contains the trivial simple path $(h(c))$. If $D \neq 0$, $\rootnode{\gentree{}}$ has a single child node, which contains a simple cycle of length no greater than $k_0$. If $D \ge 2$, this node may have one or more descendant nodes with depth $2 \le d \le D$, the contents of which satisfy $\len(\textsl{n.contents}) \le k_{d-1}$. Subject to this length restriction, there is no limit on the number of nodes with depth $2$ and deeper.
\end{remark}

\begin{figure}
  \centering
  \includegraphics[scale=1.2]{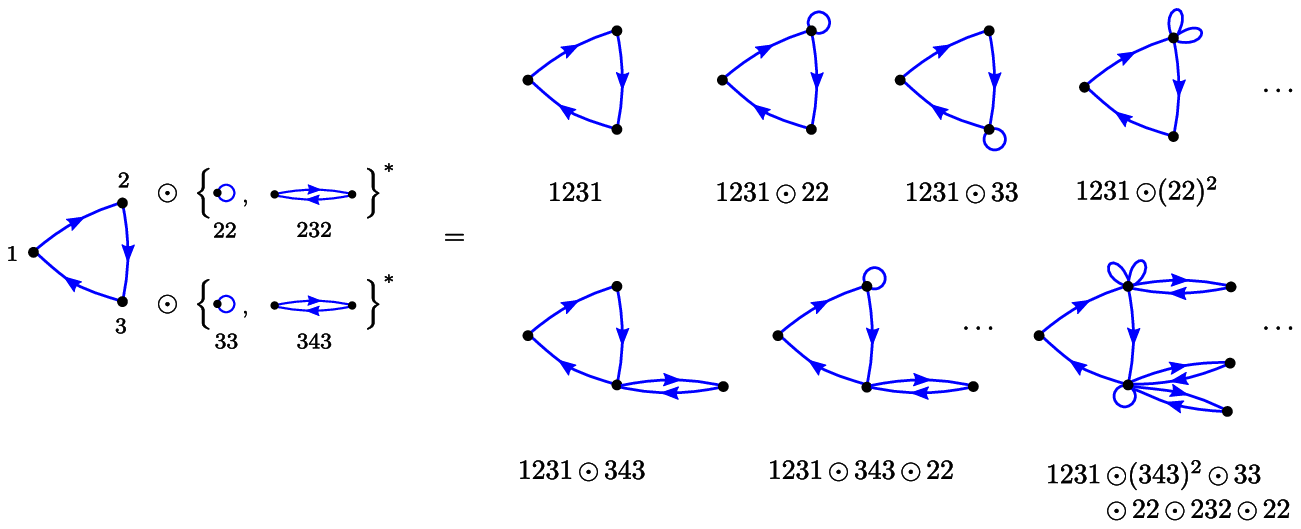}\hfill
  \caption{A set of $[3,2,0]$-structured cycles can be constructed from a simple cycle of length 3 (the triangle 1231, left) by nesting all possible configurations of $[2,0]$-structured cycles off each of its internal vertices (cf.~Definition \ref{defn:CkCycleSets} in \S\ref{sec:CycleReductionOperators}). The resultant set of cycles is infinite.}\label{fig:k_structured_cycles}
\end{figure}

\begin{remark}[The maximum dressing signature for a graph]
A natural order on dressing signatures is the quasi-lexicographic or shortlex order, in which shorter sequences precede longer ones and sequences of the same length are ordered lexicographically. The minimum dressing signature with respect to this order is $[0]$, and -- since dressing signatures can in principle be arbitrarily long -- there is no maximum.

However, for any finite graph $\G$ there is a maximum length to any simple cycle, and to the simple cycles that may be nested into that cycle, and so forth. There is therefore a maximum \key{useful} dressing signature for any finite graph, which we denote by $\dsig_{\,\text{max}} = \dsig_{\,\text{max}}(\G)$. This is defined to be the smallest dressing signature $\dsig$ such that every cycle on $\G$ is $\dsig$-structured. It follows that $C^{\dsig_{\,\text{max}}}_\G = C_\G$. For example, the maximum dressing signature on the complete graph with $n$ vertices, denoted $\G_n$, is
\begin{equation*}
  \dsig_{\,\text{max}}(\G_n) = \begin{cases}
    & [n,n-1,\ldots,1,0] \text{ if $\G_n$ contains at least one loop,}\\
    & [n,n-1,\ldots,2,0] \text{ if not.}
  \end{cases}
\end{equation*}

\end{remark}

\begin{remark}[Inclusive and exclusive dressing signatures]
     Since the elements of a dressing signature $\dsig$ define a maximum length for prime factors, the dressing signatures of Definition \ref{defn:DressingSignatures} would more precisely be called \key{inclusive} dressing signatures. In contrast we could imagine \key{exclusive} dressing signatures -- denoted $(k_0,k_1,\ldots,k_{D-1},0)$ -- with the corresponding definition that a cycle $c$ is $(k_0,k_1,\ldots,k_{D-1},0)$-structured if and only if the base cycle of $c$ has length exactly $k_0$, and all its immediate children (if any) are $(k_1,\ldots,k_{D-1},0)$-structured. However, exclusive dressing signatures turn out to be less useful than their inclusive counterparts (in particular because the set of $(k_0,k_1,\ldots,k_{D-1},0)$-structured cycles does not coincide with the set of all cycles as the dressing signature is increased) so in this article we will restrict our discussion to the inclusive case. Any mention of `dressing signature' without further qualification thus refers to an inclusive dressing signature, in the sense of Definition \ref{defn:DressingSignatures}. However, it is straightforward to extend the results we present to the case of exclusive dressing signatures.
\end{remark}

\subsection{A family of partitions of the walk set $W_\G$}\label{sec:32:FamilyOfPartitions}
In this section we state the fundamental properties required of $\Reduc{\dsig}$ and $\Dress{\dsig}{\G}$, and show that it follows from these properties that the family of sets produced by applying $\Dress{\dsig}{\G}$ to every $\dsig$-irreducible walk partition $W_\G$.

\begin{definition}\label{defn:ReductionDressingAxioms}
     Given a digraph $\G$ and a dressing signature $\dsig$, let $\Reduc{\dsig} : W_{\G} \rightarrow W_\G$ and $\Dress{\dsig}{\G} : \Reduc{\dsig}\big[W_\G\big]\rightarrow \powset{W_\G}$ be operators satisfying the following properties:
\begin{itemize}
       \item[P1.] $\Reduc{\dsig}$ is idempotent; i.e.~$\Reduc{\dsig}\big[\Reduc{\dsig}[w]\big] = \Reduc{\dsig}[w]$ for any walk $w \in W_\G$,
       \item[P2.] a walk $w$ is an element of $\Dress{\dsig}{\G}[i]$ if and only if $\Reduc{\dsig}[w] = i$.
\end{itemize}
\end{definition}

\noindent Since $\Reduc{\dsig}$ is idempotent, there exist walks that remain unchanged under its action. We term these walks $\dsig$-irreducible.

\begin{definition}[$\dsig$-irreducible walks]\label{defn:KIrreducibleWalks}
     Let $\G$ be a digraph and $w$ be a walk on $\G$. Then, for a given dressing signature $\dsig$, $w$ is said to be $\dsig$-irreducible if and only if $\Reduc{\dsig}\big[w\big] = w$. We denote the set of all $\dsig$-irreducible walks on $\G$ by $I^{\dsig}_\G$, and the set of all $\dsig$-irreducible walks from $\alpha$ to $\omega$ on $\G$ by $I^{\dsig}_{\G;\alpha\omega}$. That is,
\begin{align*}
I^{\dsig}_\G = \big\{w : w \in W_{\G} \text{ and } \Reduc{\dsig}\big[w\big] = w \big\},
\end{align*}
with an analogous expression for $I^{\dsig}_{\G;\alpha\omega}$.
\end{definition}

We give an explicit expression for the set of all $\dsig$-irreducible walks in Section \ref{subsec:StructureOfIK}.

\begin{prop}
       The image of $\Reduc{\dsig}$ is precisely the set of $\dsig$-irreducible walks.
\end{prop}
\begin{proof}
     The proof follows straightforwardly from the idempotence of $\Reduc{\dsig}$ and the definition of $I^{\dsig}_\G$. The image of $\Reduc{\dsig}$ is $\Reduc{\dsig}\big[W_\G\big] = \big\{\Reduc{\dsig}\big[w] : w \in W_\G\big\}$. Because $\Reduc{\dsig}$ is idempotent, every element of this set satisfies $\Reduc{\dsig}[u] = u$ and so is $\dsig$-irreducible. Conversely, every $\dsig$-irreducible walk $i \in I^\dsig_\G$ satisfies $i = \Reduc{\dsig}[w]$ for at least one walk $w \in W_\G$ (for example, $w = i$) and so appears in $\Reduc{\dsig}\big[W_\G\big]$.
\end{proof}
\noindent The walk reduction operator can thus be interpreted as a projector from $W_\G$ to $I^{\dsig}_\G \subseteq W_\G$, the subset of $W_\G$ consisting of $\dsig$-irreducible walks, and the walk dressing operator as a map from $I^{\dsig}_\G$ to the set of all subsets of walks on $\G$.

\begin{theorem}\label{thm:FamilyOfPartitions}
     The family of sets $\big\{\Dress{\dsig}{\G}[i] : i \in I^{\dsig}_\G\big\}$ partitions $W_\G$.
\end{theorem}
\begin{proof}
     This claim consists of three statements: (i) that no element of this family is empty; (ii) that the elements cover $W_\G$; (iii) that the elements are pairwise disjoint. We prove each of these statements in turn using the properties given in Definition \ref{defn:ReductionDressingAxioms}.
\begin{enumerate}
       \item[(i)] Consider the set $\Dress{\dsig}{\G}[i]$ for a given walk $i \in I^{\dsig}_\G$. By definition, $i$ satisfies $\Reduc{\dsig}[i] = i$. Then by P2, $i$ is an element of $\Dress{\dsig}{\G}[i]$, so that $\Dress{\dsig}{\G}[i]$ is not empty. Since this holds for every $i \in I^{\dsig}_\G$, no element of the family $\big\{\Dress{\dsig}{\G}[i] : i \in I^{\dsig}_\G\big\}$ is empty.\\

\item[(ii)] We wish to show that
\begin{align}
         \bigcup_{i\,\in I^{\dsig}_\G} \Dress{\dsig}{\G}[i] = W_\G.\label{eqn:PartitionPropII}
\end{align}
       Let us assume the opposite. Then there exists some walk $w \in W_\G$ that does not appear in any of the sets on the left-hand side. Let $j = \Reduc{\dsig}[w]$. Then $j$ is an element of $I^{\dsig}_\G$, and so $\Dress{\dsig}{\G}[j]$ appears as one of the terms on the left-hand side of Eq.~\eqref{eqn:PartitionPropII}. Further, by P2, $w$ is in $\Dress{\dsig}{\G}[j]$. Thus $w$ does appear in one of the sets on the left-hand side: a contradiction. Thus no such walk $w$ exists.\\

       \item[(iii)] We wish to show that $\Dress{\dsig}{\G}[i] \cap \Dress{\dsig}{\G}[j] = \emptyset $ for $i,j \in I^{\dsig}_\G$ and $i \neq j$. We prove this by contradiction. Assume the opposite: that there is a pair of sets $\Dress{\dsig}{\G}[i]$ and $\Dress{\dsig}{\G}[j]$ with $i \neq j$ that are not disjoint. Then there exists a walk $w \in W_\G$ that satisfies $w \in \Dress{\dsig}{\G}[i]$ and $w \in \Dress{\dsig}{\G}[j]$. Then by P2, we have that $\Reduc{\dsig}[w] = i$ and $\Reduc{\dsig}[w] = j$: a contradiction. Thus no such walk $w$ exists.
\end{enumerate}
\end{proof}

We have thus shown that if explicit definitions can be found for $\Reduc{\dsig}$ and $\Dress{\dsig}{\G}$ that satisfy properties P1 and P2 of Definition \ref{defn:ReductionDressingAxioms}, then a family of partitions of the walk set $W_\G$ follows straightforwardly. Presenting these explicit definitions and proving that they satisfy the required properties is the main topic of Section \ref{sec:S4:WalkReductionWalkDressing}.

\subsection{The walk reduction operator as an equivalence relation}\label{sec:EquivRelation}
It is common knowledge that every partition of a set is equivalent to an equivalence relation on that set, with equivalence classes equal to the elements of the partition. In this section we define the equivalence relations on $W_\G$ that correspond to the partitions given in Theorem \ref{thm:FamilyOfPartitions}.

\begin{definition}[The equivalence relations $\Kequiv$.]
Given a dressing signature $\dsig$, let $\Kequiv$ be the relation on $W_{\G}$ defined by
\begin{align*}
w_1 \Kequiv w_2 \text{ if and only if } \Reduc{\dsig}[w_1] = \Reduc{\dsig}[w_2].
\end{align*}
     This relation is reflexive, symmetric, and transitive, and therefore defines an equivalence relation on $W_\G$. We say that two walks $w_1, w_2 \in W_{\G}$ such that $w_1 \Kequiv w_2$ are $\dsig$-equivalent to each other.
\end{definition}

Two walks $w_1$ and $w_2$ are $\dsig$-equivalent if and only if they can be transformed into one another by adding or removing some number (possibly zero) of resummable cycles.

\begin{definition}[The $\dsig$-equivalence class of a walk]
Given a walk $w \in W_{\G}$, its $\dsig$-equivalence class $[w]_\dsig$ is the set of all walks that are $\dsig$-equivalent to $w$:
\begin{align*}
    [w]_\dsig = \big\{u \in W_{\G} : u \Kequiv w\big\}.
\end{align*}
\end{definition}

\begin{prop}
The $\dsig$-equivalence class of any walk $w \in W_\G$ is the set $\Dress{\dsig}{\G}\big[\Reduc{\dsig}[w]\big]$.
\end{prop}
\begin{proof}
The proof relies on property P2 of Definition \ref{defn:ReductionDressingAxioms}. Let $u$ be an element of  $[w]_\dsig$. Then by definition we have $\Reduc{\dsig}[u] = \Reduc{\dsig}[w]$, and so $u$ is in $\Dress{\dsig}{\G}\big[\Reduc{\dsig}[w]\big]$. Conversely, let $v$ be an element of $\Dress{\dsig}{\G}\big[\Reduc{\dsig}[w]\big]$. Then it follows that $\Reduc{\dsig}[v] = \Reduc{\dsig}[w]$, so that $v$ is in $[w]_\dsig$.
\end{proof}
Given an equivalence class $[w]_\dsig$, it is convenient to define a unique representative element (the `canonical representative') by which it can be referred to. Since every equivalence class $[w]_\dsig$ contains precisely one $\dsig$-irreducible walk $\Reduc{\dsig}[w]$, it is natural to adopt $\Reduc{\dsig}[w]$ as the canonical representative of $[w]_\dsig$.

\section{The walk reduction and walk dressing operators}\label{sec:S4:WalkReductionWalkDressing}
In this section we give explicit definitions for the walk reduction operator $\Reduc{\dsig}$ and walk dressing operator $\Dress{\dsig}{\G}$, and show that these definitions satisfy the desired properties stated in Definition \ref{defn:ReductionDressingAxioms}. We further obtain an explicit recursive expression for $I^\dsig_{\G}$, the set of $\dsig$-irreducible walks on $\G$. As shown in Theorem \ref{thm:FamilyOfPartitions}, knowledge of $I^\dsig_{\G}$ and $\Dress{\dsig}{\G}$ allows a partition of the set of all walks on an arbitrary digraph $\G$ to be constructed.

Since the actions of $\Reduc{\dsig}$ and $\Dress{\dsig}{\G}$ are defined in terms of resummable cycles, we first require a more precise definition of which cycles are resummable within the resummation scheme indexed by a given dressing signature $\dsig$. We take up this subject in detail in Section \ref{sec:41:ResummableCycles}. In Section \ref{sec:42:RandDeltaMotivation} we show that both $\Reduc{\dsig}$ and $\Dress{\dsig}{\G}$ admit a natural recursive definition in terms of operators that act on the individual cycles that make up $w$. Since their action on cycles parallels that of $\Reduc{\dsig}$ and $\Dress{\dsig}{\G}$ on walks, we term these operators the cycle reduction operators (denoted $\reduc{\dsig;l}$) and cycle dressing operators (denoted $\dress{\dsig;l}{\G}$) respectively. We define $\reduc{\dsig;l}$ in Section \ref{sec:CycleReductionOperators}, paving the way for the explicit definition and proof of idempotence of $\Reduc{\dsig}$ in Section \ref{sec:WalkReductionOperators}. In Section \ref{subsec:StructureOfQKL} we define a set of cycles that are invariant under the action of the cycle reduction operators. We term these cycles the $(\dsig,l)$-irreducible cycles. Their definition paves the way for the explicit definition of the set of $\dsig$-irreducible walks in Section \ref{subsec:StructureOfIK}. In Section \ref{subsec:CycleDressingOperators} we define the cycle dressing operators and in Section \ref{sec:WalkDressingOperator} the walk dressing operator. In Section \ref{subsec:rdeltaInverses} we show that the cycle reduction and cycle dressing operators are inverses of one another, in the sense that $\dress{\dsig;l}{\G}[q]$ produces the pre-image under $\reduc{\dsig;l}$ of a given $(\dsig,l)$-irreducible cycle $q$. In Section \ref{subsec:RDeltaInverses} we exploit this result to prove that $\Reduc{\dsig}$ and $\Dress{\dsig}{\G}$ are inverses of one another, thus fulfilling the aim of this section: to show that the explicit definitions for $\Reduc{\dsig}$ and $\Dress{\dsig}{\G}$ that we provide satisfy properties P1 and P2 of Definition \ref{defn:ReductionDressingAxioms}.

\subsection{Resummable cycles}\label{sec:41:ResummableCycles}
The desired actions of the walk reduction and walk dressing operators for a given dressing signature are simple to describe: $\Reduc{\dsig}$ maps a given walk $w$ to its $\dsig$-irreducible core walk by removing all resummable cycles from it, while $\Dress{\dsig}{\G}$ maps a $\dsig$-irreducible walk $i$ to the set of walks formed by dressing the vertices of $i$ by all possible configurations of resummable cycles. In order to convert this statement of intent into explicit definitions for  $\Reduc{\dsig}$ and $\Dress{\dsig}{\G}$, we require a precise definition of when a cycle can be said to be `resummable'.

The general question we want to address can be stated as follows. Let $w$ be a walk with syntax tree $\gentree{}$, $c$ be a cycle contained in a proper subtree $\textsl{t}$ of $\gentree{}$, and $\dsig$ be a dressing signature. Then we wish to decide between two possibilities: either $c$ is part of the $\dsig$-irreducible core walk $\Reduc{\dsig}[w]$ (and so is not resummable), or $c$ is produced by dressing $\Reduc{\dsig}[w]$ as prescribed by $\Dress{\dsig}{\G}$ (and so is resummable). In the latter case, it follows that deleting $c$ and all other resummable cycles from $w$ produces a walk that contains no resummable cycles.

The na\"ive solution to this question would be to assume that $c$ is resummable if and only if it is $\dsig$-structured. Indeed, for the case of $\dsig = [1,0]$ this is correct: deleting all the loops from a walk produces a walk that does not contain any loops, and $\Dress{[1,0]}{\G}$ can be defined so as to add all possible configurations of loops off the vertices of a loopless walk $i$. Identifying `resummable' with `[1,0]-structured' is thus perfectly consistent. However, this approach fails for all other dressing signatures: due to the possibility of nesting cycles inside other cycles for dressing signatures with $k_0 \ge 2$, it is not generally true that deleting all $\dsig$-structured cycles from a walk produces a walk that contains no $\dsig$-structured cycles. Alternatively, it might be thought that recursively deleting $\dsig$-structured cycles until the underlying walk no longer changes (i.e.~performing multiple passes over the walk, deleting $\dsig$-structured cycles each time) might be a valid approach. This is also not true, since it leads to the gradual deletion of arbitrarily deeply-nested cycles that cannot be reconstructed by a dressing scheme limited to adding cycles with a structure dictated by $\dsig$. (Consider for example $\dsig = [2,0]$ and the cycle $121 \nest 232 \nest \cdots\nest 898$). To illustrate these complexities it is useful to consider two simple examples.

\begin{example}\label{ex:ResummableCycles}
Consider the task of identifying which cycles in the walk $w = 12321 = 1 \nest 121 \nest 232$, if any, are resummable with respect to the dressing signature $\dsig =[2,0]$. There are only two possible candidates: $12321$ and $232$. We note firstly that 12321 is not $[2,0]$-structured, and so cannot possibly be the product of a dressing scheme that adds $[2,0]$-structured cycles. On the other hand, 232 is $[2,0]$-structured, and so might be thought to be resummable. However, removing it from $w$ leaves the walk $1 \nest 121$, which is not $[2,0]$-irreducible on account of the resummable cycle 121. Thus the cycle 232 is not the result of dressing a $[2,0]$-irreducible walk by $[2,0]$-structured cycles, and we conclude that -- despite being $[2,0]$-structured -- it is not resummable.
\end{example}

\begin{example}\label{ex:ResummableCycles2}
Next consider the task of identifying which cycles in the walk $w = 1232421 = 1 \nest 121 \nest 232 \nest 242$ are resummable with respect to the dressing signature $\dsig = [2,0]$. The only two candidates are the backtracks 232 and 242. Since removing them both leaves the walk $1 \nest 121$, which is not $[2,0]$-irreducible, we conclude that they cannot both be resummable. However, removing only one of them does leave a $[2,0]$-irreducible walk: $1\nest 12421$ if 232 is removed, and $1 \nest 12321$ if 242 is removed. Therefore precisely one of 232 and 242 is resummable. Which one is deemed resummable (and thus produced by dressing) and which is deemed to be part of the `core' irreducible walk is purely a matter of convention, which is equivalent to the decision of whether the walk $1 \nest 121 \nest 232 \nest 242$ should be generated by (a) adding 232 to the first appearance of the vertex 2 in the $[2,0]$-irreducible walk 12421, or (b) adding 242 to the second appearance of 2 in the $[2,0]$-irreducible walk 12321. We adopt the latter convention. The choice influences the exact forms taken by the cycle dressing and walk dressing operators, which we provide in Sections \ref{defn:CycleDressingOperators} and \ref{sec:WalkDressingOperator}.
\end{example}

These examples illustrate that whether or not a given cycle $c$ in a walk $w$ is resummable within the resummation scheme identified by a given dressing signature $\dsig$ depends not only on the structure of $c$, but also on its environment: specifically, on the ancestors and left siblings\footnote{The dependence on the left siblings, rather than the right siblings, is a direct consequence of the choice of convention mentioned in Example \ref{ex:ResummableCycles2}.} of $c$ in the syntax tree of $w$. In order to take this into account, we assign a quantity that we term \key{local depth} to each node in a syntax tree.

\begin{definition}[Local depth]\label{defn:LocalDepth}
     Let $\gentree{}$ be a canonical syntax tree, and $\dsig = [k_0,\ldots,k_{D-1},0]$ a dressing signature of depth $D$. Then the local depth of a node $\textsl{n}$ of $\gentree{}$ with respect to $\dsig$, denoted $\locdepth{\dsig}(\textsl{n})$, is an integer between $0$ and $D$ inclusive, defined as follows. The root node $\rootnode{\gentree{}}$ and all its immediate children have a local depth of 0, while every other node $\textsl{n}$ has a local depth given by
\begin{align*}
\locdepth{\dsig}(\textsl{n}) =
\begin{cases}
\locdepth{\dsig}(\textsl{p})  + 1&
\begin{aligned}[c]
           &\text{if $\len(\textsl{p.contents}) \le k_{\locdepth{\dsig}(\textsl{p})}$ and}\\
           &\text{\textsl{t(n$_j$).contents} is $[k_{\locdepth{\dsig}(\textsl{p})+1},\ldots,0]$-structured for $1 \le j \le i-1$,}
\end{aligned} \\[3mm]
0 & \text{otherwise,}
\end{cases}
\end{align*}
     where $\textsl{p} \equiv \textsl{n.parent}$ is the parent node of $\textsl{n}$ in \gentree{}, $i$ is the position of node $\textsl{n}$ within its hedge (counting from left to right), $\textsl{n}_j$ is a left sibling node of $\textsl{n}$, and $\textsl{t}(\textsl{n}_j)$ is a proper subtree of \gentree{} rooted at $\textsl{n}_j$. In a minor abuse of terminology, we will frequently refer to the local depth of a simple cycle $s$, which is to be understood as being shorthand for the local depth of the node containing $s$.
\end{definition}

\begin{remark}
     Note that the local depth of a simple cycle $s$ does not depend on the length of $s$ itself, but only on the length of the parent and the structure of the cycles contained in the subtrees rooted at each of the left siblings of $s$ (if any) in the syntax tree.
\end{remark}

\begin{figure}
  \centering
  \includegraphics[scale=1.35]{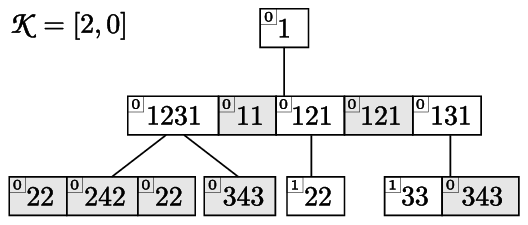}\hfill
  \includegraphics[scale=1.35]{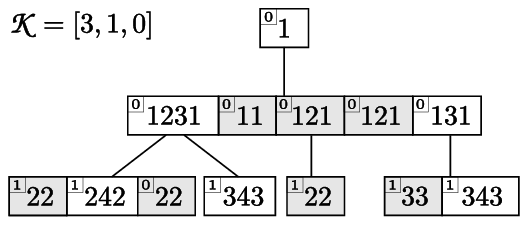}
  \caption{The canonical syntax tree for the walk $w = 1 \nest (1231 \nest 343 \nest 22 \nest 242 \nest 22)$$\nest 11 \nest (121 \nest 22)$$\nest 121 \nest (131 \nest 33 \nest 343)$, showing the local depth of each node with respect to the dressing signatures $\dsig = [2,0]$ (left) and $\dsig =[3,1,0]$ (right). The local depth is evaluated according to Def.~\ref{defn:LocalDepth}. Nodes that contain resummable cycles (see Def.~\ref{defn:ResummableCycle}) are shaded grey. The $\dsig$-irreducible walk thus consists of the white nodes in each case: $\Reduc{[2,0]}[w] = 1 \nest 1231 \nest (121 \nest 22) \nest (131 \nest 33)$ and $\Reduc{[3,1,0]}[w] = 1 \nest (1231 \nest 343 \nest 242) \nest (131 \nest 343)$.}
\end{figure}

The local depth of a node \textsl{n} can be understood as follows. Let $\dsig=[k_0,k_1,0]$ be a dressing signature of depth 2. Consider a node \textsl{n} in the syntax tree for a walk $w$ whose parent node is definitely a part of the core $\dsig$-irreducible walk $\Reduc{\dsig}[w]$. Then checking to see whether the cycle \textsl{t(n).contents} could have been produced by adding $\dsig$-structured cycles to $\Reduc{\dsig}[w]$ simply amounts to checking whether \textsl{t(n).contents} is $[k_0,k_1,0]$-structured. Such a node is assigned a local depth of 0.

On the other hand, let \textsl{n} be a node whose parent \textsl{p = n.parent} has local depth 0. Then the parent of \textsl{p} (i.e.~the grandparent of \textsl{n}) is definitely a part of $\Reduc{\dsig}[w]$, but there are two possibilities for \textsl{p} itself. The first is that by inspecting \textsl{p} and the left siblings of \textsl{n}, it can be concluded that \textsl{p} is definitely a part of the core irreducible walk. This is the case if \textsl{t(p).contents} is not $[k_0,k_1,0]$-structured, which is the case if either (i) \textsl{p.contents} is longer than $k_0$, or (ii) at least one of its child cycles that have already been seen -- i.e.~the cycles in the subtrees rooted at the left siblings of \textsl{n} -- is not $[k_1,0]$-structured. Then the situation is identical to that in the preceding paragraph: \textsl{n} is assigned local depth 0, indicating that \textsl{t(n).contents} should be checked to see if it is $[k_0,k_1,0]$-structured.

The second possibility is that cannot be concluded from this inspection that \textsl{p} is definitely part of the core irreducible walk. This is the case if (i) \textsl{p.contents} is not longer than $k_0$, and (ii) all of the cycles in the subtrees rooted at the left siblings of \textsl{n} are $[k_1,0]$-structured. Then it is possible that \textsl{t(p).contents} is the result of dressing its parent node by $[k_0,k_1,0]$-structured cycles. In this case, checking whether \textsl{t(n).contents} could have been produced by adding $\dsig$-structured cycles to $\Reduc{\dsig}[w]$ requires checking not whether \textsl{t(n).contents} is $[k_0,k_1,0]$-structured, but whether it is $[k_1,0]$-structured. We indicate this by assigning \textsl{n} a local depth of 1.\footnote{Note that whether \textsl{p} is \key{actually} a part of $\Reduc{\dsig}[w]$ or not is irrelevant: it is only important whether or not it can be concluded from \textsl{p} and the left siblings of \textsl{n} that \textsl{p} is in $\Reduc{\dsig}[w]$. The distinction is important if, for example, all of cycles in the subtrees rooted at the left siblings of \textsl{n} are $[k_1,0]$-structured, but one of the cycles in a subtree rooted at a \key{right} sibling of \textsl{n} is not. Then \textsl{t(p).contents} is not $[k_0,k_1,0]$-structured and \textsl{p} is a part of $\Reduc{\dsig}[w]$. However, since we only consider \key{left} siblings in assigning the local depth of a node, \textsl{n} is assigned a local depth of 1 regardless, and \textsl{t(n).contents} may be resummable even though \textsl{t(p).contents} is not.}

Extending this line of reasoning to deeper dressing signatures produces Definition \ref{defn:LocalDepth}. The local depth of a node $\textsl{n}$ indicates which trailing subsequence of the dressing signature $\dsig$ the cycle $\textsl{t(n).contents}$ should be compared against to check whether it is resummable or not.

The case of depth-1 dressing signatures, i.e.~$\dsig = [k_0,0]$, is straightforward: a simple cycle $s$ has local depth 1 if and only if its parent is a simple cycle not longer than $k_0$ and $s$ is the first child in its hedge. In this case the parent cycle and $s$, taken together, form part of the underlying $\dsig$-irreducible walk. Any right siblings that $s$ may have are assigned local depth 0 and compared against the dressing signature $[k_0,0]$, since they may arise from dressing the parent of $s$ by $[k_0,0]$-structured cycles. This is the case considered in Example \ref{ex:ResummableCycles2}.

Given a syntax tree \gentree{} and a dressing signature $\dsig$, the local depth can be assigned to each node in a pre-order depth-first traversal of \gentree{}, where the children within a hedge are visited from left to right.

We then have the following concise definition of resummability.

\begin{definition}[Resummable cycles]\label{defn:ResummableCycle}
     Let $w$ be a walk with syntax tree $\textsl{T}$, $c$ be a cycle corresponding to a proper subtree \textsl{t} of \gentree{}, and $\dsig = [k_0,\ldots,k_{D-1},0]$ be a dressing signature of depth $D$. Then $c$ is resummable in $w$ with respect to $\dsig$ if, and only if, $c$ is $[k_\ell,\ldots,k_{D-1},0]$-structured, where $0 \le \ell\le D$ is the local depth of $\textsl{t.root}$ with respect to $\dsig$.
\end{definition}

It is straightforward to verify that this definition makes the correct predictions in the situations considered in Examples \ref{ex:ResummableCycles} and \ref{ex:ResummableCycles2}.

\subsection{Constructing $\Reduc{\dsig}$ and $\Dress{\dsig}{\G}$}\label{sec:42:RandDeltaMotivation}
We now turn to developing explicit definitions for the walk reduction operator and walk dressing operator that satisfy the desired properties stated in Definition \ref{defn:ReductionDressingAxioms}. Since every walk can be decomposed into a simple path plus a collection of cycles (Eq.~\eqref{eqn:WalkPrimeFactorisation}), it is natural to define $\Reduc{\dsig}$ and $\Dress{\dsig}{\G}$ in terms of a collection of operators that act on cycles. These operators perform analogous roles on cycles to those of $\Reduc{\dsig}$ and $\Dress{\dsig}{\G}$ on walks, and are termed the cycle reduction operators and cycle dressing operators respectively. In this section we discuss the motivation for introducing these operators and give an overview of their properties. 

Consider first the walk reduction operator $\Reduc{\dsig}$, which deletes all resummable cycles from a walk $w$ by mapping each one to the corresponding trivial walk. This can be achieved through the following two-step construction. Let $\dsig$ have depth $D$; then we introduce a family of cycle reduction operators, which we denote by $\reduc{\dsig;l}$ for $0\le l \le D$. Specifically, $\reduc{\dsig;l}$ is a map on $C_\G \cup T_\G$ (where $T_\G$ is the set of trivial walks on $\G$) defined such that applying $\reduc{\dsig;l}$ to a cycle $c$ rooted at local depth $l$ removes all resummable cycles (up to and including $c$ itself) from $c$. Since $c$ itself must be deleted if it is resummable, it follows that $\reduc{\dsig;l}$ maps $c$ to the trivial walk $(h(c))$ if and only if $c$ is $[k_l,\ldots,k_{D-1},0]$-structured. Alternatively, if $c$ is left unchanged under the action of $\reduc{\dsig;l}$, it is termed a $(\dsig,l)$-irreducible cycle.

We then define $\Reduc{\dsig}$ such that every cycle $c$ in $w$ is mapped to $\reduc{\dsig;\ell}[c]$ in $\Reduc{\dsig}[w]$, where $\ell$ is the local depth of the base cycle of $c$. Since a cycle rooted at local depth $\ell$ is resummable if and only if it is $[k_{\ell},\ldots,k_{D-1},0]$-structured, it follows that a cycle $c$ in $w$ is mapped to $(h(c))$ in $\Reduc{\dsig}[w]$ if and only if $c$ is resummable with respect to $\dsig$, as desired. A further consequence of this construction is that a walk is $\dsig$-irreducible if and only if every cycle $c$ it contains is $(\dsig,\ell)$-irreducible, where $\ell$ is the local depth of the root node of the subtree containing $c$. Consequently, finding the set of $(\dsig,l)$-irreducible cycles is a significant step towards finding the set of $\dsig$-irreducible walks.


Next consider the walk dressing operator $\Dress{\dsig}{\G}$, which maps a $\dsig$-irreducible walk $i$ to the set of walks formed by adding all possible configurations of resummable cycles on $\G$ to $i$. The desired behaviour for $\Dress{\dsig}{\G}$ can be achieved by a two-step construction similar to that for $\Reduc{\dsig}$. Let $\dsig$ have depth $D$; then we introduce a family of cycle dressing operators, which we denote by $\dress{\dsig;l}{\G}$ for $0 \le l \le D$. These are defined such that $\dress{\dsig;l}{\G}: \reduc{\dsig;l}\big[C_\G\big]\setminus T_\G \rightarrow \powset{C_\G}$ maps a $(\dsig,l)$-irreducible cycle $q$ rooted at local depth $l$ to the set of cycles formed by adding all possible configurations of resummable cycles to $q$. Consequently, every cycle $c$ in $\dress{\dsig;l}{\G}[q]$ satisfies $\reduc{\dsig;l}[c] = q$, since any resummable cycles added by $\dress{\dsig;l}{\G}$ are removed by $\reduc{\dsig;l}$.

Then we define $\Dress{\dsig}{\G}$ such that (i) all possible configurations of $\dsig$-structured cycles are added as immediate children of the base path of $i$, and (ii) every cycle $q$ in $i$ is mapped to $\dress{\dsig;\ell}{\G'}[q]$ in $\Dress{\dsig}{\G}[i]$, where $\ell$ is the local depth of the root node of $q$ and $\G'$ is the graph on which $q$ exists (specifically, $\G$ minus a subset of its vertices -- see e.g.~Eqs.~\eqref{eqn:WalkPrimeFactorisation} and \eqref{eqn:CyclePrimeFactorisation}).

This is sufficient to produce the desired behaviour, since then every walk $w \in \Dress{\dsig}{\G}[i]$ differs from $i$ by (i) a collection of $\dsig$-structured cycles having been added as immediate children of the base path of $i$, and (ii) every cycle $q$ having been replaced by a cycle $c \in \dress{\dsig;\ell}{\G'}[q]$. Thus when $\Reduc{\dsig}$ is applied to $w$, $\reduc{\dsig;0}$ is applied to each of the newly-added $\dsig$-structured cycles, mapping each back to the corresponding trivial walk, and $\reduc{\dsig;\ell}$ is applied to every other cycle $c$, mapping it back to $q$. Therefore applying $\Reduc{\dsig}$ to $w$ recovers the $\dsig$-irreducible walk $i$.\\

\noindent Having discussed the motivation for introducing the cycle reduction and cycle dressing operators, we now state the fundamental properties they must satisfy. We give explicit definitions that satisfy these properties in Sections \ref{sec:CycleReductionOperators} and \ref{subsec:CycleDressingOperators} respectively.

\begin{definition}\label{defn:reductiondressingAxioms}
     Given a digraph $\G$, a dressing signature $\dsig$ of depth $D$, and a parameter $0 \le l \le D$, let $\reduc{\dsig;l} : C_\G \cup T_\G \rightarrow C_\G \cup T_\G$
     and $\dress{\dsig;l}{\G} : \reduc{\dsig;l}\big[C_\G\big]\setminus T_\G \rightarrow 2^{\,C_\G}$ be maps such that
\begin{itemize}
       \item[P1.] $\reduc{\dsig;l}$ is idempotent; i.e.~$\reduc{\dsig;l}\big[\reduc{\dsig;l}[d]\big] =\reduc{\dsig;l}[d]$ for every $d \in C_\G \cup T_\G$,
       \item[P2.] $\reduc{\dsig;l}[d] \in T_\G$ if and only if (i) $d \in T_\G$ or (ii) $d$ is a $[k_l,\ldots,k_{D-1},0]$-structured cycle,
       \item[P3.] a cycle $c$ is an element of $\dress{\dsig;l}{\G}[q]$ if and only if $\reduc{\dsig;l}[c] = q$.
\end{itemize}
     Recall from \S\ref{subsec:WalksPathsCycles} that $T_\G = \big\{(\alpha) : \alpha \in \G \big\}$ is the set of trivial walks on $\G$.
\end{definition}
\begin{figure}\label{fig:cycle_set_partition}
\centering
\includegraphics[scale=1.20]{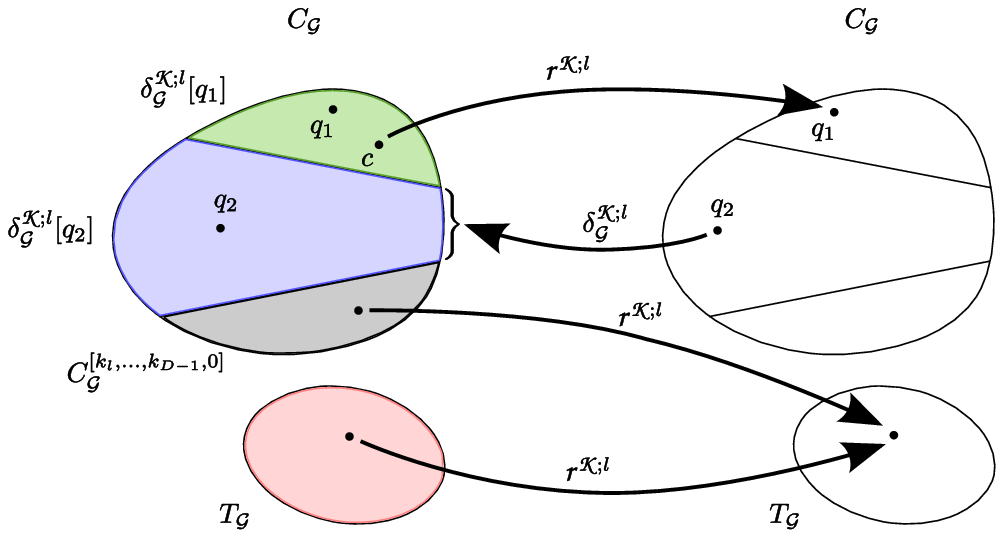}
  \caption{A schematic illustration of the action of the cycle reduction operator $\reduc{\dsig;l}$ and cycle dressing operator $\dress{\dsig;l}{\G}$ for a particular dressing signature $\dsig$ and parameter $l$. The cycle reduction operator maps a cycle $c \in C_\G$ either to the trivial walk $(h(c))$, if $c$ is $[k_l,\ldots,k_{D-1},0]$-structured, or to a $(\dsig,l)$-irreducible cycle $q \in Q^{\dsig;l}_{\G}$, otherwise. The cycle dressing operator maps a $(\dsig,l)$-irreducible cycle $q$ to the set of all cycles that satisfy $\reduc{\dsig;l}[c] = q$. The set of $[k_l,\ldots,k_{D-1},0]$-structured cycles together with the family of sets $\big\{\dress{\dsig;l}{\G}[q] : q \in Q^{\dsig;l}_{\G}\big\}$ form a decomposition of $C_{\G}$. A different value of $\dsig$ or $l$ induces a different decomposition.}
\end{figure}
The actions of $\reduc{\dsig;l}$ and $\dress{\dsig;l}{\G}$ are depicted in Figure \ref{fig:cycle_set_partition}.\\

Since $\reduc{\dsig;l}$ is idempotent, there is a subset of its domain $C_\G \cup T_\G$ that consists of elements that remain unchanged under the action of $\reduc{\dsig;l}$. Due to property P2, this subset contains the set of trivial walks. The remaining elements are cycles that satisfy $\reduc{\dsig;l}[c] = c$. We term such a cycle a $(\dsig,l)$-irreducible cycle.

\begin{definition}[$(\dsig,l)$-irreducible cycles]\label{defn:KLIrreducibleCycles}
Let $\G$ be a digraph and $\dsig$ be a dressing signature with depth $D$. Then a cycle $c \in C_\G$ is said to be $(\dsig,l)$-irreducible for $0 \le l \le D$ if and only if $\reduc{\dsig;l}[c] = c$. We denote the set of all $(\dsig,l)$-irreducible cycles off a particular vertex $\alpha$ on $\G$ by $Q^{\dsig;l}_{\G;\alpha}$; that is
\begin{align*}
Q^{\dsig;l}_{\G;\alpha}
= \left\{c : c \in C_{\G;\alpha} \text{ and } \reduc{\dsig;l}[c] = c\right\},
\end{align*}
and the set of all $(\dsig,l)$-irreducible cycles on $\G$ by $Q^{\dsig;l}_{\G} = \cup_\alpha\, Q^{\dsig;l}_{\G;\alpha}$.
\end{definition}
We give an explicit expression for $Q^{\dsig;l}_{\G;\alpha}$ in Section \ref{subsec:StructureOfQKL}. In the remainder of this section we prove some fundamental properties of $\reduc{\dsig;l}$ and $\dress{\dsig;l}{\G}$ that follow from the properties laid down in Definition \ref{defn:reductiondressingAxioms}.

\begin{prop}\label{prop:ImageOfrkl}
The image of $\reduc{\dsig;l}$ is the set $Q^{\dsig;l}_{\G} \cup T_\G$.
\end{prop}
\begin{proof}
By definition, the image of $\reduc{\dsig;l}$ is
\begin{align*}
     \reduc{\dsig;l}\big[C_\G \cup T_\G \big] &= \reduc{\dsig;l}\big[C_\G\big] \cup \reduc{\dsig;l}\big[T_\G\big] \\
&= \reduc{\dsig;l}\big[C_\G\big] \cup T_\G,
\end{align*}
where the second equality follows from the fact that $\reduc{\dsig;l}[(\alpha)] = (\alpha)$ for any trivial walk $(\alpha)$ (see property P2 above). The proposition thus corresponds to the claim that
\begin{align*}
\reduc{\dsig;l}\big[C_\G\big] \cup T_\G = Q^{\dsig;l}_\G \cup T_\G.
\end{align*}
     To prove this statement it suffices to show the inclusions (a) $\reduc{\dsig;l}\big[C_\G\big] \subseteq Q^{\dsig;l}_\G \cup T_\G$ and (b) $Q^{\dsig;l}_\G \subseteq \reduc{\dsig;l}\big[C_\G\big] \cup T_\G$.
\begin{itemize}
       \item[(a)] Let $a$ be an element of $\reduc{\dsig;l}\big[C_\G\big]$, so that $a = \reduc{\dsig;l}\big[c\big]$ for some $c \in C_\G$. Since the codomain of $\reduc{\dsig;l}$ is $C_\G \cup T_\G$, there are two possibilities:
\begin{itemize}
             \item[(i)] $a \in T_\G$. Then it is clearly the case that $a\in Q^{\dsig;l}_\G \cup T_\G$.
             \item[(ii)] $a \in C_\G$. Since $a = \reduc{\dsig;l}[c]$ for some $c \in C_\G$ and $\reduc{\dsig;l}$ is idempotent (property P1 above) it follows that $\reduc{\dsig;l}[a] = a$. Then by the definition of $Q^{\dsig;l}_\G$ we have $a \in Q^{\dsig;l}_\G$, so that $a \in Q^{\dsig;l}_\G \cup T_\G$.
\end{itemize}
             \item[(b)] Let $b$ be an element of $Q^{\dsig;l}_\G$. Then by the definition of $Q^{\dsig;l}_\G$ we have that $b = \reduc{\dsig;l}[b]$ and $b \in C_\G$. Thus $b$ is an element of $\{\reduc{\dsig;l}[c] : c \in C_\G\} = \reduc{\dsig;l}[C_\G]$.
\end{itemize}
\end{proof}
\noindent The cycle reduction operator $\reduc{\dsig;l}$ is therefore a projector from $C_\G \cup T_\G$ to $Q^{\dsig;l}_\G \cup T_\G$.

\begin{cor}\label{cor:KLIrredCyclesSetMinus}
     The set of $(\dsig,l)$-irreducible cycles is given by $Q^{\dsig,l}_\G = \reduc{\dsig;l}[C_\G] \setminus T_\G$.
\end{cor}
\begin{proof}
     By the definition of $Q^{\dsig,l}_\G$ (Defn.~\ref{defn:KLIrreducibleCycles}) we have that $Q^{\dsig,l}_\G \subseteq C_\G$, while from the convention that trivial walks are not cycles it follows that $C_\G\cap T_\G = \emptyset$. Hence $Q^{\dsig,l}_\G \cap T_\G = \emptyset$. From Proposition \ref{prop:ImageOfrkl} we further have that      $Q^{\dsig,l}_\G \cup T_\G = \reduc{\dsig;l}[C_\G] \cup T_\G$; thus $Q^{\dsig,l}_\G = \big(\reduc{\dsig;l}[C_\G] \cup T_\G\big) \setminus T_\G = \reduc{\dsig;l}[C_\G] \setminus T_\G$.
\end{proof}
     It follows from Corollary \ref{cor:KLIrredCyclesSetMinus} that the cycle dressing operator $\dress{\dsig;l}{\G}$ is a map from $(\dsig,l)$-irreducible cycles to sets of cycles: $\dress{\dsig;l}{\G} : Q^{\dsig,l}_\G \rightarrow \powset{C_\G}$.

\begin{lem}\label{lem:PartitionOfC}
For any dressing signature $\dsig$ and parameter $0 \le l \le D$, where $D$ is the depth of $\dsig$, the family of sets
\begin{align*}
    \big\{C^{[k_l,\ldots,k_{D_1},0]}_\G\big\} \cup \big\{\dress{\dsig;l}{\G}[q] : q \in Q^{\dsig;l}_\G\big\}
\end{align*}
       forms a decomposition of $C_\G$. Thus, every cycle $c \in C_\G$ is either $[k_l,\ldots,k_{D-1},0]$-structured, or can be produced by dressing a $(\dsig,l)$-irreducible cycle as prescribed by the dressing operator $\dress{\dsig;l}{\G}$.
\end{lem}
     Note that this family cannot be said to form a partition of $C_\G$ because of the possibility that $C^{[k_l,\ldots,k_{D-1},0]}_\G$ might be empty. This is the case when $\G$ contains no $[k_l,\ldots,k_{D-1},0]$-structured cycles: for example, a loopless graph contains no $[1,0]$-structured cycles.
\begin{proof}
     This claim consists of two statements: (i) that the elements of this family cover $C_\G$, and (ii) that the elements are pairwise disjoint. We address each of these statements in turn.
\begin{enumerate}
\item[(i)] We claim that
\begin{align}
C^{[k_l,\ldots,k_{D-1},0]}_\G \cup
       \bigcup_{q \in Q^{\dsig;l}_\G} \dress{\dsig;l}{\G}[q] = C_\G. \label{eqn:CycleDecompProof}
\end{align}
     Let us assume the opposite: there exists a cycle $c$ in $C_\G$ that does not appear in any of the sets on the left-hand side. Consider the result of applying $\reduc{\dsig;l}$ to $c$. Two possibilities exist:
\begin{enumerate}
       \item[(a)] $\reduc{\dsig;l}[c] \in T_\G$. Then by property P2 above, $c$ is a $[k_l,\ldots,k_{D-1},0]$-structured cycle, and so appears in $C^{[k_l,\ldots,k_{D-1},0]}_\G$ on the left-hand side of Eq.~\eqref{eqn:CycleDecompProof}.
       \item[(b)] $\reduc{\dsig;l}[c] \in Q^{\dsig;l}_\G$. Then $\reduc{\dsig;l}[c]$ appears in the union over $Q^{\dsig;l}_\G$ in Eq.~\eqref{eqn:CycleDecompProof}. Further, by property P3 above, $c$ is in $\dress{\dsig;l}{\G}[\reduc{\dsig;l}[c]]$, and so appears on the left-hand side of Eq.~\eqref{eqn:CycleDecompProof}.\\
\end{enumerate}

       \item[(ii)] To show that the members of the family are pairwise disjoint we require that (a) $C^{[k_l,\ldots,k_{D-1},0]}_\G \cap \dress{\dsig;l}{\G}[q] = \emptyset$ for $q \in Q^{\dsig;l}_\G$, and (b) $\dress{\dsig;l}{\G}[q_1] \cap \dress{\dsig;l}{\G}[q_2] = \emptyset$ for $q_1, q_2 \in Q^{\dsig;l}_\G$ and $q_1 \neq q_2$. We prove each of these in turn. The proof in each case follows           straightforwardly from the properties given in Definition \ref{defn:reductiondressingAxioms}.
\begin{enumerate}
         \item[(a)] Assume the opposite: that there exists a cycle $c$ such that $c \in C^{[k_l,\ldots,k_{D-1},0]}_\G$ and $c \in \dress{\dsig;l}{\G}[q]$ for some $q \in Q^{\dsig;l}_\G$. Then by P2 we have that $\reduc{\dsig;l}[c] \in T_\G$, on one hand, and $\reduc{\dsig;l}[c] = q$, on the other. However, as noted in the proof to Corollary \ref{cor:KLIrredCyclesSetMinus} above, $Q^{\dsig;l}_\G$ and $T_\G$ are disjoint: thus no such cycle $c$ can exist.
         \item[(b)] Assume the opposite: that there exists a cycle $c$ such that $c \in \dress{\dsig;l}{\G}[q_1]$ and $c \in \dress{\dsig;l}{\G}[q_2]$ for some $q_1,q_2 \in Q^{\dsig;l}_\G$ with $q_1 \neq q_2$. Then by P3 we have that $\reduc{\dsig;l}[c] = q_1$ and $\reduc{\dsig;l}[c] = q_2$: a contradiction. Thus no such cycle exists.
\end{enumerate}
\end{enumerate}
\end{proof}

\subsection{The cycle reduction operators}\label{sec:CycleReductionOperators}
     We now give an explicit definition for the cycle reduction operators $\reduc{\dsig;l}$ that satisfies the key properties given in Definition \ref{defn:reductiondressingAxioms}: namely, that $\reduc{\dsig;l}$ is idempotent, and that $\reduc{\dsig;l}[c]$ is equal to the trivial walk $(h(c))$ if and only if $c = (h(c))$ or $c$ is a $[k_l,\ldots,k_{D-1},0]$-structured cycle. We prove these claims later in this section.

\begin{definition}[Cycle reduction operators]\label{defn:CycleReductionOperator}
     Let $\G$ be a digraph, $\dsig$ be a dressing signature of depth $D$, and $0 \le l \le D$ be a parameter. Then the cycle reduction operator $\reduc{\dsig;l}$ is a map on the set of all cycles and trivial walks on $\G$:
\begin{align*}
\reduc{\dsig;l} : C_{\G} \cup T_\G \rightarrow C_{\G} \cup T_\G.
\end{align*}
     For any trivial walk $(\alpha)$, we define $\reduc{\dsig;l}\big[(\alpha)\big] = (\alpha)$ for every $0\le l \le D$. For any cycle $c$ we define $\reduc{\dsig;l}\big[c\big]$ as follows. Let $c$ have decomposition (cf.~Eq.~\ref{eqn:CyclePrimeFactorisation})
\begin{align*}
       c = \alpha\mu_2\cdots\mu_m\alpha \nest \left[\Nest_{j=1}^{N_m} c_{m,j}\right]\nest \cdots \nest \left[\Nest_{j=1}^{N_2} c_{2,j}\right],
\end{align*}
     where $\alpha\mu_2\cdots\mu_m\alpha$ is a simple cycle of length $m \ge 1$ off $\alpha$ on $\G$, $N_i \ge 0$ is the number of child cycles off $\mu_i$, and $c_{i,j}\in C_{\G\backslash\alpha\cdots\mu_{i-1};\mu_i}$ for $2 \le i \le m$ and $1 \le j \le N_i$. If and only if $m \le k_l$, then we introduce indices $s_2,\ldots,s_m$, where $s_i$ is the index of the first child off $\mu_i$ (counting from the left of the hedge) such that $\reduc{\dsig;l+1}[c_{i,s_i}] \neq (\mu_i)$. In other words, $c_{i,s_i}$ is the first child off $\mu_i$ that is not $[k_{l+1},\ldots,k_{D-1},0]$-structured. If no such child exists (i.e.~if there are no children off $\mu_i$, or if all children off $\mu_i$ satisfy $\reduc{\dsig;l+1}[c_{i,s_i}] = (\mu_i)$) then we set $s_i = N_i + 1$. Then the cycle reduction operator $\reduc{\dsig;\,l}$ maps $c$ to
\begin{subequations}
\begin{align}
\reduc{\dsig;\,l}\big[c\big] =
\begin{cases}
\left(\alpha\right) & \begin{aligned}
&\text{if $m \le k_l$ and every}\\[1mm]
&\text{$s_i = N_i + 1$}
\end{aligned}\\[1ex]
\!\begin{aligned}[b]
\alpha\mu_2\cdots\mu_m \alpha
\nest \left[\Nest_{j=1}^{N_m}\reduc{\dsig;\,l_{m,j}}\big[c_{m,j}\big]\right]
\nest \cdots
\nest \left[\Nest_{j=1}^{N_2}\reduc{\dsig;\,l_{2,j}}\big[c_{2,j}\big]\right]
\end{aligned}           & \text{else.}
\end{cases}\nonumber
\end{align}
where
\begin{align}
l_{i,j} = \begin{cases}
l + 1 & \text{if $m \le k_l$ and $j \le s_i$},\\
0 & \text{else,}
\end{cases}\nonumber
\end{align}
\end{subequations}
\end{definition}


It is worth commenting here on the role of the indices $s_i$. They are included to produce the desired behaviour for $\reduc{\dsig;l}$: namely, that $\reduc{\dsig;l}$ should map any $[k_l,\ldots,k_{D-1},0]$-structured cycle $c$ to $(h(c))$. Since $c$ can be $[k_l,\ldots,k_{D-1},0]$-structured only if every child cycle $c_{i,j}$ is $[k_{l+1},\ldots,k_{D-1},0]$-structured, $\reduc{\dsig;l}[c]$ must depend on $\reduc{\dsig;l+1}[c_{i,j}]$ for every child $c_{i,j}$. This dependence is achieved via the indices $s_i$.

The indices $s_i$ are evaluated if and only if the base cycle $s$ of $c$ is not longer than $k_l$. Then $s_i$ is the index of the first child cycle of $\mu_i$ that is not $[k_{l+1},\ldots,k_{D-1},0]$-structured, or $N_i + 1$ if no such child exists. Thus if and only if every $s_i$ is equal to $N_i+1$, all of the children off all of the vertices of $\alpha\mu_2\cdots\mu_m\alpha$ are $[k_{l+1},\ldots,k_{D-1},0]$-structured. Then $c$ is $[k_{l},\ldots,k_{D-1},0]$-structured, and is consequently mapped by $\reduc{\dsig;l}$ to $(h(c))$.

Note that the value of the parameter $l_{i,j}$, which specifies which cycle reduction operator is applied to the child $c_{i,j}$, is precisely the local depth of the root node of $c_{i,j}$. This ensures that when $\reduc{\dsig;l}$ is applied to $c$, every child cycle $c_{i,j}$ is mapped to $r^{\dsig;l'}[c_{i,j}]$, where $l'$ is the local depth of the root node of $c_{i,j}$.

The effect of applying $\reduc{\dsig;l}$ to a cycle $c$ rooted at local depth $l$ can be summarised as follows. There are two possible cases. Firstly, if the base cycle $s$ of $c$ is longer than $k_l$, every child of $s$ has local depth 0. In this case $\reduc{\dsig;0}$ is applied to every child cycle of $s$. Secondly, if $s$ is not longer than $k_l$, the first $\min(s_i,N_i)$ children of $s$ have local depth $l+1$, and any remaining children in the same hedge have local depth 0. Consequently, $\reduc{\dsig;l+1}$ is applied to the first $\min(s_i,N_i)$ children off $\mu_i$, and $\reduc{\dsig;0}$ to any remaining children off $\mu_i$.

In the remainder of this section we prove the two claims we have made about $\reduc{\dsig;l}$. Firstly, we show that $\reduc{\dsig;l}$ is idempotent for every $\dsig$ and $l$. Secondly, we prove that $\reduc{\dsig;l}$ maps a cycle $c$ to $(h(c))$ if and only if $c$ is $[k_l,\ldots,k_{D-1},0]$-structured.

\begin{prop}\label{prop:IdempotenceOfr}
The cycle reduction operators $\reduc{\dsig;l}$ are idempotent.
\end{prop}
\begin{proof}
     We first note that from the definition of $\reduc{\dsig;l}$ we have $\reduc{\dsig;l}\big[\reduc{\dsig;l}[(\alpha)]\big] = \reduc{\dsig;l}\big[(\alpha)\big] = (\alpha)$
     for any trivial walk $(\alpha)$ and any $0 \le l \le D$. It thus remains only to prove that $\reduc{\dsig;l}$ is idempotent for any cycle $c \in C_\G$. To this end, let $\mathbb{P}(n,l)$ for $n \ge 1$ and $0 \le l \le D$ be the statement
\begin{align}
      \reduc{\dsig;l}\big[\reduc{\dsig;l}[c]\big] = \reduc{\dsig;l}\big[c\big],\label{eq:IdempotenceOfrStatement}
\end{align}
     where $n$ is the length of $c$. We prove this statement for all $n$ and $l$ by induction on $n$. For a given $n$, the statements required to prove $\mathbb{P}(n,l)$ depend on $l$. Namely, $\mathbb{P}(n,l)$ for any $0 \le l \le D-1$ requires both $\mathbb{P}(m,0)$ and $\mathbb{P}(m,l+1)$ for all $m < n$, while $\mathbb{P}(n,D)$ only requires $\mathbb{P}(m,0)$ for all $m < n$. \\

\noindent\textit{Inductive step.}
We wish to prove the statements
\begin{align*}
\mathbb{P}(m,0) \text{ for all $m < n$ } &\implies \mathbb{P}(n,D),\\
         \mathbb{P}(m,l+1) \text{ and } \mathbb{P}(m,0) \text{ for all $m < n$ } &\implies \mathbb{P}(n,l) \text{ for $0 \le l \le D-1$}.
\end{align*}
     We proceed by explicitly evaluating $\reduc{\dsig;l}\big[c\big]$ and $\reduc{\dsig;l}\big[\reduc{\dsig;l}[c]\big]$ and showing that they are equal. Let $c$ have decomposition
\begin{align*}
c = \alpha\mu_2\cdots\mu_L\alpha
\nest \left[\Nest_{j=1}^{N_L}c_{L,j}\right]
\nest \cdots
\nest \left[\Nest_{j=1}^{N_2}c_{2,j}\right],
\end{align*}
where $1 \le L \le n$. Then two possibilities exist:
\begin{itemize}
       \item[1)] $L > k_l$. Note that is necessarily the case if $l=D$, since $k_D = 0$. Then we see from the definition of $\reduc{\dsig;l}$ that all of the parameters $l_{i,j}$ are zero, so that $\reduc{\dsig;0}$ is applied to every child cycle $c_{i,j}$. We find
\begin{align*}
\reduc{\dsig;l}\big[c\big] &= \alpha\mu_2\cdots\mu_L\alpha
\nest\left[\Nest_{j=1}^{N_L}\reduc{\dsig;0}\big[c_{L,j}\big]\right]
\nest \cdots
\nest \left[\Nest_{j=1}^{N_2}\reduc{\dsig;0}\big[c_{2,j}\big]\right]\\
&= \alpha\mu_2\cdots\mu_L\alpha
\nest\left[\Nest_{k=1}^{M_L}\reduc{\dsig;0}\big[c_{L,t_{L,k}}\big]\right]
\nest \cdots
\nest \left[\Nest_{k=1}^{M_2}\reduc{\dsig;0}\big[c_{2,t_{2,k}}\big]\right]
\end{align*}
       where $0 \le M_i \le N_i$ is the number of cycles among $c_{i,1},\ldots,c_{i,N_i}$ that are not reduced to $(\mu_i)$ by $\reduc{\dsig;0}$, and $t_{i,k}$ are their indices, which we assume without loss of generality satisfy $1 \le t_{i,1} < \cdots < t_{i,M_i} \le N_i$. Now
\begin{align*}
\reduc{\dsig;l}\big[\reduc{\dsig;l}[c]\big] &= \alpha\mu_2\cdots\mu_L\alpha
\nest\left[\Nest_{k=1}^{M_L}\reduc{\dsig;0}\big[\reduc{\dsig;0}[c_{L,t_{L,k}}]\big]\right]
\nest \cdots
             \nest \left[\Nest_{k=1}^{M_2}\reduc{\dsig;0}\big[\reduc{\dsig;0}[c_{2,t_{2,k}}]\big]\right]\\
&= \alpha\mu_2\cdots\mu_L\alpha
\nest\left[\Nest_{k=1}^{M_L}\reduc{\dsig;0}\big[c_{L,t_{L,k}}\big]\right]
\nest \cdots
\nest \left[\Nest_{k=1}^{M_2}\reduc{\dsig;0}\big[c_{2,t_{2,k}}\big]\right]\\
&= \reduc{\dsig;l}\big[c\big],
\end{align*}
     where the second equality follows from the induction hypotheses $\mathbb{P}(m_{i,k},0)$ for $2\le i \le L$ and $1\le k \le M_i$, where $m_{i,k} < n$ is the length of the cycle $c_{i,t_{i,k}}$.\\

     \item[2)] $L \le k_l$. Then the parameters $l_{i,j}$ depend on the indices $s_2,\ldots, s_L$. Specifically, $\reduc{\dsig;l+1}$ is applied to the first $\min(s_i,N_i)$ child cycles (counting from the left), and $\reduc{\dsig;0}$ is applied to any remaining child cycles. Two possibilities exist:
\begin{itemize}
           \item[a)] $s_i = N_i + 1$ for every $2 \le i \le L$; i.e.~every $c_{i,j}$ satisfies $\reduc{\dsig;l+1}[c_{i,j}] = (\mu_i)$. Then $\reduc{\dsig;l}\big[c\big] = (\alpha)$, and $\reduc{\dsig;l}\big[\reduc{\dsig;l}[c]\big] = \reduc{\dsig;l}\big[(\alpha)\big] = (\alpha)$.

           \item[b)] At least one of the indices $s_2,\ldots,s_L$ is not equal to $N_i + 1$. Let $\mathcal{I} = \{i_1,\ldots,i_p\}$ be the index set of those indices, so that $s_i \neq N_i+1$ if and only if $i \in \mathcal{I}$. Then
\begin{align*}
\reduc{\dsig;l}\big[c\big]
&=
\begin{aligned}[t]
\alpha\mu_2\cdots\mu_L\alpha
\nest \left[\Nest_{j=1}^{N_L}\reduc{\dsig;l_{L,j}}\big[c_{L,j}\big]\right]
\nest \cdots
\nest \left[\Nest_{j=1}^{N_2}\reduc{\dsig;l_{2,j}}\big[c_{2,j}\big]\right]\\
\end{aligned}\\
&=
\begin{aligned}[t]
\alpha\mu_2\cdots\mu_L\alpha
                 &\nest \reduc{\dsig;l+1}\big[c_{i_p,s_{i_p}}\big] \nest \left[\Nest_{j={s_{i_p}+1}}^{N_{i_p}} \reduc{\dsig;0}\big[c_{i_p,j}\big]\right]\\
&\nest \cdots\\
                 &\nest \reduc{\dsig;l+1}\big[c_{i_1,s_{i_1}}\big] \nest \left[\Nest_{j={s_{i_1}+1}}^{N_{i_1}} \reduc{\dsig;0}\big[c_{i_1,j}\big]\right]
\end{aligned}
\end{align*}
     where the second equality is obtained on noting that it follows from the definition of $s_i$ that the first $s_i - 1$ child cycles off each vertex $\mu_i$ are mapped to $(\mu_i)$ by $\reduc{\dsig;l+1}$. Since trivial walks do not contribute to the nesting product, they can safely be deleted.  Thus, all cycles off $\mu_i$ (for $i \notin \mathcal{I}$) and the first $s_i -1$ cycles off $\mu_i$ (for $i \in \mathcal{I}$) disappear. Next we follow an analogous procedure to that outlined in case 1) above. Let $M_1$ be the number of the $N_{i_1} - s_{i_1}$ cycles $c_{i_1,s_{i_1}+1},\ldots, c_{i_1,N_{i_1}}$ off $\mu_{i_1}$ that are not mapped to $(\mu_{i_1})$ by $\reduc{\dsig;0}$, and let $s_{i_1}+1 \le t_{1,1} < \cdots < t_{1,M_1} \le N_{i_1}$ be the indices of these cycles. Introducing analogous quantities for $i_2,\ldots,i_p$, we arrive at the final expression for $\reduc{\dsig;l}\big[c\big]$:
\begin{equation}\label{eqn:rIdempotence:oneReduc}
\reduc{\dsig;l}\big[c\big]
=
\begin{aligned}[t]
\alpha\mu_2\cdots\mu_L\alpha
                 &\nest \reduc{\dsig;l+1}\big[c_{i_p,s_{i_p}}\big] \nest \left[\Nest_{k=1}^{M_p} \reduc{\dsig;0}\big[c_{i_p,t_{p,k}}\big]\right]\\
&\nest \cdots\\
                 &\nest \reduc{\dsig;l+1}\big[c_{i_1,s_{i_1}}\big] \nest \left[\Nest_{k=1}^{M_1} \reduc{\dsig;0}\big[c_{i_1,t_{1,k}}\big]\right].
\end{aligned}
\end{equation}
     We now wish to apply $\reduc{\dsig;l}$ to this result. Since the base cycle of $\reduc{\dsig;l}[c]$ has length $L \le k_l$, we first need to evaluate the indices $s_1',\ldots,s_p'$, where $s_1'$ is the index of the first cycle among the children $\reduc{\dsig;l+1}\big[c_{i_1,s_{i_1}}\big]$, $\reduc{\dsig;0}\big[c_{i_1,t_{1,1}}\big]$, \ldots, $\reduc{\dsig;0}\big[c_{i_1,t_{1,M_1}}\big]$ that is not mapped to $(\mu_{i_1})$ by $\reduc{\dsig;l+1}$. We claim that $s_1' = 1$. The proof of this is straightforward: by the definition of $s_{i_1}$, we know that $c_{i_1,s_{i_1}}$ satisfies $\reduc{\dsig;l+1}[c_{i_1,s_{i_1}}] \neq (\mu_{i_1})$. Further, from the induction hypothesis $\mathbb{P}(m_{1},l+1)$, where $m_{1} < n$ is the length of $c_{i_1,s_{i_1}}$, we have $\reduc{\dsig;l+1}\big[ \reduc{\dsig;l+1}[c_{i_1,s_{i_1}}]\big] = \reduc{\dsig;l+1}[c_{i_1,s_{i_1}}]$. Thus $\reduc{\dsig;l+1}\big[ \reduc{\dsig;l+1}[c_{i_1,s_{i_1}}]\big] \neq (\mu_{i_1})$, which proves our claim that $s_1' = 1$. An analogous argument shows that $s_2' = \cdots = s_p' = 1$. Thus, applying $\reduc{\dsig;l}$ to Eq.~\eqref{eqn:rIdempotence:oneReduc} yields
\begin{align*}
\reduc{\dsig;l}\big[\reduc{\dsig;l}[c]\big]
&=
\begin{aligned}[t]
\alpha\mu_2\cdots\mu_L\alpha
                 &\nest \reduc{\dsig;l+1}\big[\reduc{\dsig;l+1}[c_{i_p,s_{i_p}}]\big] \nest \left[\Nest_{k=1}^{M_p} \reduc{\dsig;0}\big[\reduc{\dsig;0}[c_{i_p,t_{p,k}}]\big]\right]\\
&\nest \cdots\\
                 &\nest \reduc{\dsig;l+1}\big[\reduc{\dsig;l+1}[c_{i_1,s_{i_1}}]\big] \nest \left[\Nest_{k=1}^{M_1} \reduc{\dsig;0}\big[\reduc{\dsig;0}[c_{i_1,t_{1,k}}]\big]\right]
\end{aligned}\\
&=  \reduc{\dsig;l}\big[c\big]
\end{align*}
     where the second equality follows on noting that the right-hand side of Eq.~\eqref{eqn:rIdempotence:oneReduc} is recovered immediately upon using the induction hypotheses $\mathbb{P}(m_{k,j},0)$ for $1 \le j \le M_k$ and $1 \le k \le p$, where $m_{k,j} < n$ is the length of $c_{i_k,t_{k,j}}$, and $\mathbb{P}(m_{k},l+1)$ for $1 \le k \le p$, where $m_k < n$ is the length of $c_{i_k,s_{i_k}}$.\\
\end{itemize}
\end{itemize}

\noindent\textit{Base case.}
     We wish to prove that $\mathbb{P}(1,l)$ holds for all $0 \le l \le D$. Since the only cycle of length 1 is the loop $\alpha\alpha$, this corresponds to the claim that
\begin{align*}
       \reduc{\dsig;l}\big[\reduc{\dsig;l}[\alpha\alpha]\big] = \reduc{\dsig;l}\big[\alpha\alpha\big].
\end{align*}
We identify two cases:
\begin{itemize}
       \item[1)] $ 0 \le l \le D-1$. Then $k_l \ge 1$, so $\reduc{\dsig;l}[\alpha\alpha] = (\alpha)$ and $\reduc{\dsig;l}[\reduc{\dsig;l}[\alpha\alpha]] = \reduc{\dsig;l}[(\alpha)] = (\alpha)$.
       \item[2)] $l = D$. Since $k_D = 0$, $\reduc{\dsig;D}[\alpha\alpha] = \alpha\alpha$, and $\reduc{\dsig;D}\big[\reduc{\dsig;D}[\alpha\alpha]\big] = \reduc{\dsig;D}[\alpha\alpha] = \alpha\alpha$.
\end{itemize}
\end{proof}

We now prove the claim that $\reduc{\dsig;l}$ maps a cycle $c$ to the trivial walk $(h(c))$ if and only if $c$ is $[k_l,\ldots,k_{D-1},0]$-structured. We begin by defining the set of all $\dsig$-structured cycles on $\G$.

\begin{definition}[The set of $\dsig$-structured cycles]\label{defn:CkCycleSets}
     Recall from Definition \ref{defn:KStructuredCycles} that, given a dressing signature $\dsig = [k_0,k_1,\ldots,k_{D-1},0]$, a cycle $c$ is $\dsig$-structured if and only if the base cycle of $c$ has length no greater than $k_0$ and any children of $c$ are $[k_1,\ldots,k_{D-1},0]$-structured. It follows that the set of all $\dsig$-structured cycles off $\alpha$ on $\G$, denoted by $C_{\G;\alpha}^\dsig$, is given by
\begin{align*}
C^{\dsig}_{\G;\alpha} = \begin{cases}
\emptyset & \text{if $D = 0$,}\\
\bigcup_{L=1}^{k_0} \limits S_{\G;\alpha;L}
\nest C_{\G\backslash\alpha\mu_2\cdots\mu_{L-1};\mu_L}^{[k_1,\ldots,k_{D-1},0]\:*}
\nest \cdots
\nest C_{\G\backslash\alpha;\mu_2}^{[k_1,\ldots,k_{D-1},0]\:*}&
\text{otherwise,}
\end{cases}
\end{align*}
     where $D$ is the depth of $\dsig$, $\alpha\mu_2\cdots\mu_L\alpha \in S_{\G;\alpha;L}$ is a simple cycle of length $L$ off $\alpha$ on $\G$, and $[k_1,\ldots,k_{D-1},0]$ is the dressing signature obtained by deleting the first element of $\dsig$. Further, the set of all $\dsig$-structured cycles on $\G$ is given by
\begin{align*}
C_\G^\dsig = \bigcup_{\alpha\,\in\,\G} C_{\G;\alpha}^\dsig.
\end{align*}
\end{definition}

\begin{prop}\label{prop:TrivialEquivalentCycles}
       Let $c$ be a cycle off $\alpha$ on a digraph $\G$, and $\dsig = [k_0,\ldots,k_{D-1},0]$ be a dressing signature of depth $D$. Then $\reduc{\dsig;l}\big[c\big] = (\alpha)$ if and only if $c \in C_{\G;\alpha}^{[k_l,\ldots,k_{D-1},0]}$, for $0 \le l \le D$.
\end{prop}
\begin{proof}
       Note firstly that $C_{\G;\alpha}^{[0]}$ is empty, and so the statement is vacuously true for $l = D$. It remains to prove that
\begin{align*}
\reduc{\dsig;l}\big[c\big] = (\alpha) \iff c \in C_{\G;\alpha}^{[k_l,\ldots,k_{d-1},0]}
\end{align*}
       for $l = 0,1,\ldots, D-1$. We prove the forward and backward directions separately. Before beginning the proof proper we note that it follows from Definition \ref{defn:CkCycleSets} that a cycle is an element of $C_{\G;\alpha}^{[k_l,\ldots,k_{D-1},0]}$ for $0 \le l \le D-1$ if and only if its decomposition is of the form
\begin{subequations}\label{eqn:Prop:FormOfDecomp}
\begin{align}
\alpha\mu_2\cdots\mu_L\alpha
\nest \bigg[\Nest_{j=1}^{N_L}c_{L,j}\bigg]
\nest \cdots
\nest \bigg[\Nest_{j=1}^{N_2}c_{2,j}\bigg]
\end{align}
where $\alpha\mu_2\cdots\mu_L\alpha$ is a simple cycle of length $L \le k_l$ off $\alpha$, $N_i \ge 0$ is the number of child cycles off $\mu_i$, and
\begin{align}
c_{i,j} \in C_{\G\backslash\alpha\mu_2\cdots\mu_{i-1};\mu_i}^{[k_{l+1},\ldots,k_{D-1},0]}
\end{align}
for $2\le i \le L$ and $1 \le j \le N_i$.
\end{subequations}\\

\noindent\textbf{Forward.} For $0 \le l \le D-1$, let $\mathbb{P}(l)$  be the statement
\begin{align*}
         \reduc{\dsig;l}\big[c\big] = (\alpha) \implies c \in C_{\G;\alpha}^{[k_l,\ldots,k_{d-1},0]}.
\end{align*}
       We will prove this by showing that the equality $\reduc{\dsig;l}\big[c\big] = (\alpha)$ implies that $c$ has the structure described in Eq.~\eqref{eqn:Prop:FormOfDecomp}, so that $c$ is an element of $C_{\G;\alpha}^{[k_l,\ldots,k_{d-1},0]}$. The proof proceeds by induction on the parameter $l$, starting at $l = D-1$ and working downwards to $l = 0$. Within the forward part of the proof, we assume that $c$ has decomposition
\begin{align*}
c = \alpha\beta_2\cdots\beta_P\alpha
\nest \left[\Nest_{j=1}^{M_P} b_{P,j}\right]
\nest \cdots
\nest \left[\Nest_{j=1}^{M_2}b_{2,j}\right],
\end{align*}
     where $\alpha\beta_2\cdots\beta_P\alpha$ is a simple cycle of length $P \ge 1$ off $\alpha$, $M_i \ge 0$ is the number of children off $\beta_i$ for $2 \le i \le P$, and $b_{i,j} \in C_{\G\backslash\alpha\cdots\beta_{i-1};\beta_i}$ for $1 \le j \le M_i$.\\

\noindent\textit{Inductive step.}
     The inductive step is the statement $\mathbb{P}(l+1) \implies \mathbb{P}(l)$ for $0 \le l \le D-2$. Referring to
     Definition \ref{defn:CycleReductionOperator}, we find that $\reduc{\dsig;l}[c] = (\alpha)$ implies that $P \le k_l$ and $s_i = M_i + 1$ for every $i$. Then from the definition of $s_i$, it follows that $\reduc{\dsig;l+1}[b_{i,j}] = (\beta_i)$ for all $1\le j \le M_i$. Then from the induction hypothesis $\mathbb{P}(l+1)$ we have that $b_{i,j} \in C_{\G\backslash\alpha\cdots\beta_{i-1};\beta_i}^{[k_{l+1},\ldots,k_{D-1},0]}$. Thus the decomposition of $c$ has the structure of Eq.~\eqref{eqn:Prop:FormOfDecomp}, and $c$ is in $C_{\G;\alpha}^{[k_l,\ldots,k_{D-1},0]}$.\\

\noindent\textit{Base case.} We wish to prove the statement $\mathbb{P}(D-1)$.
     In this case, $\reduc{\dsig;D-1}\left[c\right] = (\alpha)$ implies that $P \le k_{D-1}$ and all child cycles $b_{i,j}$ (if any) satisfy $\reduc{\dsig;D}[b_{i,j}] = (\beta_i)$. Note that a necessary condition for $\reduc{\dsig;D}[b_{i,j}] = (\beta_i)$ is that $L_{i,j} \le k_D$, where $L_{i,j}$ is the length of the base cycle of $b_{i,j}$. However, by the definition of a dressing signature we have $k_D = 0$. Since there are no cycles with length 0, there are no cycles that satisfy $\reduc{\dsig;D}[b_{i,j}] = (\beta_i)$. Thus, $c$ does not have any children; i.e.~$M_i = 0$ for all $i$. Then $c$ is a simple cycle of length $P \le k_{D-1}$, and is an element of $C^{[k_{D-1},0]}_{\G;\alpha} = \bigcup_{L=1}^{k_{D-1}} C_{\G;\alpha;L}$.\\

\noindent\textbf{Backward.} For $0 \le l \le D-1$, let $\mathbb{P}(l)$  be the statement
\begin{align*}
         c \in C_{\G;\alpha}^{[k_l,\ldots,k_{D-1},0]} \implies \reduc{\dsig;l}\big[c\big] = (\alpha).
\end{align*}
     Since $c$ is an element of $C_{\G;\alpha}^{[k_l,\ldots,k_{D-1},0]}$, its decomposition is of the form given in Eq.~\eqref{eqn:Prop:FormOfDecomp}. In particular, the length of the base cycle of $c$ is $L \le k_l$. Thus to prove $\mathbb{P}(l)$ it is sufficient to show that $s_i = N_i + 1$ for all $2 \le i \le L$: that is, that all of the $N_i$ children $c_{i,j}$ off $\mu_i$ satisfy $\reduc{\dsig;l+1}[c_{i,j}] = (\mu_i)$. The proof again proceeds by induction on the parameter $l$, from $l = D-1$ down to $l = 0$.\\

\noindent\textit{Inductive step.}
     The inductive step is the statement $\mathbb{P}(l+1) \implies \mathbb{P}(l)$ for $0 \le l \le D-2$.
     From Eq.~\eqref{eqn:Prop:FormOfDecomp}, we see that $c_{i,j} \in C_{\G\backslash\alpha\mu_2\cdots\mu_{i-1};\mu_i}^{[k_{l+1},\ldots,k_{D-1},0]}$ for $2\le i \le L$ and $1\le j \le N_i$. Then from the induction hypothesis $\mathbb{P}(l+1)$ we have $\reduc{\dsig;l+1}\big[c_{i,j}\big] = (\mu_i)$, so that $s_i = N_i+1$ for all $2 \le i \le L$. It then follows from the definition of $\reduc{\dsig;l}$ that $\reduc{\dsig;l}[c] = (\alpha)$.\\

\noindent\textit{Base case.}
     The base case of the induction is the statement $\mathbb{P}(D-1)$. Then $c$ is a member of $C^{[k_{D-1},0]}_{\G;\alpha}$; that is, a simple cycle of length no greater than $k_{D-1}$. Since $c$ has no children, we have $s_i = N_i + 1 = 1$ for all $2\le i \le L$, and so $\reduc{\dsig;D-1}[c] = (\alpha)$.
\end{proof}

\subsection{The walk reduction operator}\label{sec:WalkReductionOperators}
Having defined the cycle reduction operators $\reduc{\dsig;l}$, we are now in a position to give the explicit definition of the walk reduction operator $\Reduc{\dsig}$. As noted, the action of $\Reduc{\dsig}$ on a walk $w$ is to remove all resummable cycles from $w$, thereby projecting $w$ to its $\dsig$-irreducible core walk $\Reduc{\dsig}[w] \in I^\dsig_{\G}$.

\begin{definition}\label{defn:WalkReductionOperator}
Let $\G$ be a digraph and $\dsig$ be a dressing signature. Then the walk reduction operator $\Reduc{\dsig} : W_\G \rightarrow W_\G$ is a map on $W_\G$ defined as follows. Let $w$ be a walk on $\G$ with decomposition (cf.~Eq.~\eqref{eqn:WalkPrimeFactorisation})
\begin{align*}
w = \alpha\mu_2\cdots\mu_m\omega
\nest \left[\Nest_{j=1}^{N_{m+1}} c_{m+1,j}\right]
\nest \cdots
\nest \left[\Nest_{j=1}^{N_2} c_{2,j}\right]
\nest \left[\Nest_{j=1}^{N_1} c_{1,j}\right],
\end{align*}
     where $\alpha\mu_2\cdots\mu_m\omega \in P_\G$ is a simple path of length $m \ge 0$ on $\G$, $N_i \ge 0$ is the number of child cycles off $\mu_i$, and $c_{i,j} \in C_{\G\backslash\alpha\cdots\mu_{i-1};\mu_i}$ is a cycle off $\mu_i$. Then $\Reduc{\dsig}$ maps $w$ to
\begin{align*}
\Reduc{\dsig}\big[w\big] =\alpha\mu_2\cdots\mu_m\omega
\nest \left[\Nest_{j=1}^{N_{m+1}} \reduc{\dsig;\,0}\big[c_{m+1,j}\big]\right]
\nest \cdots
\nest \left[\Nest_{j=1}^{N_2} \reduc{\dsig;\,0}\big[c_{2,j}\big] \right]
\nest \left[\Nest_{j=1}^{N_1} \reduc{\dsig;\,0}\big[c_{1,j}\big]\right]
\end{align*}
     where $\reduc{\dsig;0}$ is the cycle reduction operator of Definition \ref{defn:CycleReductionOperator}.
\end{definition}

We now prove that $\Reduc{\dsig}$ is idempotent, thus showing that the explicit definition we give here satisfies property P1 of Definition \ref{defn:ReductionDressingAxioms}.
\begin{prop}\label{prop:IdempotenceOfR}
The walk reduction operator $\Reduc{\dsig}$ is idempotent.
\end{prop}
The proof follows straightforwardly from the fact that $\reduc{\dsig;0}$ is idempotent.
\begin{proof}
Let $w \in W_\G$ be a walk on $\G$ with decomposition
\begin{align*}
w = \alpha\mu_2\cdots\mu_L\omega
\nest \left[\Nest_{j=1}^{N_{L+1}}c_{L+1,j}\right]
\nest \cdots
\nest \left[\Nest_{j=1}^{N_1}c_{1,j}\right]
\end{align*}
     where $\alpha\mu_2\cdots\mu_L\omega$ is a simple path of length $L \ge 0$ and $c_{i,j} \in C_{\G\backslash\alpha\cdots\mu_{i-1};\mu_i}$ for $1\le i \le L+1$ and $1 \le j \le N_i$. Then
\begin{align}
\Reduc{\dsig}\big[w\big] &= \alpha\mu_2\cdots\mu_L\omega
\nest \left[\Nest_{j=1}^{N_{L+1}}\reduc{\dsig;0}\big[c_{L+1,j}\big]\right]
\nest \cdots
\nest \left[\Nest_{j=1}^{N_1}\reduc{\dsig;0}\big[c_{1,j}\big]\right]\nonumber\\
&= \alpha\mu_2\cdots\mu_L\omega
\nest \left[\Nest_{k=1}^{M_{L+1}}\reduc{\dsig;0}\big[c_{L+1,t_{L+1,k}}\big]\right]
\nest \cdots
       \nest \left[\Nest_{k=1}^{M_1}\reduc{\dsig;0}\big[c_{1,t_{1,k}}\big]\right]\label{eqn:IdempotenceOfR:OneReduc}
\end{align}
     where $0 \le M_i \le N_i$ for $1 \le i \le L+1$ is the number of cycles among $c_{i,1},\ldots,c_{i,N_i}$ that are not mapped to $(\mu_i)$ by $\reduc{\dsig;0}$, and $t_{i,k}$ are their indices, which we assume without loss of generality to satisfy $1 \le t_{i,1} < \cdots < t_{i,M_i} \le N_i$. Applying $\Reduc{\dsig}$ again to this result gives
\begin{align*}
\Reduc{\dsig}\big[\Reduc{\dsig}[w]\big]
&= \alpha\mu_2\cdots\mu_L\omega
       \nest \left[\Nest_{k=1}^{M_{L+1}}\reduc{\dsig;0}\big[\reduc{\dsig;0}[c_{L+1,t_{L+1,k}}]\big]\right]
\nest \cdots
       \nest \left[\Nest_{k=1}^{M_1}\reduc{\dsig;0}\big[\reduc{\dsig;0}[c_{1,t_{1,k}}]\big]\right]\\
&=\Reduc{\dsig}\big[w\big],
\end{align*}
     where the second equality is obtained on noting that $\reduc{\dsig;0}$ is idempotent (see Proposition \ref{prop:IdempotenceOfr}) and comparing with Eq.~\eqref{eqn:IdempotenceOfR:OneReduc}.
\end{proof}

\subsection{The set of $(\dsig,l)$-irreducible cycles} \label{subsec:StructureOfQKL}
Recall from Section \ref{sec:42:RandDeltaMotivation} that the set of $(\dsig,l)$-irreducible cycles off $\alpha$ on $\G$, denoted by $Q^{\dsig;l}_{\G;\alpha}$, is the set of all cycles off $\alpha$ on $\G$ that satisfy $\reduc{\dsig;l}[q] = q$, where $\reduc{\dsig;l}$ is defined in Section \ref{sec:CycleReductionOperators}. The $(\dsig,l)$-irreducible cycles are of central importance in defining the set of $\dsig$-irreducible walks. In this section we use the definition of $\reduc{\dsig;l}$ given in Defn.~\ref{defn:CycleReductionOperator} to develop an explicit recursive formula for $Q^{\dsig;l}_{\G;\alpha}$.

To understand the structure of the cycles that appear in $Q^{\dsig;l}_{\G;\alpha}$, let us consider what structure a cycle must have in order to be $(\dsig,l)$-irreducible. We begin by addressing the question of which simple cycles off $\alpha$ -- if any -- are $(\dsig,l)$-irreducible.

Let $s = \alpha\mu_2\cdots\mu_L\alpha$ be a simple cycle off $\alpha$ on $\G$. Since $s$ has no child cycles, the action of $\reduc{\dsig;l}$ on $s$ is straightforward: referring to Definition \ref{defn:CycleReductionOperator} we find $\reduc{\dsig;l}[s] = s$ if $L > k_l$, and $\reduc{\dsig;l}[s] = (\alpha)$ if $L \le k_l$. Thus every simple cycle that is longer than $k_l$ is $(\dsig,l)$-irreducible, while every simple cycle that is shorter than $k_l$ is not. However, in the latter case it is easy to modify $s$ so as to obtain a $(\dsig,l)$-irreducible cycle. The simplest way is to add a single child cycle off an internal vertex of $s$. Explicitly, consider adding a child $c$ off an arbitrary internal vertex of $s$. Then
\begin{align*}
\reduc{\dsig;l}\big[\alpha\mu_2\cdots\mu_L\alpha \nest c\big] =
\begin{cases}
(\alpha) & \text{if $\reduc{\dsig;l+1}[c] = h(c)$,}\\
\alpha\mu_2\cdots\mu_L\alpha \nest \reduc{\dsig;l+1}[c] & \text{else.}
\end{cases}
\end{align*}
It follows that if $\reduc{\dsig;l+1}\big[c\big] = c$, then $\reduc{\dsig;l}\big[\alpha\mu_2\cdots\mu_L\alpha \nest c\big] =  \alpha\mu_2\cdots\mu_L\alpha \nest c$. In other words, although a simple cycle of length less than or equal to $k_l$ is itself not $(\dsig,l)$-irreducible, it can be extended into a $(\dsig,l)$-irreducible cycle by the simple step of nesting a $(\dsig,l+1)$-irreducible child off one of its internal vertices. This reasoning also holds for the case where one $(\dsig,l+1)$-irreducible child is nested off each vertex of a non-empty subset $\{\mu_{i_1},\ldots,\mu_{i_p}\}$ of the internal vertices of $s$. For example, if $k_l = 3$, $s = \alpha\mu_2\mu_3\alpha$, and $q_2, q_3$ are $(\dsig,l+1)$-irreducible cycles off $\mu_2$ and $\mu_3$ respectively, then $\alpha\mu_2\mu_3\alpha \nest q_2$, $\alpha\mu_2\mu_3\alpha \nest q_3$, and $\alpha\mu_2\mu_3\alpha \nest q_3 \nest q_2$ are all $(\dsig,l)$-irreducible.

This line of reasoning shows that every element of the following set is $(\dsig,l)$-irreducible:
\begin{align}
       \underbrace{\big(S_{\G;\alpha}\setminus \cup_{L=1}^{k_l}S_{\G;\alpha;L}\big)}_{\text{simple cycles longer than $k_l$}}
\cup \underbrace{\bigcup_{L=1}^{k_l} \bigcup_{\mathcal{I} \in
\nepowset{\{2,\ldots,L\}}} S_{\G;\alpha;L}
\nest Q^{\dsig;l+1}_{\G\backslash\alpha\cdots\mu_{i_p-1};\mu_{i_p}}
\nest \cdots
       \nest Q^{\dsig;l+1}_{\G\backslash\alpha\cdots\mu_{i_1-1};\mu_{i_1}}}_{\substack{\text{simple cycles not longer than $k_l$ with}\\\text{one or more $(\dsig,l+1)$-irreducible child cycles}}}\label{eqn:KLirredSubset}
\end{align}
where $S_{\G;\alpha;L}$ is the set of simple cycles of length $L$ off $\alpha$, $\alpha\mu_2\cdots\mu_L\alpha$ is a simple cycle of length $L$ off $\alpha$, and $\mathcal{I} = \{i_1,\ldots,i_p\}$ is a non-empty subset of $\{2,\ldots,L\}$ that indexes $p$ internal vertices of $\alpha\mu_2\cdots\mu_L\alpha$, where $1 \le p \le L-1$. [We denote the set of all non-empty subsets of $\{2,\ldots,L\}$ by $\nepowset{\{2,\ldots,L\}}$].

Although every element of the set given in Eq.~\eqref{eqn:KLirredSubset} is $(\dsig,l)$-irreducible, not every $(\dsig,l)$-irreducible cycle appears in Eq.~\eqref{eqn:KLirredSubset}: other $(\dsig,l)$-irreducible cycles can be constructed by taking an element $c$ of this set and adding one or more child cycles off its internal vertices. These child cycles must be carefully chosen so that their presence does not spoil the $(\dsig,l)$-irreducibility of the entire cycle. The structure that the child cycles are permitted to have can be determined from the action of $\reduc{\dsig;l}$. There are two cases, corresponding to the two subsets that are braced in Eq.~\eqref{eqn:KLirredSubset}.

Consider first the case where $c$ comes from the first braced term: i.e.~$c$ is a simple cycle longer than $k_l$. Then $\reduc{\dsig;l}[c]$ applies $\reduc{\dsig;0}$ to every immediate child of $c$. It follows that so long as every child cycle we add is $(\dsig,0)$-irreducible, the new cycle produced from $c$ remains $(\dsig,l)$-irreducible. Alternatively, consider the case where $c$ comes from the second braced term. Then its base cycle is not longer than $k_l$, and at least one internal vertex of the base cycle has precisely one child, which is $(\dsig,l+1)$-irreducible. Consider adding further children to $c$. Then $\reduc{\dsig;l}[c]$ applies $\reduc{\dsig;l+1}[c]$ to the first child off each vertex, and (since the first child is $(\dsig,l+1)$-irreducible and is thus not mapped to a trivial walk) $\reduc{\dsig;0}$ to all subsequent children. Thus as long as the child cycles we add are (i) attached to a vertex that already has a child, (ii) positioned \key{after} the first child, and (iii) $(\dsig,0)$-irreducible, the resultant cycle remains $(\dsig,l)$-irreducible.

We therefore conclude that a cycle is $(\dsig,l)$-irreducible if and only if (a) its base cycle is longer than $k_l$, and all of its immediate child cycles (if any) are $(\dsig,0)$-irreducible; or (b) its base cycle is not longer than $k_l$, it has at least one child cycle, and the first child cycle off a given vertex is $(\dsig,l+1)$-irreducible and all subsequent children (if any) in that hedge are $(\dsig,0)$-irreducible. We formalise this statement in the following theorem.

\begin{theorem}\label{thm:StructureOfQKL}
     Let $\G$ be a digraph and $\dsig$ be a dressing signature of depth $D$. Then for $0 \le l \le D$ the set of all $(\dsig,l)$-irreducible cycles off $\alpha$ on $\G$ is
\begin{align}
Q^{\dsig;l}_{\G;\alpha} =
&\big(S_{\G;\alpha}\setminus \cup_{k=1}^{k_l} S_{\G;\alpha;k}\big)
\nest Q^{\dsig;0\:*}_{\G\backslash\alpha\cdots\beta_{k-1};\beta_k}
\nest \cdots
\nest Q^{\dsig;0\:*}_{\G\backslash\alpha;\beta_2}\label{eqn:StructureOfQKL}\\
         &\cup \bigcup_{L=1}^{k_l}\, \bigcup_{\mathcal{I}\,\in\, \nepowset{\{2,\ldots,L\}} } S_{\G;\alpha;L}
\begin{aligned}[t]
&\nest Q^{\dsig;l+1}_{\G\backslash\alpha\cdots\mu_{i_p-1};\mu_{i_p}}
\nest Q^{\dsig;0\:*}_{\G\backslash\alpha\cdots\mu_{i_p-1};\mu_{i_p}} \\
&\nest \cdots \\
&\nest Q^{\dsig;l+1}_{\G\backslash\alpha\cdots\mu_{i_1-1};\mu_{i_1}}
\nest Q^{\dsig;0\:*}_{\G\backslash\alpha\cdots\mu_{i_1}-1;\mu_{i_1}}
\end{aligned}
\nonumber
\end{align}
     where $\alpha\beta_2\cdots\beta_k\alpha$ is a simple cycle of length $k > k_l$ off $\alpha$,
     $\alpha\mu_2\cdots\mu_L\alpha$ is a simple cycle of length $L \le k_l$ off $\alpha$, and $\mathcal{I} = \{i_1,\ldots,i_p\}$ is a non-empty subset of $\{2,\ldots,L\}$ containing the indices of at least one internal vertex of $\alpha\mu_2\cdots\mu_L\alpha$.
\end{theorem}

\begin{proof}
     Theorem \ref{thm:StructureOfQKL} corresponds to the statement that a cycle off $\alpha$ on $\G$ is $(\dsig,l)$-irreducible if and only if it is an element of $Q^{\dsig;l}_{\G;\alpha}$; that is
\begin{align}
c\in Q^{\dsig;l}_{\G;\alpha} \iff \reduc{\dsig;\,l}\big[c\big] = c.\nonumber
\end{align}
     We prove the forward and backward directions of this statement separately. The proof proceeds by induction on the number of vertices in $\G$. We begin with a general remark that will be required in both sections of the proof.\\

     \noindent From Eq.~\eqref{eqn:StructureOfQKL} and the discussion preceding it, it follows that a cycle $c$ is in $Q^{\dsig;l}_{\G;\alpha}$ if and only if one of two conditions holds. In each case we denote the base cycle of $c$ by $s$.
\begin{enumerate}
    \item[a)] If $s$ is longer than $k_l$ and all immediate children of $s$ are $(\dsig,0)$-irreducible, then $c$ appears in the first line of Eq.~\eqref{eqn:StructureOfQKL}, and has decomposition
\begin{align}
\alpha\mu_2\cdots\mu_L\alpha
\nest \bigg[\Nest_{j=1}^{N_L} c_{L,j}\bigg]
\nest \cdots
\nest \bigg[\Nest_{j=1}^{N_2} c_{2,j}\bigg]\label{eqn:StructureOfQKL:FormA1}
\end{align}
where $L > k_l$, $N_i \ge 0$ for $2 \le i \le L$, and $c_{i,j} \in Q^{\dsig;0}_{\G\backslash\alpha\cdots\mu_{i-1};\mu_i}$ for $1 \le j \le N_i$.\\

    \item[b)] If $s$ is not longer than $k_l$, has at least one child, and the first child off each internal vertex of $s$ is $(\dsig,l+1)$-irreducible and any remaining children in the same hedge are $(\dsig,0)$-irreducible, then $c$ appears in the second part of Eq.~\eqref{eqn:StructureOfQKL}. Then $c$ has decomposition
\begin{subequations}\label{eqn:StructureOfQKL:FormBGroup}
\begin{align}
\alpha\mu_2\cdots\mu_L\alpha
\nest \left[\Nest_{j=1}^{N_{i_p}} c_{i_p,j}\right]
\nest \cdots
\nest \left[\Nest_{j=1}^{N_{i_1}} c_{i_1,j}\right]\label{eqn:StructureOfQKL:FormB1}
\end{align}
     where $L \le k_l$, $\{i_1,\ldots, i_p\}$ is a non-empty subset of $\{2,\ldots, L\}$, $N_{i_k} \ge 1 $ for $1 \le k \le p$, and
\begin{align}
c_{i_k,j} \in \begin{cases}
Q^{\dsig;l+1}_{\G\backslash\alpha\cdots\mu_{i_k-1};i_k}
& \text{if $j = 1$,}\\[2mm]
Q^{\dsig;0}_{\G\backslash\alpha\cdots\mu_{i_k-1};i_k}
& \text{else.}
\end{cases}\label{eqn:StructureOfQKL:FormB2}
\end{align}
\end{subequations}
\end{enumerate}
\noindent In the forward direction of the proof, we show that any cycle $c$ that has the structure of Eq.~\eqref{eqn:StructureOfQKL:FormA1} or Eq.~\eqref{eqn:StructureOfQKL:FormBGroup} satisfies $\reduc{\dsig;l}[c] = c$. In the backward direction, we show that any cycle that satisfies $\reduc{\dsig;l}[c] = c$ must have the structure of either Eq.~\eqref{eqn:StructureOfQKL:FormA1} or Eq.~\eqref{eqn:StructureOfQKL:FormBGroup}, and so is an element of $Q^{\dsig;l}_{\G;\alpha}$.\\

\noindent\textbf{Forward.}
     Let $n \ge 1$ be the number of vertices in $\G$, and $\mathbb{L}(n,l)$ for $n \ge 1$ and $0\le l \le D$ be the statement
\begin{align}
c\in Q^{\dsig;l}_{\G;\alpha}\implies
\reduc{\dsig;\,l}\big[c\big] = c. \nonumber
\end{align}
     We prove this statement for all $n$ and $l$ by induction on $n$. For a given $n$, the statements required to prove $\mathbb{L}(n,l)$ depend on $l$. Specifically, $\mathbb{L}(n,l)$ for any $0 \le l \le D-1$ requires $\mathbb{L}(m,l+1)$ and $\mathbb{L}(m,D)$ for all $m < n$, while $\mathbb{L}(n,D)$ only requires $\mathbb{L}(m,0)$ for $m < n$.
\\

\noindent \textit{Inductive step.} The inductive step corresponds to the statements
\begin{align*}
\mathbb{L}(m,0) \text{ for all $m < n$ } &\implies \mathbb{L}(n,D),\\
         \mathbb{L}(m,l+1) \text{ and } \mathbb{L}(m,0) \text{ for all $m < n$ } &\implies \mathbb{L}(n,l) \text{ for $0 \le l \le D-1$}.
\end{align*}
       As noted above, given an element $c$ of $Q^{\dsig;l}_{\G;\alpha}$, there are two possibilities:
\begin{itemize}
       \item[a)] The base cycle of $c$ is longer than $k_l$ (note that this is necessarily the case if $l = D$, since $k_D = 0$). Then $c$ has the form of Eq.~\eqref{eqn:StructureOfQKL:FormA1}, and since $\reduc{\dsig;\,l}$ applies $\reduc{\dsig;0}$ to every child cycle, we have
\begin{align*}
\reduc{\dsig;\,l}\big[c\big] &= \alpha\mu_2\cdots\mu_L\alpha
\nest \left[\Nest_{j=1}^{N_L} \reduc{\dsig;\,0}\big[c_{L,j}\big]\right]
\nest \cdots
\nest \left[\Nest_{j=1}^{N_2} \reduc{\dsig;\,0}\big[c_{2,j}\big]\right] \\
&=\alpha\mu_2\cdots\mu_L\alpha
\nest \left[\Nest_{j=1}^{N_L} c_{L,j}\right]
\nest \cdots
\nest \left[\Nest_{j=1}^{N_2} c_{2,j}\right]\\
& = c,
\end{align*}
       where the second equality follows from the induction hypotheses $\mathbb{L}(m_i,0)$ for $2\le i \le L$, where $m_i < n$ is the number of vertices in $\mathcal{G}\backslash\alpha\cdots\mu_{i-1}$.\\

       \item[b)] The base cycle of $c$ is not longer than $k_l$. Then $c$ has the form of Eq.~\eqref{eqn:StructureOfQKL:FormBGroup}. In this case applying $\reduc{\dsig;\,l}$ to $c$ applies $\reduc{\dsig;\,l+1}$ to the first $\min(N_i,s_i)$ children off a given internal vertex $\mu_i$, and $\reduc{\dsig;0}$ to any remaining children. Recall from \S\ref{sec:CycleReductionOperators} that $s_i$ is the index of the first child off $\mu_i$ that satisfies $\reduc{\dsig;l+1}[c_{i,s_i}] \neq (\mu_i)$. We claim that $s_i = 1$. This can be proven as follows. From Eq.~\eqref{eqn:StructureOfQKL:FormBGroup}, we have that $c_{i,1}$ is an element of $Q^{\dsig;l+1}_{\G\backslash\alpha\cdots\mu_{i-1};\mu_i}$.
           Since $(\mu_i)$ is not in $Q^{\dsig;l+1}_{\G\backslash\alpha\cdots\mu_{i-1};\mu_i}$ (see e.g.~Defn.~\ref{defn:KLIrreducibleCycles}) it follows that $c_{i,1} \neq (\mu_i)$. By the induction hypothesis $\mathbb{L}(m_i,l+1)$, where $m_i < n$ is the number of vertices in $\G\backslash\alpha\cdots\mu_{i-1}$, we have that $\reduc{\dsig;l+1}[c_{i,1}] = c_{i,1}$. Thus $\reduc{\dsig;l+1}[c_{i,1}] \neq (\mu_i)$, and $s_i = 1$. Since this argument holds for $i = i_1,i_2,\ldots,i_p$, we have $s_{i_1} = s_{i_2} = \cdots = s_{i_p} = 1$, so that
\begin{align*}
\reduc{\dsig;l}\big[c\big] &=
\alpha\mu_2\cdots\mu_L\alpha
             \nest \left[c_{i_p,1}\nest \Nest_{j=2}^{N_{i_p}} \reduc{\dsig;0}[c_{i_p,j}]\right]
            \nest \cdots
            \nest \left[c_{i_1,1}\nest \Nest_{j=2}^{N_{i_1}} \reduc{\dsig;0}[c_{i_1,j}]\right]\\
&=   \alpha\mu_2\cdots\mu_L\alpha
\nest \left[\Nest_{j=1}^{N_{i_p}} c_{i_p,j}\right]
\nest \cdots
\nest \left[\Nest_{j=1}^{N_{i_1}} c_{i_1,j}\right]\\
&= c,
\end{align*}
where in obtaining the first equality we have already used the fact that $\reduc{\dsig;l+1}[c_{i,1}] = c_{i,1}$ for $i = i_1,\ldots,i_p$, and the second equality follows from applying the induction hypotheses $\mathbb{L}(m_1,0),\ldots,\mathbb{L}(m_p,0)$, where $m_k < n$ is the number of vertices in the graph $\G\backslash\alpha\cdots\mu_{i_k-1}$.\\
\end{itemize}

\noindent\textit{Base case.} The base case of the induction is the statement $\mathbb{L}(1,l)$ for $0 \le l \le D$. We prove this statement by considering each of the two possible graphs on a single vertex.
\begin{itemize}
         \item[a)] $\G$ has no loop. Then there are no cycles on $\G$, so $Q^{\dsig;l}_{\G;\alpha}$ is empty for every $l$, and $\mathbb{L}(1,l)$ is vacuously true.
\item[b)] $\G$ has a loop. Then there are two possibilities:
\begin{itemize}
           \item[i)] $0 \le l \le D-1$. Then $Q^{\dsig;l}_{\G;\alpha}$ is empty, and $\mathbb{L}(1,l)$ is vacuously true.
           \item[ii)] $ l = D$. Then $Q^{\dsig;D}_{\G;\alpha} = \{\alpha\alpha\}$. Since $\reduc{\dsig;D}[\alpha\alpha] = \alpha\alpha$, then $\mathbb{L}(1,D)$ holds.\\
\end{itemize}
\end{itemize}

\noindent\textbf{Backward.}
     Let $n \ge 1$ be the number of vertices in $\G$, and $\mathbb{L}(n,l)$ for $n \ge 1$ and $0\le l \le d$ be the statement
\begin{align*}
\reduc{\dsig;\,l}\big[c\big] = c \implies c\in Q^{\dsig;l}_{\G;\alpha}.
\end{align*}
     The proof again proceeds by induction on $n$, with the same dependency structure as in the forward direction.\\

\noindent \textit{Inductive step.} The inductive step corresponds to the statements
\begin{align*}
\mathbb{L}(m,0) \text{ for all $m < n$ } &\implies \mathbb{L}(n,D),\\
         \mathbb{L}(m,l+1) \text{ and } \mathbb{L}(m,0) \text{ for all $m < n$ } &\implies \mathbb{L}(n,l) \text{ for $0 \le l \le D-1$}.
\end{align*}
     Let the decomposition of $c$ into a simple cycle and a collection of cycles be
\begin{align}
c = \alpha\mu_2\cdots\mu_L\alpha
\nest \left[\Nest_{j=1}^{N_L}c_{L,j}\right]
\nest \cdots
\nest \left[\Nest_{j=1}^{N_2}c_{2,j}\right]\label{eqn:StructureOfQKL:DecompOfC}
\end{align}
         where $N_i \ge 0$ for $2\le i \le L$, and $c_{i,j} \in C_{\G\backslash\alpha\cdots\mu_{i-1};\mu_i}$ for $1\le j \le N_i$. We will show that it follows from $c = \reduc{\dsig;l}[c]$ that the decomposition of $c$ is of the form of either Eq.~\eqref{eqn:StructureOfQKL:FormA1} or Eq.~\eqref{eqn:StructureOfQKL:FormBGroup}, so that $c$ is in $Q^{\dsig;l}_{\G;\alpha}$. Indeed, consider $\reduc{\dsig;l}[c]$; then there are two possibilities:
\begin{itemize}
    \item[a)] $L > k_l$. Then $\reduc{\dsig;l}$ applies $\reduc{\dsig;0}$ to every child of c, and the equality $c = \reduc{\dsig;l}\big[c\big]$ corresponds to
    \begin{multline*}
        \alpha\mu_2\cdots\mu_L\alpha
        \nest \left[\Nest_{j=1}^{N_L}c_{L,j}\right]
        \nest \cdots
        \nest \left[\Nest_{j=1}^{N_2}c_{2,j}\right]\\
        =
        \alpha\mu_2\cdots\mu_L\alpha
        \nest \left[\Nest_{j=1}^{N_L}\reduc{\dsig;0}\big[c_{L,j}\big]\right]
        \nest \cdots
        \nest \left[\Nest_{j=1}^{N_2}\reduc{\dsig;0}\big[c_{2,j}\big]\right]
    \end{multline*}
Since the decomposition of a cycle into a simple cycle plus a collection of cycles is unique \cite{Giscard2012}, this equality can only be satisfied if $c_{i,j} = \reduc{\dsig;0}\big[c_{i,j}\big]$ for all $c_{i,j}$. In other words, in order for the entire cycle $c$ to remain unmodified by $\reduc{\dsig;l}$, each of the children $c_{i,j}$ must remain unmodified by $\reduc{\dsig;0}$. It then follows from the induction hypotheses $\mathbb{L}(m_i,0)$, where $m_i < n$ is the number of vertices in $\G\backslash\alpha\cdots\mu_{i-1}$ (i.e.~the graph on which $c_{i,j}$ exists) that $c_{i,j} \in Q^{\dsig;0}_{\G\backslash\alpha\cdots\mu_{i-1};\mu_i}$. Then $c$ has the form of Eq.~\eqref{eqn:StructureOfQKL:FormA1}, and so appears in $Q^{\dsig;l}_{\G;\alpha}$.

\item[b)] $L \le k_l$. In this case $\reduc{\dsig;l}\big[c\big]$ depends on the indices $s_2,\ldots,s_L$. Recall that $s_i$ is the index of the first child off $\mu_i$ that is not $[k_l,\ldots,k_{D-1},0]$-structured, with $s_i = N_i + 1$ if no such child exists. Then $\reduc{\dsig;l}[c] = (\alpha)$ if and only if $s_i = N_i + 1$ for every $i = 2,\ldots, L$.

    It is convenient to explicitly divide the vertices $\mu_i$ into those which have no children, and those which have at least one child. Let $\mathcal{I} = \{i_1,\ldots,i_p\} \subseteq \{2,\ldots,L\}$ be the index set of the vertices that have at least one child, so that $N_i \neq 0$ if and only if $i \in \mathcal{I}$. Then the decomposition of $c$ is
    \begin{align}
        c = \alpha\mu_2\cdots\mu_L\alpha
        \nest \left[\Nest_{j=1}^{N_{i_p}}c_{i_p,j}\right]
        \nest \cdots
        \nest \left[\Nest_{j=1}^{N_{i_1}}c_{i_1,j}\right] \label{eqn:StructureOfQKL:cstruc}
    \end{align}
    Then $i \notin \mathcal{I} \implies s_i = N_i + 1$. Recall that if every $s_i = N_i +1$, then $\reduc{\dsig;l}[c] = (\alpha)$. Thus if (i) $\mathcal{I}$ is empty, or (ii) $\mathcal{I}$ is not empty but $s_i = N_i + 1$ for every $i \in \mathcal{I}$, then $\reduc{\dsig;l}[c] = (\alpha)$. However, it follows from the fact that $\reduc{\dsig;l}\big[c\big] = c$ that $\reduc{\dsig;l}\big[c\big] \neq (\alpha)$. Thus we have that $\mathcal{I}$ is not empty and at least one of the $s_i : i \in \mathcal{I}$ is smaller than $N_i + 1$. From Definition \ref{defn:CycleReductionOperator} we find that $\reduc{\dsig;l}[c]$ evaluates to
    \begin{equation*}
    \reduc{\dsig;l}\big[c\big] = \alpha\mu_2\cdots\mu_L\alpha
    \begin{aligned}[t]
    &\nest \reduc{\dsig;l+1}\big[c_{i_p,s_{i_p}}\big]
    \nest \left[\Nest_{j=s_{i_p}+1}^{N_{i_p}} \reduc{\dsig;0}[c_{i_p,j}]\right]\\
    &\nest \cdots \\
    &\nest \reduc{\dsig;l+1}\big[c_{i_1,s_{i_1}}\big]
    \nest \left[\Nest_{j=s_{i_1}+1}^{N_{i_1}} \reduc{\dsig;0}[c_{i_1,j}]\right],
    \end{aligned}
    \end{equation*}
    where we have deleted the first $s_i -1$ cycles off $\mu_i$ for $i \in \mathcal{I}$, since these cycles are all mapped to $(\mu_i)$ by $\reduc{\dsig;l+1}$ and do not contribute to the nesting product. Comparing Eq.~\eqref{eqn:StructureOfQKL:cstruc} to this expression, we find that the equality $c = \reduc{\dsig;l}[c]$ can only hold if
    \begin{align}
      \Nest_{j=1}^{N_{i}}c_{i,j}= \reduc{\dsig;l+1}\big[c_{i,s_{i}}\big]
    \nest \Nest_{j=s_{i}+1}^{N_{i}} \reduc{\dsig;0}\big[c_{i,j}\big]\label{eqn:StructureOfQKL:concat}
    \end{align}
    for every $i \in \mathcal{I}$. We now note that each of the terms on each side of this expression is a cycle, and the nesting operation on cycles is equivalent to concatenation. The equality in Eq.~\eqref{eqn:StructureOfQKL:concat} thus requires that the number of terms on each side is the same: thus $s_i = 1$. Further, since concatenation is non-commutative, the cycles must match term-wise: i.e.~$c_{i,1} = \reduc{\dsig;l+1}[c_{i,1}]$ and $c_{i,j} = \reduc{\dsig;0}[c_{i,j}]$ for $2 \le j \le N_i$. In other words, $c$ can only remain unmodified by $\reduc{\dsig;l}$ if for each of its vertices that have children, the first child remains unmodified by $\reduc{\dsig;l+1}$ and any subsequent children remain unmodified by $\reduc{\dsig;0}$. Then by the induction hypotheses $\mathbb{L}(m_i,l+1)$ and $\mathbb{L}(m_i,0)$, where $m_i < n$ is the number of vertices in $\G\backslash\alpha\cdots\mu_{i-1}$, we have $c_{i,1} \in Q^{\dsig;l+1}_{\G\backslash\alpha\cdots\mu_{i-1},\mu_i}$ and $c_{i,j} \in Q^{\dsig;0}_{\G\backslash\alpha\cdots\mu_{i-1},\mu_i}$ for $2 \le j \le N_i$. Since this argument holds for every $i \in \mathcal{I}$, $c$ has the form given in Eq.~\eqref{eqn:StructureOfQKL:FormBGroup}, and so is an element of $Q^{\dsig;l}_{\G;\alpha}$.
\end{itemize}

     \noindent\textit{Base case.} The base case of the induction is $\mathbb{L}(1,l)$ for $l = 0,1,\ldots,D$.
We prove this statement by considering each of the two possible graphs on a single vertex.
\begin{itemize}
         \item[a)] $\G$ has no loop. Then $\G$ contains no cycles, so there is no such cycle $c$ such that $c=\reduc{\dsig;l}[c]$, and $\mathbb{L}(1,l)$ holds vacuously.
         \item[b)] $\G$ has a loop $\alpha\alpha$ of length $L = 1$. Then we have one of two cases:
\begin{itemize}
           \item[i)] $l = 0,1,\ldots,D-1$. Then $L \le k_l$ since $k_l \ge 1$, and $\reduc{\dsig;l}[\alpha\alpha] = (\alpha)$. Therefore there is no cycle on $\G$ that satisfies $c=\reduc{\dsig;l}[c]$, and $\mathbb{L}(1,l)$ is vacuously true.
           \item[ii)] $l = D$. Then $L > k_D$ since $k_D = 0$, and $\alpha\alpha$ satisfies $\reduc{\dsig;D}[\alpha\alpha] = \alpha\alpha$. Further, we have that $Q^{\dsig;D}_{\G;\alpha} = \{\alpha\alpha\}$. Thus $\alpha\alpha \in Q^{\dsig;D}_{\G;\alpha}$, and $\mathbb{L}(1,D)$ is true.
\end{itemize}
\end{itemize}
\end{proof}

In this section we gave an explicit definition for $Q^{\dsig;l}_{\G;\alpha}$, the set of all $(\dsig,l)$-irreducible cycles off $\alpha$ on $\G$. This set plays a central role in the definition of the set of $\dsig$-irreducible walks, which we address in the following section.

\subsection{The set of $\dsig$-irreducible walks}
\label{subsec:StructureOfIK}
Recall that the set of $\dsig$-irreducible walks from $\alpha$ to $\omega$ on $\G$, denoted $I^\dsig_{\G;\alpha\omega}$, is the set of all walks on $\G$ from $\alpha$ to $\omega$ that satisfy $\Reduc{\dsig}[w] = w$, where $\Reduc{\dsig}$ is the walk reduction operator defined in Section \ref{sec:WalkReductionOperators}. In this section we give an explicit recursive form for $I^\dsig_{\G;\alpha\omega}$.

Since we have already developed an explicit expression for the set of $(\dsig,l)$-irreducible cycles in the previous section, the form of the $\dsig$-irreducible walks is relatively straightforward. Recall that by the definition of $\Reduc{\dsig}$ in Section \ref{sec:WalkReductionOperators}, applying $\Reduc{\dsig}$ to a walk $w$ leaves the base path $p$ of $w$ untouched and maps every immediate child cycle $c$ of $p$ to $\reduc{\dsig;0}[c]$. It follows that $w$ is left unmodified by $\Reduc{\dsig}$ if and only if every one of the child cycles $c$ is unmodified by $\reduc{\dsig;0}$ -- or in other words, if every child cycle is $(\dsig,0)$-irreducible. This observation is formalised and proven in the following theorem.

\begin{theorem}\label{thm:StructureOfIK}
Let $\G$ be a digraph and $\dsig$ be a dressing signature. Then the set of $\dsig$-irreducible walks from $\alpha$ to $\omega$ on $\G$ is given by
\begin{align}
I^\dsig_{\G;\alpha\omega} =
P_{\G;\alpha\omega}
\nest Q^{\dsig;0\:*}_{\G\backslash\alpha\cdots\mu_L;\omega}
\nest \cdots
\nest Q^{\dsig;0\:*}_{\G\backslash\alpha;\mu_2}
\nest Q^{\dsig;0\:*}_{\G;\alpha},\label{eqn:StructureOfIK}
\end{align}
       where $\alpha\mu_2\cdots\mu_L\omega \in P_{\G;\alpha\omega}$ is a simple path of length $L \ge 0$ from $\alpha$ to $\omega $ on $\G$, and $Q^{\dsig;0}_{\G;\alpha}$ is the set of all $(\dsig,l)$-irreducible cycles off $\alpha$ on $\G$, as defined in Theorem \ref{thm:StructureOfQKL}.
\end{theorem}
\begin{remark}
This result is universally valid in that for any chosen dressing signature $\dsig$, the set of $\dsig$-irreducible walks is given by Eq.~\eqref{eqn:StructureOfIK}. In view of this, it is worth investigating the predictions of Eq.~\eqref{eqn:StructureOfIK} in the limiting cases of $\dsig = [0]$ and $\dsig = \dsig_{\text{max}}$.

In the case of $\dsig = [0]$, no cycles are resummable and all cycles are $\dsig$-irreducible: thus $Q^{[0];0}_{\G;\alpha} = C_{\G;\alpha}$. Then the right-hand side of Eq.~\eqref{eqn:StructureOfIK} becomes  $P_{\G;\alpha\omega}$
$\nest C^{\,*}_{\G\backslash\alpha\cdots\mu_\ell;\omega}
\nest \cdots$
$\nest C^{\,*}_{\G\backslash\alpha;\mu_2}$
$\nest C^{\,*}_{\G;\alpha}$, which we recognize from Eq.~\eqref{eqn:WalkPrimeFactorisation} as the prime decomposition of $W_{\G;\alpha\omega}$. Thus Eq.~\eqref{eqn:StructureOfIK} reduces to $I^{[0]}_{\G;\alpha\omega} = W_{\G;\alpha\omega}$, reproducing the expected result that all walks on $\G$ are [0]-irreducible. Conversely, in the case of $\dsig = \dsig_{\,\text{max}}$, all cycles are resummable and no cycles are $\dsig$-irreducible, so that $Q^{\dsig_{\,\text{max}};0}_{\G;\alpha} = \emptyset$. Then Eq.~\eqref{eqn:StructureOfIK} gives $I^{\dsig_{\,\text{max}}}_{\G;\alpha\omega} = P_{\G;\alpha\omega}$. This is perfectly consistent with the observations that (i) the $\dsig_{\,\text{max}}$-irreducible walks are those that do not contain any resummable cycles, and (ii) all cycles are resummable with respect to $\dsig_{\,\text{max}}$.

For dressing signatures between $[0]$ and $\dsig_{\,\text{max}}$, the $\dsig$-irreducible walks interpolate smoothly between the set of all walks and the set of simple paths.
\end{remark}

\begin{proof}
Recall that a walk is $\dsig$-irreducible if and only if $\Reduc{\dsig}\big[w\big] = w$. Thus
Theorem \ref{thm:StructureOfIK} corresponds to the statement
\begin{align*}
w \in
P_{\G;\alpha\omega}
\nest Q^{\dsig;0\:*}_{\G\backslash\alpha\cdots\mu_L;\omega}
\nest \cdots
\nest Q^{\dsig;0\:*}_{\G\backslash\alpha;\mu_2}
\nest Q^{\dsig;0\:*}_{\G;\alpha}
\iff \Reduc{\dsig}\big[w\big] = w.
\end{align*}
We prove the forward and backward directions of this statement separately. In each direction the result follows straightforwardly from Theorem \ref{thm:StructureOfQKL}. Before we begin, we note that a walk is in $P_{\G;\alpha\omega}
\nest Q^{\dsig;0\:*}_{\G\backslash\alpha\cdots\mu_L;\omega}
\nest \cdots
\nest Q^{\dsig;0\:*}_{\G\backslash\alpha;\mu_2}
\nest Q^{\dsig;0\:*}_{\G;\alpha}$
if and only if its decomposition into a simple path plus a collection of cycles has the form
\begin{align}\label{eqn:StructureOfIK:wstruc}
\alpha\mu_2\cdots\mu_L\omega
\nest \left[\Nest_{j=1}^{N_{L+1}} q_{L+1,j} \right]
\nest \cdots
\nest \left[\Nest_{j=1}^{N_2} q_{2,j} \right]
\nest \left[\Nest_{j=1}^{N_1} q_{1,j} \right]
\end{align}
where $\alpha\mu_2\cdots\mu_L\omega$ is a simple path, $N_i \ge 0$, and $q_{i,j} \in Q^{\dsig;0}_{\G\backslash\alpha\cdots\mu_{i-1};\mu_i}$
for $1 \le i \le L+1$ and $1\le j \le N_i$.\\

\noindent\textbf{Forward.}
We wish to show that
\begin{align*}
w \in P_{\G;\alpha\omega}
\nest Q^{\dsig;0\:*}_{\G\backslash\alpha\cdots\mu_L;\omega}
\nest \cdots
\nest Q^{\dsig;0\:*}_{\G\backslash\alpha;\mu_2}
\nest Q^{\dsig;0\:*}_{\G;\alpha}
 \implies \Reduc{\dsig}[w] = w.
\end{align*}
Then $w$ has the form given by Eq.~\eqref{eqn:StructureOfIK:wstruc}, so that
\begin{align*}
\Reduc{\dsig}\big[w\big] &= \alpha\mu_2\cdots\mu_L\omega
\nest \left[\Nest_{j=1}^{N_{L+1}} \reduc{\dsig;0}\big[q_{L+1,j}\big] \right]
\nest \cdots
\nest \left[\Nest_{j=1}^{N_2} \reduc{\dsig;0}\big[q_{2,j}\big] \right]
\nest \left[\Nest_{j=1}^{N_1} \reduc{\dsig;0}\big[q_{1,j}\big] \right]\\
&= \alpha\mu_2\cdots\mu_L\omega
\nest \left[\Nest_{j=1}^{N_{L+1}} q_{L+1,j}\right]
\nest \cdots
\nest \left[\Nest_{j=1}^{N_2} q_{2,j}\right]
\nest \left[\Nest_{j=1}^{N_1} q_{1,j}\right]
\\
&= w,
\end{align*}
where the second equality follows on noting that each $q_{i,j}$ is $(\dsig,0)$-irreducible and so satisfies $q_{i,j} = \reduc{\dsig;0}[q_{i,j}]$.\\

\noindent\textbf{Backward.}
We wish to show that
\begin{align*}
\Reduc{\dsig}[w] = w \implies w \in P_{\G;\alpha\omega}
\nest Q^{\dsig;0\:*}_{\G\backslash\alpha\cdots\mu_L;\omega}
\nest \cdots
\nest Q^{\dsig;0\:*}_{\G\backslash\alpha;\mu_2}
\nest Q^{\dsig;0\:*}_{\G;\alpha}.
\end{align*}
We prove this by showing that any walk satisfying $w = \Reduc{\dsig}[w]$ must have the structure of Eq.~\eqref{eqn:StructureOfIK:wstruc}. Let $w$ have decomposition
\begin{align}
w = \alpha\mu_2\cdots\mu_L\omega
\nest \left[\Nest_{j=1}^{N_{L+1}} c_{L+1,j} \right]
\nest \left[\Nest_{j=1}^{N_2} c_{2,j} \right]
\nest \cdots
\nest \left[\Nest_{j=1}^{N_1} c_{1,j} \right],
\label{eqn:StructureOfIK:backwards:wdecomp}
\end{align}
where $N_i \ge 0$ is the number of child cycles off $\mu_i$ for $1\le i \le L+1$, and $c_{i,j}\in C_{\G\backslash\alpha\cdots\mu_{i-1};\mu_i}$ and $1 \le j \le N_i$. Then
\begin{align*}
\Reduc{\dsig}\big[w\big] = \alpha\mu_2\cdots\mu_L\omega
\nest \left[\Nest_{j=1}^{N_{L+1}} \reduc{\dsig;0}\big[c_{L+1,j}\big] \right]
\nest \cdots
\nest \left[\Nest_{j=1}^{N_2} \reduc{\dsig;0}\big[c_{2,j}\big] \right]
\nest \left[\Nest_{j=1}^{N_1} \reduc{\dsig;0}\big[c_{1,j}\big] \right].
\end{align*}
Comparing this expression with Eq.~\eqref{eqn:StructureOfIK:backwards:wdecomp}, it follows from the properties of the nesting operation between cycles that $\Reduc{\dsig}\big[w\big] = w$ only if $\reduc{\dsig;0}\big[c_{i,j}\big] = c_{i,j}$ for all $1 \le i \le L+1$ and $1 \le j \le N_i$. Then from Theorem \ref{thm:StructureOfQKL} with $l = 0$ we have that $c_{i,j} \in Q^{\dsig;0}_{\G\backslash\alpha\cdots\mu_{i-1};\mu_i}$. The decomposition of $w$ thus matches the structure described in Eq.~\eqref{eqn:StructureOfIK:wstruc}, so that $w$ is an element of $P_{\G;\alpha\omega}
\nest Q^{\dsig;0\:*}_{\G\backslash\alpha\cdots\mu_L;\omega}
\nest \cdots
\nest Q^{\dsig;0\:*}_{\G\backslash\alpha;\mu_2}
\nest Q^{\dsig;0\:*}_{\G;\alpha}$ as required.
\end{proof}

\subsection{The cycle dressing operators}\label{subsec:CycleDressingOperators}
We now turn to developing an explicit definition for the walk dressing operator $\Dress{\dsig;l}{\G}$. We begin by defining the cycle dressing operators $\dress{\dsig;l}{\G}$. The key characteristic of the cycle dressing operators is that $\dress{\dsig;l}{\G}$ maps a $(\dsig,l)$-irreducible cycle $q$ to the set of cycles formed by adding zero or more resummable cycles to $q$. Consequently, $\dress{\dsig;l}{\G}[q]$ is the set of all cycles that satisfy $\reduc{\dsig;l}[c] = q$. We prove this statement in Section \ref{subsec:rdeltaInverses}.

\begin{definition}[The cycle dressing operators]\label{defn:CycleDressingOperators}
Let $\G$ be a digraph, $\dsig$ be a dressing signature of depth $D$, and $0 \le l \le D$ be a parameter. Then the cycle dressing operator $\dress{\dsig;l}{\G}$ is a map $\dress{\dsig;l}{\G} : Q^{\dsig;l}_{\G} \rightarrow 2^{\,C_\G}$ defined as follows. Let $q$ be a $(\dsig,l)$-irreducible cycle on $\G$ with decomposition
\begin{align*}
q = \alpha\mu_2\ldots\mu_L\alpha
\nest \left[\Nest_{j=1}^{N_L}c_{L,j}\right]
\nest \cdots
\nest \left[\Nest_{j=1}^{N_2}c_{2,j}\right],
\end{align*}
where $\alpha\mu_2\cdots\mu_L\alpha \in S_\G$ is a simple cycle of length $L \ge 1$ and $N_i \ge 0$ is the number of child cycles off $\mu_i$. Then $\dress{\dsig;l}{\G}$ maps $q$ to the set
\begin{align*}
\dress{\dsig;l}{\G}\big[q\big] = \left\{\alpha\mu_2\cdots\mu_L\alpha\right\}
             &\nest \Nest_{j=1}^{N_L} \left(C^{[k_{l_{L,j}},\ldots,k_{D-1},0]\:*}_{\G\backslash\alpha\cdots\mu_{L-1};\mu_L} \nest\dress{\dsig;l_{L,j}}{\G\backslash\alpha\cdots\mu_{L-1}}\big[c_{L,j}\big]
\right)
\nest D_{\G\backslash\alpha\cdots\mu_{L-1};\mu_L}^*\\
&\nest \cdots \\
         &\nest \Nest_{j=1}^{N_2} \left(C^{[k_{l_{2,j}},\ldots,k_{D-1},0]\:*}_{\G\backslash\alpha;\mu_2} \nest\dress{\dsig;l_{2,j}}{\G\backslash\alpha}\big[c_{2,j}\big]
\right)
\nest D_{\G\backslash\alpha;\mu_2}^*
\end{align*}
where
\begin{align*}
D_{\G\backslash\alpha\cdots\mu_{i-1};\mu_i} =
\begin{cases}
         C^{[k_{l+1},\ldots,k_{D-1},0]}_{\G\backslash{\alpha\cdots\mu_{i-1};\mu_i}} & \text{if $L \le k_l$ and $N_i = 0$,}\\
C^{\dsig}_{\G\backslash{\alpha\cdots\mu_{i-1};\mu_i}} & \text{otherwise,}
\end{cases}
\end{align*}
and the parameters $l_{i,j}$ for $2\le i \le L$ and $1 \le j \le N_i$ are given by
\begin{align*}
l_{i,j} = \begin{cases}
l+1& \text{if $L\le k_l$ and $j = 1$,}\\
0& \text{otherwise.}
\end{cases}
\end{align*}
\end{definition}

\begin{figure}\label{fig:cycle_dressing_example}
  \centering
  \includegraphics[scale=1.25]{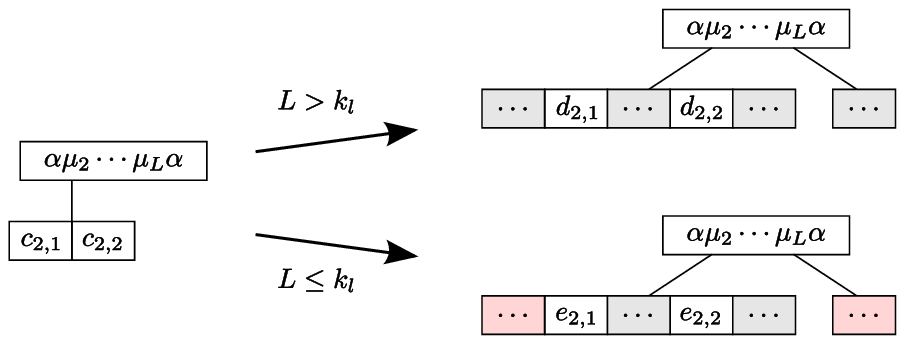}
  \caption{The action of $\dress{\dsig;l}{\G}$ on a cycle $q$ (see text for details). (Left) the cycle $q$. The vertex $\mu_L$ has no child cycles, and serves as an example of the case where $N_i = 0$. (Right) An example of the structure of an element of $\dress{\dsig;l}{\G}[q]$ in the two possible cases of $L > k_l$ and $L \le k_l$. The nodes shaded grey (pink) represent sequences of zero or more $\dsig$-structured ($[k_{l+1},\ldots,k_{D-1},0]$-structured) cycles. Here $d_{2,1}$ and $d_{2,2}$ are elements of the sets $\dress{\dsig;0}{\G\backslash\alpha}[c_{2,1}]$ and $\dress{\dsig;0}{\G\backslash\alpha}[c_{2,2}]$ respectively, while $e_{2,1}$ is an element of $\dress{\dsig;l+1}{\G\backslash\alpha}[c_{2,1}]$ and $e_{2,2}$ is an element of $\dress{\dsig;0}{\G\backslash\alpha}[c_{2,2}]$.}
\end{figure}

The action of $\dress{\dsig;l}{\G}$ on a $(\dsig,l)$-irreducible cycle $q$ depends on the length of the base cycle of $q$, and can be summarised as follows. If the base cycle of $q$ is longer than $k_l$, $\dress{\dsig;l}{\G}$ maps $q$ to the set of cycles formed from $q$ by (i) adding zero or more $\dsig$-structured cycles off every vertex $\mu_i$ before every child $c_{i,j}$ and after the final child $c_{i,N_i}$, and (ii) replacing every child $c_{i,j}$ by an element from the set $\dress{\dsig;0}{\G'}[c_{i,j}]$, where $\G'$ is the graph on which $c_{i,j}$ exists. On the other hand, if the base cycle of $q$ is not longer than $k_l$, $\dress{\dsig;l}{\G}$ maps $q$ to the set of cycles formed from $q$ by (i) adding zero or more $[k_{l+1},\ldots,k_{D-1},0]$-structured cycles to every internal vertex of $q$ that has no children, (ii) adding zero or more $[k_{l+1},\ldots,k_{D-1},0]$-structured cycles before the first child, zero or more $\dsig$-structured cycles before every subsequent child, and zero or more $\dsig$-structured cycles after the final child off every internal vertex of $q$ that has at least one child, (iii) replacing the first child $c_{i,1}$ by an element of $\dress{\dsig;l+1}{\G'}[c_{i,1}]$ and every subsequent child $c_{i,j}$ by an element of $\dress{\dsig;0}{\G'}[c_{i,j}]$. Figure \ref{fig:cycle_dressing_example} illustrates the action of $\dress{\dsig;l}{\G}$.

\subsection{The walk dressing operator}\label{sec:WalkDressingOperator}
     We now give the explicit definition for the walk dressing operator $\Dress{\dsig}{\G}$. The defining characteristic of $\Dress{\dsig}{\G}$ is that it maps a $\dsig$-irreducible walk $i$ to the set of all walks that can be formed by adding zero or more resummable cycles to $i$. Correspondingly, every walk $w$ in the set $\Dress{\dsig}{\G}[i]$ satisfies $\Reduc{\dsig}[w] = i$. We prove this statement in Section \ref{subsec:RDeltaInverses}.

\begin{definition}[The walk dressing operator]\label{defn:WalkDressingOperator}
     Let $\G$ be a digraph and $\dsig$ a dressing signature. Then the walk dressing operator $\Dress{\dsig}{\G}$ is a map $\Dress{\dsig}{\G} : I^\dsig_\G \rightarrow \powset{W_\G}$ defined as follows. Let $w$ be a $\dsig$-irreducible walk on $\G$ with decomposition
\begin{align*}
w = \alpha\mu_2\cdots\mu_L\omega
\nest \left[\Nest_{j=1}^{N_{L+1}}c_{L+1,j}\right]
\nest \cdots
\nest \left[\Nest_{j=1}^{N_2}c_{2,j}\right]
\nest \left[\Nest_{j=1}^{N_1}c_{1,j}\right],
\end{align*}
where $\alpha\mu_2\cdots\mu_L\omega \in P_\G$ is a simple path of length $L \ge 0$ on $\G$, $N_i \ge 0$ is the number of child cycles off $\mu_i$, and $c_{i,j}\in C_{\G\backslash\alpha\cdots\mu_{i-1};\mu_i}$ for $1 \le i \le L+1$ and $1 \le j \le N_i$. Then $\Dress{\dsig}{\G}$ maps $w$ to the set
\begin{align*}
\Dress{\dsig}{\G}\big[w\big] = \left\{\alpha\mu_2\cdots\mu_L\omega\right\}
&\nest \left[\Nest_{j=1}^{N_{L+1}}\left(
         C_{\G\backslash\alpha\cdots\mu_L;\omega}^{\dsig\:*} \nest\dress{\dsig;0}{\G\backslash\alpha\cdots\mu_L}\big[c_{L+1,j}\big]\right)\right] \nest C_{\G\backslash\alpha\cdots\mu_L;\omega}^{\dsig\:*}\\
&\nest \:\cdots \\
&\nest \left[\Nest_{j=1}^{N_1}\left(
         C_{\G;\alpha}^{\dsig\:*}\nest\dress{\dsig;0}{\G}\big[c_{1,j}\big]\right)\right] \nest C_{\G;\alpha}^{\dsig\:*},
\end{align*}
where $\dress{\dsig;0}{\G}$ is the cycle dressing operator of Definition \ref{defn:CycleDressingOperators}.
\end{definition}

\begin{figure}\label{fig:walk_dressing_example}
\includegraphics[scale=1.25]{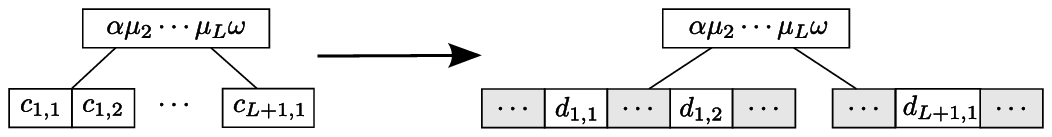}
  \caption{The action of the walk dressing operator $\Dress{\dsig}{\G}$ on a $\dsig$-irreducible walk $i$ (see text for details). (Left) the walk $i$. (Right) The general structure of the walks in the set $\Dress{\dsig}{\G}[i]$. The shaded nodes represent sequences of zero or more $\dsig$-structured cycles, while $d_{i,j}$ is an element from the set $\dress{\dsig;0}{\G\backslash\alpha\cdots\mu_{i-1}}[c_{i,j}]$.}
  \end{figure}

The walk dressing operator $\Dress{\dsig}{\G}$ maps a $\dsig$-irreducible walk $i$ to the set of walks obtained from $i$ by (i) adding zero or more $\dsig$-structured cycles before every child $c_{i,j}$ and after the final child $c_{i,N_i}$ off every vertex $\mu_i$, and (ii) replacing every child $c_{i,j}$ with an element from the set $\dress{\dsig;0}{\G'}[c_{i,j}]$, where $\G'$ is the graph on which $c_{i,j}$ exists. This action is illustrated in Figure \ref{fig:walk_dressing_example}.

\subsection{The cycle reduction and cycle dressing operators are inverses}
\label{subsec:rdeltaInverses}
In this section we prove that the cycle reduction and cycle dressing operators of Definitions \ref{defn:CycleReductionOperator} and \ref{defn:CycleDressingOperators} satisfy property P2 of Definition \ref{defn:reductiondressingAxioms}: namely, that for a given $(\dsig,l)$-irreducible cycle $q$, the set $\dress{\dsig;l}{\G}[q]$ contains all cycles, and only those cycles, that satisfy $\reduc{\dsig;l}[c] = q$.

\begin{lem}\label{lem:rdeltaInverses}
Let $\G$ be a digraph, $\dsig$ be a dressing signature with depth $D$, $c \in C_{\G;\alpha}$ be a cycle off $\alpha$ on $\G$, and $q \in Q^{\dsig;l}_{\G;\alpha}$ be a $(\dsig,l)$-irreducible cycle off $\alpha$ on $\G$ for some $0 \le l \le D$. Then $c$ is an element of $\dress{\dsig;l}{\G}[q]$ if and only if $\reduc{\dsig;l}[c] = q$.
\end{lem}
\begin{proof}
We wish to show that for a $(\dsig,l)$-irreducible cycle $q$ off $\alpha$ and an arbitrary cycle $c$ off $\alpha$, then
\begin{align}
c \in \dress{\dsig;l}{\G}[q] \iff \reduc{\dsig;l}[c] = q.\label{eqn:rdeltInverses:statement}
\end{align}
We prove the forward and backward directions of this statement separately. Although the proof in each direction is somewhat lengthy, the central idea is very simple. In the forward direction, we wish to show that if $c$ can be obtained from $q$ by adding a collection of resummable cycles off the internal vertices of $q$, then applying $\reduc{\dsig;l}$ to $c$ removes those same cycles (by mapping each to the corresponding trivial walk) so that $q$ is recovered. In the backward direction, we wish to show the opposite: if $q$ can be obtained from $c$ by deleting some child cycles from $c$, then applying $\dress{\dsig;l}{\G}$ to $q$ adds those cycles back, thereby recovering $c$.

The proof in each direction proceeds by induction on the number of vertices in $\G$. We begin with some general remarks that will be useful in both directions of the proof.\\

\noindent\textbf{Preliminary remarks.}
As noted in the proof of Theorem \ref{thm:StructureOfQKL}, there are two possible cases for an element $q$ of $Q^{\dsig;l}_{\G;\alpha}$, which are mutually exclusive.
\begin{itemize}
\item[1)] $q$ has decomposition given by Eq.~\eqref{eqn:StructureOfQKL:FormA1}: i.e.~the base cycle of $q$ is longer than $k_l$, and every child cycle is $(\dsig,0)$-irreducible. Then $\dress{\dsig;l}{\G}$ maps $q$ to the set
\begin{align*}
\dress{\dsig;l}{\G}\big[q\big] = \big\{\alpha\mu_2\cdots\mu_L\alpha\big\}
             &\nest \left[\Nest_{j=1}^{N_L}\left(C^{\dsig\:*}_{\G\backslash\alpha\cdots\mu_{L-1};\mu_L} \nest\dress{\dsig;0}{\G\backslash\alpha\cdots\mu_{L-1}}\big[q_{L,j}\big]\right)\right]
             \nest C^{\dsig\:*}_{\G\backslash\alpha\cdots\mu_{L-1};\mu_L}\\
&\nest \cdots \\
             &\nest \left[\Nest_{j=1}^{N_2}\left(C^{\dsig\:*}_{\G\backslash\alpha;\mu_2} \nest \dress{\dsig;0}{\G\backslash\alpha}\big[q_{2,j}\big]\nest \right)\right] \nest C_{\G\backslash\alpha;\mu_2}^{\dsig\:*}.
\end{align*}
This set consists of all the cycles formed from $q$ by adding zero or more $\dsig$-structured cycles to the vertex $\mu_i$ in front of every child $q_{i,j}$ and after the final child $q_{i,N_i}$, and additionally applying $\dress{\dsig;0}{\G\backslash\alpha\cdots\mu_{i-1}}$ to every child cycle $q_{i,j}$. A cycle is a member of this set if and only if its decomposition has the form
\begin{subequations}\label{eqn:Thm2:FormACDecompGroup}
\begin{equation}
\begin{split}
\alpha\mu_2\cdots\mu_L\alpha
         &\nest \Bigg[\Nest_{j=1}^{N_L}\bigg(\Nest_{m=1}^{P_{L,j}} f_{L,j,m} \nest d_{L,j} \bigg)\Bigg]\nest\left[\Nest_{m=1}^{P_{L,N_L+1}} f_{L,N_L+1,m} \right]\\[3mm]
&\nest \cdots \\[3mm]
         &\nest \Bigg[\Nest_{j=1}^{N_2}\bigg(\underbrace{\Nest_{m=1}^{P_{2,j}} f_{2,j,m}}_{\text{from $C^{\dsig\:*}_{\G\backslash\alpha;\mu_2}$}}
         \nest \underbrace{\vphantom{\Nest_{m=1}^{P_{2,j}}}d_{2,j}}_{\text{from $\dress{\dsig;0}{\G\backslash\alpha}[q_{2,j}]$}} \bigg)\Bigg]\nest\Bigg[\underbrace{\vphantom{\Nest_{m=1}^{P_{2,j}}}\Nest_{m=1}^{P_{2,N_2+1}} f_{2,N_2+1,m}}_{\text{from $C^{\dsig\:*}_{\G\backslash\alpha;\mu_2}$}} \Bigg].
\end{split}
\end{equation}
for some $P_{2,1},\ldots,P_{L,N_L+1} \ge 0$, with
\begin{equation}
\begin{aligned}
f_{i,j,m} \in C^{\dsig}_{\G\backslash\alpha\cdots\mu_{i-1};\mu_i} \quad\text{ and } \quad
       d_{i,j} \in \dress{\dsig;0}{\G\backslash\alpha\cdots\mu_{i-1}}\big[q_{i,j}\big] \label{eqn:Thm2:FormACDecompPart2}
\end{aligned}
\end{equation}
\end{subequations}
for indices in the relevant ranges. Here the cycles denoted $f_{i,j,m}$ are $\dsig$-structured cycles that have been added to the vertex $\mu_i$ in front of every child $q_{i,j}$ and after the final child $q_{i,N_i}$, and $d_{i,j}$ is an element from the set of cycles produced by applying $\dress{\dsig;0}{\G\backslash\alpha\cdots\mu_{i-1}}$ to $q_{i,j}$.\\

\item[2)] $q$ has a decomposition of the form given in Eq.~\eqref{eqn:StructureOfQKL:FormBGroup}. Then $\dress{\dsig;l}{\G}$ maps $q$ to the set
\begin{equation}\label{eqn:DressingResultCaseB}
\begin{aligned}
&\big\{\alpha\mu_2\cdots\mu_L\alpha\big\}\\
& \nest C^{[k_{l+1},\ldots,k_{d-1},0]\:*}_{\G\backslash\alpha\cdots\mu_{i_p-1};\mu_{i_p}}
\nest \dress{\dsig;l+1}{\G\backslash\alpha\cdots\mu_{i_p-1}}\big[q_{p,1}\big]
\nest \left[\Nest_{j=2}^{N_{p}}\left(C^{\dsig\:*}_{\G\backslash\alpha\cdots\mu_{i_p-1};\mu_{i_p}} \nest \dress{\dsig;0}{\G\backslash\alpha\cdots\mu_{i_p-1}}\big[q_{p,j}] \right) \right]
\nest C^{\dsig\:*}_{\G\backslash\alpha\cdots\mu_{i_p-1}}\\
&\nest\cdots \\
& \nest C^{[k_{l+1},\ldots,k_{d-1},0]\:*}_{\G\backslash\alpha\cdots\mu_{i_1-1};\mu_{i_1}}
\nest \dress{\dsig;l+1}{\G\backslash\alpha\cdots\mu_{i_1-1}}\big[q_{1,1}\big]
\nest \left[\Nest_{j=2}^{N_{1}}\left(C^{\dsig\:*}_{\G\backslash\alpha\cdots\mu_{i_1-1};\mu_{i_1}} \nest \dress{\dsig;0}{\G\backslash\alpha\cdots\mu_{i_1-1}}\big[q_{1,j}] \right) \right]
\nest C^{\dsig\:*}_{\G\backslash\alpha\cdots\mu_{i_1-1}}\\
           &\nest C_{\G\backslash\alpha\cdots\mu_{j_q-1};\mu_{j_q}}^{[k_{l+1},\ldots,k_{d-1},0]\:*}
           \nest \cdots
           \nest C_{\G\backslash\alpha\cdots\mu_{j_1-1};\mu_{j_1}}^{[k_{l+1},\ldots,k_{d-1},0]\:*},
\end{aligned}
\end{equation}
where $\{j_1,\ldots, j_q\} = \{2,\ldots, L\}\backslash\{i_1,\ldots,i_p\}$ index the internal vertices of $\alpha\mu_2\cdots\mu_L\alpha$ that have no child cycles.\footnote{Note that the indices $j_1,\ldots, j_q$ and $i_1,\ldots, i_p$ will generally be interleaved, and so the factors on the right-hand side of Eq.~\eqref{eqn:DressingResultCaseB} are not in the correct order as written: in order to respect the nesting conditions they must be nested into the cycle $\alpha\mu_2\cdots\mu_L\alpha$ in reverse order of root index, from $\mu_L$ down to $\mu_2$.} This set consists of all cycles formed from $q$ by nesting zero or more $[k_{l+1},\ldots,k_{D-1},0]$-structured cycles off every internal vertex of $q$ that has no children, and adding zero or more $[k_{l+1},\ldots,k_{D-1},0]$-structured cycles before the first child, zero or more $\dsig$-structured cycles before every subsequent child, and zero or more $\dsig$-structured cycles after the final child, off every internal vertex that has at least one child. Additionally, $\dress{\dsig;l+1}{\G}$ is applied to the first child off each vertex, and $\dress{\dsig;0}{\G}$ is applied to any subsequent children off the same vertex.

A cycle is an element of the set given in Eq.~\eqref{eqn:DressingResultCaseB} if and only if its decomposition has the form
\begin{subequations}\label{eqn:Thm2:FormBCDecompGroup}
\begin{equation}\label{eqn:Thm2:FormBCDecompPart1}
\begin{aligned}
\alpha\mu_2\cdots\mu_L\alpha
&\nest\left[\Nest_{m=1}^{P_{p,1}} f_{p,1,m}\right] \nest d_{p,1} \nest \left[\Nest_{j=2}^{N_p} \left(\Nest_{m=1}^{P_{p,j}} f_{p,j,m} \nest d_{p,j}\right)\right] \nest\left[\Nest_{m=1}^{P_{p,N_p+1}} f_{p,N_p+1,m}\right]\\[2mm]
&\nest \cdots\\[2mm]
&\nest\Bigg[\Nest_{m=1}^{P_{1,1}} f_{1,1,m}\Bigg] \nest d_{1,1} \nest \Bigg[\Nest_{j=2}^{N_1} \bigg(\Nest_{m=1}^{P_{1,j}} f_{1,j,m} \nest d_{1,j}\bigg)\Bigg] \nest\left[\Nest_{m=1}^{P_{1,N_1+1}} f_{1,N_1+1,m}\right]\\
&\nest \left[\Nest_{m=1}^{R_q} g_{q,m}\right] \nest \cdots \nest \left[\Nest_{m=1}^{R_1} g_{1,m}\right]
\end{aligned}
\end{equation}
for some $P_{1,1},\ldots,P_{p,N_p+1} \ge 0$ and $R_1,\ldots, R_q \ge 0$, with
\begin{equation}\label{eqn:Thm2:FormBCDecompPart2}
\begin{aligned}
f_{k,1,m} &\in C^{[k_{l+1},\ldots,k_{D-1},0]}_{\G\backslash\alpha\cdots\mu_{i_k-1};\mu_{i_k}}
\qquad&d_{k,1} &\in \dress{\dsig;l+1}{\G\backslash\alpha\cdots\mu_{i_k-1};\mu_{i_k}}[q_{k,1}],\\
f_{k,j,m} &\in C^{\dsig}_{\G\backslash\alpha\cdots\mu_{i_k-1};\mu_{i_k}}
\qquad&d_{k,j} &\in \dress{\dsig;0}{\G\backslash\alpha\cdots\mu_{i_k-1};\mu_{i_k}}[q_{k,j}],\\
g_{k,m} &\in C^{[k_{l+1},\ldots,k_{D-1},0]}_{\G\backslash\alpha\cdots\mu_{j_k-1};\mu_{j_k}}
\end{aligned}
\end{equation}
\end{subequations}
for indices in the relevant ranges. Here the cycles denoted $f_{k,1,m}$ are the $[k_{l+1},\ldots,k_{D-1},0]$-structured cycles that have been added to the vertex $\mu_{i_k}$ before its first child, while $f_{k,j,m}$ for $j > 1$ are the $\dsig$-structured cycles that have been added to $\mu_{i_k}$ before every subsequent child and after its final child. The cycles denoted $g_{k,m}$ are the $[k_{l+1},\ldots,k_{D-1},0]$-structured cycles that have been attached to the vertex $\mu_{j_k}$ (which had no children in the original cycle). Finally, $d_{k,1}$ is an element of the cycle set produced by applying $\dress{\dsig;l+1}{\G\backslash\alpha\cdots\mu_{i_k-1}}$ to $q_{k,1}$, and $d_{k,j}$ for $j > 1$ is an element of the cycle set produced by applying $\dress{\dsig;0}{\G\backslash\alpha\cdots\mu_{i_k-1}}$ to $q_{k,j}$.
\end{itemize}

We now begin the body of the proof. The proof relies heavily on Proposition \ref{prop:TrivialEquivalentCycles}, which states that $\reduc{\dsig;l}$ maps any $[k_{l},\ldots,k_{D-1},0]$-structured cycle to the corresponding trivial walk. \\

\noindent\textbf{Forward.}
     Let $n \ge 1$ be the number of vertices in $\G$, and $\mathbb{L}(n,l)$ for $n \ge 1$ and $0\le l \le D$ be the statement that for $q \in Q^{\dsig;l}_{\G;\alpha}$,
\begin{align*}
    c \in \dress{\dsig;l}{\G}\big[q\big] \implies \reduc{\dsig;l}\big[c\big] = q.
\end{align*}
We prove this statement for all $n$ and $l$ by induction on $n$. The
The proof proceeds by induction on $n$. The dependency structure of the induction is identical to that seen previously: that is, $\mathbb{L}(n,l)$ for $0 \le l \le D-1$ requires $\mathbb{L}(m,l+1)$ and $\mathbb{L}(m,D)$ for all $m < n$, while $\mathbb{L}(m,D)$ only requires $\mathbb{L}(m,0)$ for all $m < n$.\\

\noindent \textit{Inductive step.} The inductive step corresponds to the statements
\begin{align*}
\mathbb{L}(m,0) \text{ for all $m < n$ } &\implies \mathbb{L}(n,D),\\
         \mathbb{L}(m,l+1) \text{ and } \mathbb{L}(m,0) \text{ for all $m < n$ } &\implies \mathbb{L}(n,l) \text{ for $0 \le l \le D-1$.}
\end{align*}
     Given that $c$ is an element of $\dress{\dsig;l}{\G}\big[q\big]$, it must either have the structure given in Eq.~\eqref{eqn:Thm2:FormACDecompGroup} or that given in Eq.~\eqref{eqn:Thm2:FormBCDecompGroup}. In each case we show that applying $\reduc{\dsig;l}$ to $c$ recovers $q$, by explicitly evaluating $\reduc{\dsig;l}[c]$.
\begin{itemize}
       \item[a)] The base cycle of $c$ is longer than $k_l$. Then $c$ has the form of Eq.~\eqref{eqn:Thm2:FormACDecompGroup}. Note that this is necessarily the case if $l = D$. In this case when $\reduc{\dsig;l}$ is applied to $c$, $\reduc{\dsig;0}$ is applied to every child cycle $c_{i,j}$. Thus we have (cf.~Eq.~\eqref{eqn:Thm2:FormACDecompGroup}):
\begin{align*}
\reduc{\dsig;l}\big[c\big] = \alpha\mu_2\cdots\mu_L\alpha
         &\nest \Nest_{j=1}^{N_L}\bigg(\Nest_{m=1}^{P_{L,j}} \reduc{\dsig;0}\big[f_{L,j,m}\big]
         \nest \reduc{\dsig;0}\big[d_{L,j}\big] \bigg)
         \nest\left[\Nest_{m=1}^{P_{L,N_L+1}} \reduc{\dsig;0}\big[f_{L,N_L+1,m}\big] \right]\\[3mm]
&\nest \cdots \\[3mm]
         &\nest \Nest_{j=1}^{N_2}\bigg(\Nest_{m=1}^{P_{2,j}} \reduc{\dsig;0}\big[f_{2,j,m}\big]
         \nest \reduc{\dsig;0}\big[d_{2,j}\big] \bigg)
         \nest\Bigg[\Nest_{m=1}^{P_{2,N_2+1}} \reduc{\dsig;0}\big[f_{2,N_2+1,m}\big] \Bigg].
\end{align*}
Recall that all of the cycles denoted $f_{i,j,m}$ are $\dsig$-structured, while $d_{i,j}$ is an element of the set $\dress{\dsig;0}{\G\backslash\alpha\cdots\mu_{i-1}}[q_{i,j}]$ (see Eq.~\eqref{eqn:Thm2:FormACDecompPart2}). It follows from Proposition \ref{prop:TrivialEquivalentCycles} that every cycle $f_{i,j,m}$ is mapped by $\reduc{\dsig;0}$ to the corresponding trivial walk $(\mu_i)$. Since the trivial walks do not contribute to the nesting product, each of these factors can be deleted. Thus we find
\begin{align*}
       \reduc{\dsig;l}\big[c\big] &= \alpha\mu_2\cdots\mu_L\alpha
       \nest \left[\Nest_{j=1}^{N_L} \reduc{\dsig;0}\big[d_{L,j}\big]\right]
       \nest \cdots
       \nest \left[\Nest_{j=1}^{N_2} \reduc{\dsig;0}\big[d_{2,j}\big]\right]\\
       &=\alpha\mu_2\cdots\mu_L\alpha
       \nest \left[\Nest_{j=1}^{N_L} q_{L,j}\right]
       \nest \cdots
       \nest \left[\Nest_{j=1}^{N_2} q_{2,j}\right]
       \\
& = q,
\end{align*}
     where the second equality follows from the induction hypotheses $\mathbb{L}(m_i,0)$ for $2\le i \le L$, where $m_i < n$ is the number of vertices in $\G\backslash\alpha\cdots\mu_{i-1}$ (i.e.~the graph on which $d_{i,j}$ exists). \\

     \item[b)] The base cycle of $c$ is not longer than $k_l$. Then $c$ has the form of Eq.~\eqref{eqn:Thm2:FormBCDecompGroup}. In this case the result of applying $\reduc{\dsig;l}$ to $c$ depends on the values of the indices $s_2,\ldots,s_L$, where $s_i$ is the index of the first child off $\mu_i$ that is not $[k_{l+1},\ldots,k_{D-1},0]$-structured. Then for each vertex $\mu_i$, $\reduc{\dsig;l+1}$ is applied to the first $\min(s_i,N_i)$ children, and $\reduc{\dsig;0}$ is applied to any remaining children.

         To proceed, we consider the two sets of vertices $\mu_{j_1},\ldots,\mu_{j_q}$ and $\mu_{i_1},\ldots,\mu_{i_p}$ separately. Consider a representative vertex from the first set: say $\mu_{j_1}$. All of the $R_1$ children off $\mu_{j_1}$ (i.e.~those denoted $g_{1,m}$ in Eq.~\eqref{eqn:Thm2:FormBCDecompPart1}) are $[k_{l+1},\ldots,k_{D-1},0]$-structured. It follows that $s_{j_1} = R_1 + 1$, so that $\reduc{\dsig;l+1}$ is applied to all $R_1$ children off $\mu_{j_1}$. Since all of these children are $[k_{l+1},\ldots,k_{D-1},0]$-structured, $\reduc{\dsig;l+1}$ maps each of them to the corresponding trivial walk $(\mu_{j_1})$.  Applying $\reduc{\dsig;l}$ to $c$ thus removes all of the children $g_{1,1}, \ldots, g_{1,R_1}$. An identical argument holds for the children off the vertices $\mu_{j_2},\ldots,\mu_{j_q}$.

         Next consider a representative vertex from the second set: say $\mu_{i_1}$. Then the first $P_{1,1}$ children off $\mu_{i_1}$ (i.e.~those denoted $f_{1,1,m}$ in Eq.~\eqref{eqn:Thm2:FormBCDecompPart1}) are all $[k_{l+1},\ldots,k_{D-1},0]$-structured. Thus $s_{i_1}$ cannot be smaller than $P_{1,1} + 1$. We now claim that $s_{i_1}$ is equal to $P_{1,1} + 1$: that is, $d_{1,1}$ is not $[k_{l+1},\ldots,k_{D-1},0]$-structured. The proof of this is straightforward: recall from Eq.~\eqref{eqn:Thm2:FormBCDecompPart2} that $d_{1,1}$ is an element of $\dress{\dsig;l+1}{\G\backslash\alpha\cdots\mu_{i_1-1}}[q_{k,1}]$. Thus it follows from the induction hypothesis $\mathbb{L}(m_1,l+1)$, where $m_1 < n$ is the number of vertices in $\G\backslash\alpha\cdots\mu_{i_1-1}$, that $\reduc{\dsig;l+1}[d_{1,1}] = q_{k,1}$. Because $q_{k,1}$ is a cycle and $(\mu_{i_1})$ is not a cycle, then $q_{k,1} \neq (\mu_{i_1})$. Hence $\reduc{\dsig;l+1}[d_{1,1}] \neq (\mu_{i_1})$, which proves the claim that $s_{i_1} = P_{1,1} + 1$. Thus $\reduc{\dsig;l+1}$ is applied to the first $P_{1,1} + 1$ children off $\mu_{i_1}$: that is, those up to and including $d_{1,1}$. Since the first $P_{1,1}$ cycles (i.e.~those denoted $f_{1,1,m}$) are $[k_{l+1},\ldots,k_{D-1},0]$-structured, $\reduc{\dsig;l+1}$ maps each of them to $(\mu_{i_1})$. Further, $\reduc{\dsig;0}$ is applied to all subsequent children. Since the child cycles that are a result of dressing (i.e.~those denoted $f_{1,j,m}$ for $2 \le j \le N_1+1$) are all $\dsig$-structured, $\reduc{\dsig;0}$ maps each of them to $(\mu_{i_1})$. Applying $\reduc{\dsig;l}$ to $c$ thus removes all of the children $f_{1,1,1},\ldots,f_{1,N_1+1,P_{1,N_1+1}}$. An identical argument holds for the vertices $\mu_{i_2},\ldots, \mu_{i_p}$.Thus we have
     \begin{align*}
        \reduc{\dsig;l}\big[c\big] &= \alpha\mu_2\cdots\mu_L\alpha
         \nest \reduc{\dsig;l+1}\big[d_{p,1}\big] \nest \left[ \Nest_{j=2}^{N_p} \reduc{\dsig;0}\big[d_{p,j}\big] \right]
         \nest \cdots
         \nest \reduc{\dsig;l+1}\big[d_{1,1}\big] \nest \left[ \Nest_{j=2}^{N_1} \reduc{\dsig;0}\big[d_{1,j}\big] \right]\\
        &= \alpha\mu_2\cdots\mu_L\alpha
         \nest q_{p,1}\nest \left[ \Nest_{j=2}^{N_p} q_{p,j} \right]
         \nest \cdots
         \nest q_{1,1}\nest \left[ \Nest_{j=2}^{N_1} q_{1,j}\right]\\
         &= q,
    \end{align*}
    where in writing the first equality we have deleted all of the cycles that are mapped to trivial walks, as discussed above, and the second equality follows from the induction hypotheses $\mathbb{L}(m_k,0)$ for $1 \le k \le p$, where $m_k$ is the number of vertices in $\G\backslash\alpha\cdots\mu_{i_k-1}$.\\
\end{itemize}

     \noindent\textit{Base case.} The base case of the induction is the statement $\mathbb{L}(1,l)$ for $l = 0,1,\ldots,D$. We prove this statement by considering the two possible graphs on a single vertex.
\begin{itemize}
         \item[a)] $\G$ has no loop. Then $\G$ contains no cycles, and
         $Q^{\dsig;l}_{\G;\alpha}$ is empty for all $l = 0,1,\ldots, D$. Thus $\mathbb{L}(1,l)$ is vacuously true.
\item[b)] $\G$ has a loop $\alpha\alpha$. Then two possibilities exist:
\begin{itemize}
           \item[i.] $ 0 \le l \le D-1$. Then $Q^{\dsig;l}_{\G;\alpha}$ is empty, and $\mathbb{L}(1,l)$ is vacuously true.
           \item[ii.] $l = D$. Then $Q^{\dsig;D}_{\G;\alpha} = \{\alpha\alpha\}$. Setting $q = \alpha\alpha$ we find that $\dress{\dsig;D}{\G}[q] = \dress{\dsig;D}{\G}[\alpha\alpha] =\{\alpha\alpha\}$. Setting $c = \alpha\alpha$ yields $\reduc{\dsig;D}[c] = \reduc{\dsig;D}[\alpha\alpha] = \alpha\alpha = q$, and $\mathbb{L}(1,D)$ holds.\\
\end{itemize}
\end{itemize}

\noindent\textbf{Backward.} Let $n \ge 1$ be the number of vertices in $\G$, and
     $\mathbb{L}(n,l)$ for $n\ge 1$ and $0 \le l \le D$ be the statement that for $q \in Q^{\dsig;l}_{\G;\alpha}$ and an arbitrary cycle $c$,
     \begin{align*}
        \reduc{\dsig;l}\big[c\big] = q \implies c \in \dress{\dsig;l}{\G}\big[q\big].
    \end{align*}
    The proof again proceeds by induction on $n$, with the same dependency structure as in the forwards direction. Throughout this section we assume that $c$ has decomposition
\begin{align}
     \alpha\beta_2\cdots\beta_P\alpha
     \nest \left[\Nest_{j=1}^{M_P} c_{P,j}\right]
     \nest \cdots
    \nest \left[\Nest_{j=1}^{M_2} c_{2,j}\right]\label{Thm2:Backward:CDecomp}
\end{align}
     for some $P \ge 1$ and $M_2,\ldots, M_P \ge 0$, where $c_{ij} \in C_{\G\backslash\alpha\cdots\beta_{i-1};\beta_i}$ for $2\le i \le P$ and $1\le j \le M_i$. We then show that the equality $\reduc{\dsig;l}[c] = q$ implies that the decomposition of $c$ matches the form of the elements of $\dress{\dsig;l}{\G}[q]$ (i.e.~either Eq.~\eqref{eqn:Thm2:FormACDecompGroup} or \eqref{eqn:Thm2:FormBCDecompGroup}) so that $c$ is an element of $\dress{\dsig;l}{\G}[q]$.
\\

\noindent\textit{Inductive step.}
The inductive step corresponds to the statements
\begin{align*}
\mathbb{L}(m,0) \text{ for all $m < n$ } &\implies \mathbb{L}(n,D),\\
\mathbb{L}(m,l+1) \text{ and } \mathbb{L}(m,0) \text{ for all $m < n$ } &\implies \mathbb{L}(n,l) \text{ for $0 \le l \le D-1$.}
\end{align*}
     Since the base cycle of $\reduc{\dsig;l}[c]$ is the same as that of $c$, it follows from the equality $\reduc{\dsig;l}[c] = q$ that the base cycle of $c$ is the same as the base cycle of $q$. Thus we write $\alpha\mu_2\cdots\mu_L\alpha$ instead of $\alpha\beta_2\cdots\beta_P\alpha$. As before, there are two possibilities:
\begin{enumerate}
\item[1)] The base cycle of $c$ is longer than $k_l$. Then when $\reduc{\dsig;l}$ is applied to $c$, $\reduc{\dsig;0}$ is applied to every child $c_{i,j}$, so that
\begin{align*}
\reduc{\dsig;l}\big[c\big] = \alpha\mu_2\cdots\mu_L\alpha
\nest \left[\Nest_{j=1}^{M_L}\reduc{\dsig;0}\big[c_{L,j}\big]\right]
\nest \cdots
\nest \left[\Nest_{j=1}^{M_2}\reduc{\dsig;0}\big[c_{2,j}\big]\right].
\end{align*}
Meanwhile, $q$ has the structure given in Eq.~\eqref{eqn:StructureOfQKL:FormA1}. The equality $\reduc{\dsig;l}\big[c\big] = q$ thus requires that $\Nest_{j=1}^{M_i}\reduc{\dsig;0}\big[c_{i,j}\big] = \Nest_{j=1}^{N_i}q_{i,j}$ for $2 \le i \le L$. Because some of the factors $\reduc{\dsig;0}\big[c_{ij}\big]$ may be trivial, it does not directly follow that $M_i = N_i$, but only that $M_i \ge N_i$. Specifically, we conclude that among the children $c_{i,1},\ldots,c_{i,M_i}$ there are $N_i$ that are mapped by $\reduc{\dsig;0}$ to $q_{i,1},\ldots, q_{i,N_i}$. Let $1 \le I_{i,1} < I_{i,2} < \cdots < I_{i,N_i} \le M_i$ be the indices of those cycles, so that $\reduc{\dsig;0}\big[c_{i,I_{i,j}}\big] = q_{i,j}$, and define $I_{i,0} = 0$. Then the sequence of $M_i$ cycles can be rewritten as
\begin{align*}
\Nest_{j=1}^{M_i}\reduc{\dsig;0}\big[c_{i,j}\big] =
\Nest_{j=1}^{N_i}\left[\Nest_{m=I_{i,j-1}+1}^{I_{i,j}-1} \reduc{\dsig;0}\big[c_{i,m}\big] \nest \reduc{\dsig;0}\big[c_{i,I_{i,j}}\big]\right]
\nest \Nest_{m = I_{i,N_i}+1}^{M_i} \reduc{\dsig;0}\big[c_{i,m}\big]
\end{align*}
where
\begin{align*}
\reduc{\dsig;0}\big[c_{ij}\big] =
\begin{cases}
q_{i,p} & \text{if $j = I_{i,p}$ for some $1 \le p \le N_i$,}\\
(\mu_i) & \text{otherwise.}
\end{cases}
\end{align*}
       Now from the induction hypothesis $\mathbb{L}(m_i,0)$, where $m_i < n$ is the number of vertices in $\G\backslash\alpha\cdots\mu_{i-1}$, it follows that $c_{i,I_{i,p}} \in \dress{\dsig;0}{\G\backslash\alpha\cdots\mu_{i-1}}\big[q_{i,p}\big]$ for $1 \le p \le N_i$, while from Proposition \ref{prop:TrivialEquivalentCycles} it follows that $c_{i,j} \in C^{\dsig}_{\G\backslash\alpha\cdots\mu_{i-1};\mu_i}$ if $j \notin \{I_{i,1},\ldots, I_{i,N_i}\}$. The decomposition of $c$ thus has the form of Eq.~\eqref{eqn:Thm2:FormACDecompGroup}, and so $c \in \dress{\dsig;l}{\G}\big[q\big]$.\\

\item[2)] The base cycle of $c$ is not longer than $k_l$. In this case $\reduc{\dsig;l}[c]$ depends on the values of $s_2,\ldots,s_L$, where $s_i$ is the index of the first child off $\mu_i$ that is not $[k_{l+1},\ldots,k_{D-1},0]$-structured. Meanwhile, $q$ has the structure of Eq.~\eqref{eqn:StructureOfQKL:FormBGroup}. In order to proceed, we note that the values that each $s_i$ can take are constrained by the equality $\reduc{\dsig;l}[c] = q$. Specifically, note that the cycle $\reduc{\dsig;l}[c]$ has a maximum of $N_i - s_i + 1$ children off $\mu_i$. (This follows from the observation that the first $s_i - 1$ child cycles are mapped by $\reduc{\dsig;l+1}$ to $(\mu_i)$ and disappear, by the definition of $s_i$.) Since $q$ has no cycles off (for example) $\mu_{j_1}$, we conclude that $s_{j_1} = S_{j_1} + 1$, or in other words, all of the child cycles $c_{j_1,1},\ldots,c_{j_1,S_{j_1}}$ off $\mu_{j_1}$ in $c$ (see Eq.~\eqref{Thm2:Backward:CDecomp}) are $[k_{l+1},\ldots,k_{D-1},0]$-structured. An identical argument holds for the vertices $\mu_{j_2},\ldots, \mu_{j_q}$.

Conversely, we conclude that since $q$ does have children off (for example) $\mu_{i_1}$, then $s_{i_1} \neq N_{i_1} + 1$. Then $\reduc{\dsig;l+1}$ is applied to the first $s_{i_1}$ children off $\mu_{i_1}$ (thereby mapping the first $s_{i_1}-1$ of them to $(\mu_{i_1})$), and $\reduc{\dsig;0}$ is applied to any subsequent children. It follows that the first $s_{i_1} -1$ children off $\mu_{i_1}$ are $[k_{l+1},\ldots,k_{D-1},0]$-structured. An identical argument holds for each of the vertices $\mu_{i_2},\ldots,\mu_{i_p}$. We now consider the structure of the remaining children off $\mu_{i_k}$, i.e.~those with indices $s_{i_k}$ through to $M_{i_k}$.\\

Since all of the children off $\mu_{j_1},\ldots,\mu_{j_q}$ and the first $s_{i_k}-1$ children off $\mu_{i_k}$ for $1 \le k \le p$ are each mapped to the corresponding trivial walk, we have
\begin{align*}
\reduc{\dsig;l}\big[c\big] = \alpha\mu_2\cdots\mu_L\alpha
    &\nest\reduc{\dsig;l+1}\big[c_{i_p,s_{i_p}}\big] \nest
    \Nest_{j=s_{i_p}+1}^{M_{i_p}} \reduc{\dsig;0}\big[c_{i_p,j}\big]\\
    &\nest \cdots \\
    &\nest\reduc{\dsig;l+1}\big[c_{i_1,s_{i_1}}\big] \nest
    \Nest_{j=s_{i_1}+1}^{M_{i_1}} \reduc{\dsig;0}\big[c_{i_1,j}\big],
\end{align*}
where we have already deleted all of the cycles that reduce to trivial walks. We now compare this expression with the decomposition of $q$ given in Eq.~\eqref{eqn:StructureOfQKL:FormBGroup}. The equality $\reduc{\dsig;l+1}[c] = q$ implies that the cycles off each vertex must match. Consider the vertex $\mu_{i_1}$ as an example. Then
\begin{equation*}
\reduc{\dsig;l+1}\big[c_{i_1,s_{i_1}}\big] \nest
    \Nest_{j=s_{i_1}+1}^{S_{i_1}} \reduc{\dsig;0}\big[c_{i_1,j}\big] = \Nest_{j=1}^{N_1} q_{1,j}.
\end{equation*}
For this to hold, it must be the case that $\reduc{\dsig;l+1}\big[c_{i_1,s_{i_1}}\big] = q_{1,1}$. Then it follows from the induction hypothesis $\mathbb{L}(m_1,l+1)$ (where $m_1 < n$ is the number of vertices in $\G\backslash\alpha\cdots\mu_{i_1}$) that $c_{i_1,s_{i_1}} \in \dress{\dsig;l+1}{\G\backslash\alpha\cdots\mu_{i_1-1}}\big[q_{1,1}\big]$. We further conclude that $N_1-1$ of the factors $\reduc{\dsig;0}[c_{i_1,s_{i_1}-1}]$,$\ldots$,$\reduc{\dsig;0}[c_{i_1,M_{i_1}}]$ are equal to $q_{1,2},\ldots,q_{1,N_1}$, and the remaining factors are equal to the trivial walk $(\mu_{i_1})$.
Let $s_{i_1}+1 \le I_{1,2} < I_{1,3} < \cdots < I_{1,N_{1}} \le S_{i_1}$ be the indices of the cycles mapped to $q_{1,2},\ldots, q_{1,N_1}$, and define $I_{1,1} = s_{i_1}$. Then the sequence of the last $M_{i_1} - s_{i_1}$ child cycles off $\mu_{i_1}$ can be rewritten as
\begin{subequations}\label{eqn:Thm3:detailedStructureGroup}
\begin{equation}
  \Nest_{j=s_{i_1}+1}^{S_{i_1}} c_{i_1,j} =
   \Nest_{j=2}^{N_1}\left[\Nest_{m=I_{1,j-1}+1}^{I_{1,j}-1}\!\!c_{i_1,m}\:\nest \:c_{i_1,I_{1,j}}\right]
   \nest \Nest_{m=I_{1,N_1}+1}^{S_{i_1}} c_{i_1,m}
\end{equation}
where
\begin{align*}
\reduc{\dsig;0}\big[c_{i_1,j}\big] =
\begin{cases}
q_{1,k} & \text{if $j = I_{1,k}$ for some $2\le k \le N_1$,} \\
(\mu_{i_1}) & \text{if $j \notin \{ I_{1,2},\ldots,I_{1,N_1}\}$.}
\end{cases}
\end{align*}
\end{subequations}
     Then it follows from the induction hypothesis $\mathbb{L}(m_1,0)$, where $m_1$ is the number of vertices in $\G\backslash\alpha\cdots\mu_{i_1-1}$, that $c_{i_1,I_{1,k}} \in \dress{\dsig;0}{\G\backslash\alpha\cdots\mu_{i_1-1}}\big[q_{1,k}\big]$. Further, from Proposition \ref{prop:TrivialEquivalentCycles} we have that $c_{i_1,j} \in C^{\dsig}_{\G\backslash\alpha\cdots\mu_{i_1-1};\mu_{i_1}}$ for $j \notin \{I_{1,2},\ldots,I_{1,N_{1}}\}$. An identical argument holds for each of the vertices $\mu_{i_2},\ldots,\mu_{i_p}$.\\

    Piecing all of the above conclusions together, we find that the decomposition of $c$ matches the form of Eq.~\eqref{eqn:Thm2:FormBCDecompGroup}, and so $c$ is an element of the set $ \dress{\dsig;l}{\G}\big[q\big]$.\\
\end{enumerate}

\noindent\textit{Base case.}
     The base case of the induction is the statement $\mathbb{L}(1,l)$ for $l = 0,1,\ldots, D$. We prove this statement by considering each of the two possible graphs on a single vertex.
\begin{itemize}
         \item[a.] $\G$ has no loop. Then $\G$ contains no cycles, so $C_{\G;\alpha}$ is empty, and $\mathbb{L}(1,l)$ is vacuously true.
\item[b.] $\G$ has a loop $\alpha\alpha$. Then two cases exist:
\begin{itemize}
           \item[i)] $ 0 \le l \le D-1$. Then $Q^{\dsig;l}_{\G;\alpha}$ is empty, and $\mathbb{L}(1,l)$ is vacuously true.
           \item[ii)] $l = D$. Then $Q^{\dsig;D}_{\G;\alpha} = \{\alpha\alpha\}$. The only cycle that satisfies $\reduc{\dsig;D}[c] = \alpha\alpha$ is $c=\alpha\alpha$, which is also a member of $\dress{\dsig;D}{\G}[\alpha\alpha]$, so $\mathbb{L}(1,l)$ is true.
\end{itemize}
\end{itemize}
\end{proof}

This concludes the proof that the cycle dressing operator $\dress{\dsig;l}{\G}$ produces the pre-image of a given $(\dsig,l)$-irreducible cycle under $\reduc{\dsig;l}$.

\subsection{The walk reduction and walk dressing operators are inverses}
\label{subsec:RDeltaInverses}

In this section we prove that the walk reduction and walk dressing operators of Definitions \ref{defn:WalkReductionOperator} and \ref{defn:WalkDressingOperator} are inverses of one another, in the sense that for a given $\dsig$-irreducible walk $i$ the set $\Dress{\dsig}{\G}$ contains all walks, and only those walks, that satisfy $\Reduc{\dsig}[w] = i$.
\begin{theorem}\label{thm:RDeltaInverse}
Let $\G$ be a digraph, $\dsig$ be a dressing signature, $w \in W_{\G;\alpha\omega}$ be a walk from $\alpha$ to $\omega$ on $\G$, and $i \in I^{\dsig}_{\G;\alpha\omega}$ be a $\dsig$-irreducible walk from $\alpha$ to $\omega$ on $\G$. Then $w$ is an element of $\Dress{\dsig}{\G}[i]$ if and only if $\Reduc{\dsig}[w] = i$.
\end{theorem}

\begin{proof}
We wish to show that for a $\dsig$-irreducible walk $i$ and an arbitrary walk $w$,
\begin{align}
  w \in \Dress{\dsig}{\G}[w] \iff \Reduc{\dsig}[w] = i.
\end{align}
We prove the forward and backward directions of this result separately. As for Lemma \ref{lem:rdeltaInverses}, the central idea is very simple. In the forward direction, we wish to show that if $w$ can be obtained by adding a (possibly empty) collection of resummable cycles to $i$, then those same cycles are removed (i.e.~mapped to the corresponding trivial walk) when $\Reduc{\dsig}$ is applied to $w$. In the backward direction, we wish to show that if deleting all resummable cycles from $w$ produces $i$, then applying $\Dress{\dsig}{\G}$ to $i$ re-adds this configuration of resummable cycles, thereby reproducing $w$.

The bulk of the proof in each direction follows directly from Lemma \ref{lem:rdeltaInverses}. We begin with some preliminary remarks that will be required in both sections of the proof.\\

\noindent\textbf{Preliminary remarks.}
Since $i$ is $\dsig$-irreducible, it follows from Theorem \ref{thm:StructureOfIK} that $i$ is a member of the set $
P_{\G;\alpha\omega}
\nest Q^{\dsig,0\:*}_{\G\backslash\alpha\cdots\mu_L;\omega}
\nest \cdots
\nest Q^{\dsig,0\:*}_{\G\backslash\alpha;\mu_2}
\nest Q^{\dsig,0\:*}_{\G;\alpha}
$ and so has decomposition given by Eq.~\eqref{eqn:StructureOfIK:wstruc}. Then applying $\Dress{\dsig}{\G}$ to $i$ adds zero or more $\dsig$-structured cycles to every vertex $\mu_i$ before every child $q_{i,j}$ and after the final child $q_{i,N_i}$, and additionally applies $\dress{\dsig;0}{\G\backslash\alpha\cdots\mu_{i-1}}$ to every child cycle $q_{i,j}$, thereby mapping $i$ to the set
\begin{align*}
\Dress{\dsig}{\G}\big[i\big] = \{\alpha\mu_2\cdots\mu_L\omega\}
&\nest \left[\Nest_{j=1}^{N_{L+1}} \left(C^{\dsig\:*}_{\G\backslash\alpha\cdots\mu_L;\omega} \nest \dress{\dsig;0}{\G\backslash\alpha\cdots\mu_L}\big[q_{L+1,j}\big]  \right)\right]
\nest C^{\dsig\:*}_{\G\backslash\alpha\cdots\mu_L;\omega}\\[2mm]
&\nest \cdots\\[2mm]
&\nest \left[\Nest_{j=1}^{N_1} \left(C^{\dsig\:*}_{\G;\alpha}\nest
\dress{\dsig;0}{\G}\big[q_{1,j}\big] \right) \right]
\nest C^{\dsig\:*}_{\G;\alpha}.
\end{align*}
It follows that a walk $w$ is an element of $\Dress{\dsig}{\G}[i]$ if and only if its decomposition into a simple path plus a collection of cycles is of the form
\begin{subequations}\label{eqn:Thm2:Proof:RequiredFormGroup}
\begin{equation}
\begin{aligned}
w = \alpha\mu_2\cdots\mu_L\omega
&\nest\left[\Nest_{j=1}^{N_{L+1}}\left(\Nest_{m=1}^{P_{L+1,j}} f_{L+1,j,m} \nest d_{L+1,j}\right)\right]
\nest \left[\Nest_{m=1}^{P_{L+1,N_{L+1}+1}} f_{L+1,N_{L+1}+1,m}\right]\\[2mm]
&\nest \cdots \\[2mm]
&\nest\left[\Nest_{j=1}^{N_1}\left(\Nest_{m=1}^{P_{1,j}} f_{1,j,m} \nest d_{1,j}\right)\right]
\nest \left[\Nest_{m=1}^{P_{1,N_1+1}} f_{1,N_1+1,m}\right]
\end{aligned}
\end{equation}
for some $P_{1,1},\ldots,P_{L+1,N_{L+1}+1} \ge 0$, where for indices in the relevant ranges
\begin{equation}
f_{i,j,m} \in C^\dsig_{\G\backslash\alpha\cdots\mu_{i-1};\mu_i} \quad\text{ and } \quad
d_{i,j} \in \dress{\dsig;0}{\G\backslash\alpha\cdots\mu_{i-1}}\big[q_{i,j}\big].
\end{equation}
\end{subequations}
Here the cycles denoted $f_{i,j,m}$ are $\dsig$-structured cycles that have been added off $\mu_i$ before every child cycle $q_{i,j}$ and after the final child $q_{i,N_i}$, and $d_{i,j}$ is an element from the set of cycles formed by applying $\dress{\dsig;0}{\G\backslash\alpha\cdots\mu_{i-1}}$ to $q_{i,j}$.\\

We now proceed to the body of the proof.\\

\noindent\textbf{Forward.}
We wish to prove that for a $\dsig$-irreducible walk $i$ and an arbitrary walk $w$
\begin{align*}
\Reduc{\dsig}[w] = i \implies w \in \Dress{\dsig}{\G}\big[i\big].
\end{align*}
We will show that it follows from the equality $\Reduc{\dsig}[w] = i$ that $w$ has the structure given in Eq.~\eqref{eqn:Thm2:Proof:RequiredFormGroup}, and is thus an element of $\Dress{\dsig}{\G}\big[i\big]$.\\

\noindent Let $w$ have decomposition
\begin{align}
w = \alpha\beta_2\cdots\beta_P\omega
\nest \left[\Nest_{j=1}^{M_{P+1}}c_{P+1,j}\right]
\nest \cdots
\nest \left[\Nest_{j=1}^{M_1}c_{1,j}\right]\label{eqn:Thm2:WDecomp}
\end{align}
where $\alpha\beta_2\cdots\beta_P\omega$ is a simple path of length $P \ge 0$ and $c_{i,j} \in C_{\G\backslash\alpha\cdots\mu_{i-1}}$. Then when $\Reduc{\dsig}$ is applied to $w$, $\reduc{\dsig;0}$ is applied to every child cycle; thus
\begin{align*}
\Reduc{\dsig}\big[w\big] = \alpha\beta_2\cdots\beta_P\omega
\nest \left[\Nest_{j=1}^{M_{P+1}}\reduc{\dsig;0}\big[c_{P+1,j}\big]\right]
\nest \cdots
\nest \left[\Nest_{j=1}^{M_1}\reduc{\dsig;0}\big[c_{1,j}\big]\right].
\end{align*}
By comparing with Eq.~\eqref{eqn:StructureOfIK:wstruc}, we see that $\Reduc{\dsig}\big[w\big] = i$ can only hold if $w$ and $i$ have the same base path: thus $\alpha\beta_2\cdots\beta_P\omega = \alpha\mu_2\cdots\mu_L\omega$, and $P = L.$ Further, the equality requires that $N_i$ of the factors $\reduc{\dsig;0}\big[c_{i,1}\big],\ldots, \reduc{\dsig;0}\big[c_{i,M_i}\big]$ be equal to $q_{i,1}, \ldots, q_{i,N_i}$, and the remaining $M_i - N_i$ are equal to the trivial walk $(\mu_i)$. Let $I_{i,k}$ for $1\le k \le N_i$ index the cycles that are mapped by $\reduc{\dsig;0}$ to $q_{i,k}$, and set $I_{i,0} = 0$. Then we can rewrite the sequence of $M_1$ cycles off $\alpha$ as
\begin{align}
  \Nest_{j=1}^{M_1} c_{1,j} = \Nest_{k=1}^{N_1}\left(
\Nest_{m={I_{1,k-1}+1}}^{I_{1,k}-1}\!\!c_{1,m} \:\nest \:c_{1,I_{1,k}}\right)
\nest\Nest_{m=I_{1,N_1}+1}^{M_1} c_{1,m},
\end{align}
where
\begin{align*}
\reduc{\dsig;0}\big[c_{1,j}\big] =
\begin{cases} q_{1,k} & \text{ if $j = I_{i,k}$ for some $1\le k \le N_i$,}\\
(\mu_i) & \text{ if $j \notin \{I_{i,1},\ldots,I_{i,N_i}\}$.}
\end{cases}
\end{align*}
Then it follows from Lemma \ref{lem:rdeltaInverses} and Proposition \ref{prop:TrivialEquivalentCycles} that
\begin{align*}
c_{1,j} \in \begin{cases}
\dress{\dsig;0}{\G}\big[q_{1,k}\big] & \text{if $j = I_{i,k}$ for some $1\le k \le N_i$,}\\
C^\dsig_{\G;\alpha} & \text{if $j \notin \{I_{i,1},\ldots,I_{i,N_i}\}$.}
\end{cases}
\end{align*}
An identical argument holds for the sequence of cycles off every other vertex $\mu_2,\ldots,\omega$. The decomposition of $w$ thus matches the structure of Eq.~\eqref{eqn:Thm2:Proof:RequiredFormGroup}, and so $w$ is an element of $\Dress{\dsig}{\G}\big[i\big]$.\\

\noindent\textbf{Backward.}
We wish to prove that for a $\dsig$-irreducible walk $i$ and an arbitrary walk $w$,
\begin{align}
w \in \Dress{\dsig}{\G}\big[i\big] \implies \Reduc{\dsig}[w] = i.
\end{align}
First note that if $w$ is an element of $\Dress{\dsig}{\G}\big[i\big]$ then the decomposition of $w$ has the structure given in Eq.~\eqref{eqn:Thm2:Proof:RequiredFormGroup}. Then when $\Reduc{\dsig}$ is applied to $w$, $\reduc{\dsig;0}$ is applied to every child cycle. Since all of the cycles denoted $f_{i,j,m}$ in Eq.~\eqref{eqn:Thm2:Proof:RequiredFormGroup} are $\dsig$-structured, they are each mapped to the corresponding trivial walk $(\mu_i)$. Since the trivial walk does not contribute to the nesting product, this corresponds to deleting each of the cycles $f_{i,j,m}$. We thus find
\begin{align*}
\Reduc{\dsig}\big[w\big] &= \alpha\mu_2\cdots\mu_L\omega
\nest \left[\Nest_{j=1}^{N_{L+1}} \reduc{\dsig;0}\big[d_{L+1,j}\big]\right]
\nest \cdots
\nest \left[\Nest_{j=1}^{N_{1}} \reduc{\dsig;0}\big[d_{1,j}\big]\right]\\
&=\alpha\mu_2\cdots\mu_L\omega\nest \left[\Nest_{j=1}^{N_{L+1}} q_{L+1,j}\right]
\nest \cdots
\nest \left[\Nest_{j=1}^{N_{1}} q_{1,j}\right]\\
&= i,
\end{align*}
where the second equality follows from noting that $d_{i,j}$ is an element of $\dress{\dsig;0}{\G\backslash\alpha\cdots\mu_{i-1}}[q_{i,j}]$ and applying Lemma \ref{lem:rdeltaInverses}.
\end{proof}

We have therefore proven that the explicit definitions of $\Reduc{\dsig}$ and $\Dress{\dsig}{\G}$ given in Sections \ref{sec:WalkReductionOperators} and \ref{sec:WalkDressingOperator} satisfy the properties that we assumed in Definition \ref{defn:ReductionDressingAxioms}. It follows that the family of sets $\{\Dress{\dsig}{\G}[i] : i \in I^{\dsig}_\G\}$, where $I^\dsig_\G$ is given in Theorem \ref{thm:StructureOfIK}, form a partition of the walk set $W_\G$.
\section{The sum of all walks on a directed graph}\label{sec:S5:DressedWalkSums}
In Section \ref{sec:S3:FamilyOfPartitions} we showed that for any digraph $\G$ and dressing signature $\dsig$, which may be freely chosen, the collection of walk sets formed by applying the walk dressing operator $\Dress{\dsig}{\G}$ to every $\dsig$-irreducible walk on $\G$ forms a partition of $W_\G$. An immediate consequence of this result is that the sum of all walks in $W_\G$ can be separated into a sum over $\dsig$-irreducible walks and a sum over the walks in the set $\Dress{\dsig}{\G}[i]$ for each $\dsig$-irreducible walk $i$, with the guarantee that this restructuring does not lead to any walk in the sum being double-counted or missed out. In other words, for any dressing signature $\dsig$ we have that
\begin{align*}
  \sum_{w\,\in\,W_{\G}} w = \sum_{i \,\in\,I^{\dsig}_\G} \:\sum_{w\,\in\,\Dress{\dsig}{\G}[i]} w,
\end{align*}
where $I^\dsig_\G$ is the set of $\dsig$-irreducible walks, and $\Dress{\dsig}{\G}[i]$ is the walk dressing operator defined in Section \ref{sec:WalkDressingOperator}.

In this section we exploit this observation to develop a family of partially-resummed series representations for the sum of all walks on $\G$. The central point to note is that $\Dress{\dsig}{\G}[i]$ is a set containing $i$ plus all walks that can be obtained by nesting one or more resummable cycles into $i$. It follows that summing over all elements of $\Dress{\dsig}{\G}[i]$ is equivalent to summing over all possible configurations of resummable cycles that can be added to $i$. As we show later in this section, this summation takes the form of an infinite geometric series of resummable cycles off each vertex of $i$, which can be formally resummed to a closed-form expression. The end effect of summing over all elements of $\Dress{\dsig}{\G}[i]$ is to replace each vertex $\mu$ of $i$ by a \key{dressed vertex} $(\mu)_\G^\dsig$ which represents the sum of all possible configurations of resummable cycles on $\G$ off $\mu$. In this way the sum over all walks on $\G$ is recast as a sum over $\dsig$-irreducible walks with dressed vertices.

This section is structured as follows. In Section \ref{sec:51:ResummedCharSeries} we introduce formal power series, which provide the mathematical underpinning for the ensuing discussion of sums of walks. We define the aforementioned dressed vertex $(\mu)_\G^\dsig$ as the sum of all resummable cycles off $\mu$ on $\G$, and show that this sum has the form of an infinite geometric series of cycles and can thus be formally resummed. Finally, we present in Theorem \ref{thm:CharSeriesIsDressedKIrred} the main result of this section: the reformulation of the sum of all walks on $\G$ as a sum of dressed $\dsig$-irreducible walks.

In Section \ref{sec:52:ResummedWeightSeries} we extend the resummation results of Section \ref{sec:51:ResummedCharSeries} to the case of weighted digraphs. By following an analogous procedure to that outlined above, we show that the sum over all walk weights on a weighted directed graph can be recast as a sum of the weights of the $\dsig$-irreducible walks, each of which is dressed so as to include the weights of all possible configurations of resummable cycles off them. These results make it possible to develop new partially-resummed series representations for any problem that can be formulated as a sum over walk weights on a directed graph.

\subsection{Resummed series for the sum of all walks on a digraph}\label{sec:51:ResummedCharSeries}
In order to discuss the sum of all walks on a digraph $\G$, we must first introduce the concept of a formal power series. We remain as brief as possible, referring to the literature for further details \cite{Kuich1986,Berstel1988,Droste2009}.

\begin{definition}[Formal power series in $W_\G$]
Let $\G$ be a digraph, $W_\G$ be the set of all walks on $\G$, and $\mathbb{Z}$ be the set of integers. Then a formal power series in $W_\G$ over $\mathbb{Z}$ is a mapping $f:W_\G \rightarrow \mathbb{Z}$. The value of $f$ at a particular $w \in W_\G$, also termed the coefficient of $w$ in $f$, is denoted by $(f,w)$, and $f$ itself is written as a formal sum
\begin{align}
  f = \sum_{w \,\in \,W_\G} (f,w)\, w.
\end{align}
The space of all formal power series in $W_\G$ over $\mathbb{Z}$ is denoted by $\fps{\mathbb{Z}}{W_\G}$.
\end{definition}

Sums of walks on a digraph $\G$ are elements of $\fps{\mathbb{Z}}{W_\G}$. Of particular interest is the sum of walks that are found in a particular subset of $W_\G$.

\begin{definition}[Characteristic series of a set of walks]
Let $A \subseteq W_\G$ be a subset of the set of all walks on $\G$. Then the \key{characteristic series} of $A$ is the formal power series $f_{\!A}$ with coefficients
\begin{align}
  (f_{\!A},w) = \begin{cases}
    1 & \text{if $w \in A$,}\\
    0 & \text{if $w \notin A$,}
  \end{cases}
\end{align}
so that $f_{\!A} = \sum_{w \,\in\, A} w$.
\end{definition}

\begin{definition}[The characteristic series of $W_{\G;\alpha\omega}$]\label{defn:CharSeriesSigmaAW}
Given a digraph $\G$, the characteristic series of all walks from $\alpha$ to $\omega$ on $\G$, denoted by $\Sigma_{\G;\alpha\omega}$, is the formal power series in $W_\G$ over $\mathbb{Z}$ defined by
  \begin{align}
    \Sigma_{\G;\alpha\omega} = \sum_{w \,\in\, W_{\G;\alpha\omega}} w. \label{eqn:CharSeries}
  \end{align}
The coefficient $\big(\Sigma_{\G;\alpha\omega},w\big)$ of a walk $w$ in $\Sigma_{\G;\alpha\omega}$ is 1 if $w$ is found in $W_{\G;\alpha\omega}$, and 0 otherwise.
\end{definition}

The goal of this section is to exploit the results of Sections \ref{sec:S3:FamilyOfPartitions} and \ref{sec:S4:WalkReductionWalkDressing} to identify and factorise out various infinite geometric series of cycles that appear on the right-hand side of Eq.~\eqref{eqn:CharSeries}. These infinite geometric series can then be resummed into a formal closed form, leading to a new expression for the characteristic series $\Sigma_{\G;\alpha\omega}$. In order to carry out this decomposition of $\Sigma_{\G;\alpha\omega}$ into subseries, it is necessary to extend the nesting operation to formal power series.

\begin{definition}[Nesting product of two walk series]\label{defn:NestingWalkSeries}
Let $f$ and $g$ be formal power series in $W_\G$ over $\mathbb{Z}$. Then $f \nest g$ is the formal power series obtained by nesting every term of $g$ into every term of $f$, and multiplying the corresponding coefficients:
\begin{align}
  f \nest g = \sum_{w_1, w_2\, \in \,W_{\G} } (f,w_1) (g,w_2)\,w_1 \nest w_2.
\end{align}
\end{definition}

\begin{remark}
Definition \ref{defn:NestingWalkSeries} is consistent with defining the nesting operation to be distributive over addition, so that two formal power series $f = a w_1 + b w_2$ and $g = c w_3 + d w_4$ multiply in the expected way:
\begin{align*}
  f \nest g &= \big(a w_1 + b w_2\big)\nest\big(c w_3 + d w_4\big)\\
  &= ac\,w_1 \nest w_3 + ad\,w_1 \nest w_4 + bc\, w_2\nest w_3 + bd\, w_2 \nest w_4.
\end{align*}
\end{remark}

For a given dressing signature $\dsig$, the effect of resumming the infinite geometric series of cycles in the characteristic series $\Sigma_{\G;\alpha\omega}$ is to reduce $\Sigma_{\G;\alpha\omega}$ from a sum over all walks to a sum over $\dsig$-irreducible walks, each of which is modified so as to include the sums over all possible configurations of resummable cycles off each of its vertices. We implement this modification by replacing each vertex in the walk by a $\dsig$-dressed vertex, which represents the sum of all possible configurations of resummable cycles off that vertex.

\begin{definition}[$\dsig$-dressed vertex]\label{defn:KDressedVertex}
Let $\G$ be a digraph, $\dsig$ be a dressing signature, and $\alpha$ be a vertex on $\G$. Then the vertex $\alpha$ dressed by $\dsig$-structured cycles on $\G$ (also termed the $\dsig$-dressed vertex), denoted by $(\alpha)_{\G}^\dsig$, is the formal power series defined by
\begin{subequations}\label{eqn:KDressedVertex}
\begin{align}
(\alpha)_{\G}^\dsig = \sum_{w \,\in \,C^{\dsig *}_{\G;\alpha}} \!\!w
\end{align}
where $C^{\dsig}_{\G;\alpha}$ is the set of $\dsig$-structured cycles of Definition \ref{defn:CkCycleSets}. In other words, $(\alpha)_{\G}^\dsig$ is the characteristic series of $C^{\dsig *}_{\G;\alpha}$. Then by exploiting the definition of the Kleene star and formally resumming the geometric series of $\dsig$-structured cycles, we find
\begin{align}
  (\alpha)_{\G}^\dsig
  = \sum_{n=0}^\infty \sum_{w \,\in\,\big(C_{\G;\alpha}^\dsig\big)^n} \!\!w
  \:=\: \sum_{n=0}^\infty \Bigg[\sum_{c\,\in\,C_{\G;\alpha}^\dsig} c \,\Bigg]^n
  \equiv \bigg[1 - \sum_{c \,\in\,C_{\G;\alpha}^\dsig} c\bigg]^{-1},
\end{align}
where the first equality follows from the definition of the Kleene star (see Definition \ref{defn:KleeneStar}), the second from the definition of the nesting power of a set of cycles (see Definition \ref{defn:NestingPower}), and the third is to be taken as a definition of the short-hand expression on the right-hand side. The dressed vertex $(\alpha)^\dsig_{\G}$ is the sum of all walks off $\alpha$ on $\G$ that consist of zero or more $\dsig$-structured cycles off $\alpha$ concatenated together: that is,
\begin{align}
         (\alpha)_{\G}^\dsig =  (\alpha) + \sum_{c \in C^{\dsig}_{\G;\alpha}} c + \sum_{\substack{c_1 \in C^{\dsig}_{\G;\alpha}\\ c_2 \in C^{\dsig}_{\G;\alpha}}} c_1 \nest c_2 + \cdots
\end{align}
\end{subequations}
If $C^{\dsig}_{\G;\alpha}$ is empty then all terms past the first vanish, and $(\alpha)_{\G}^\dsig$ is equal to the trival walk $ (\alpha)$. Two simple examples of dressed vertices are depicted in Figure \ref{fig:dressed_vertex}.
\end{definition}

\begin{figure}
\centering
\includegraphics[scale=1.25]{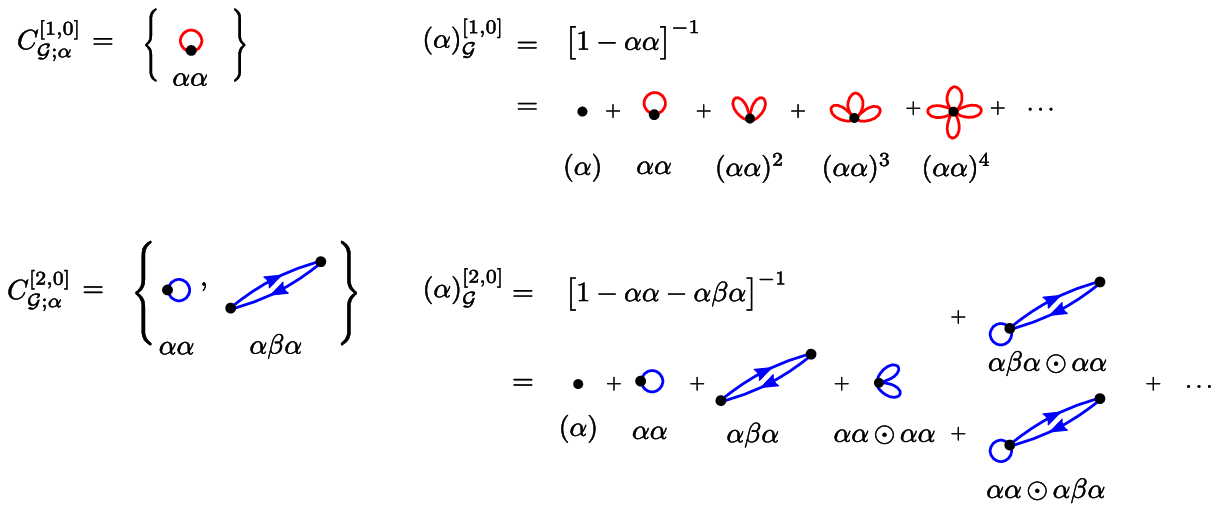}
  \caption{Two examples of a set of $\dsig$-structured cycles $C^{\dsig}_{\G;\alpha}$ (left) and the corresponding $\dsig$-dressed vertex $(\alpha)^\dsig_\G$ (right). The top example shows the case of $\dsig = [1,0]$, in which case $(\alpha)^{[1,0]}_\G$ is simply the (formal) infinite geometric series in the loop $\alpha\alpha$. The bottom example shows $\dsig = [2,0]$ in the case where only a single backtrack $\alpha\beta\alpha$ exists. Then $(\alpha)^{[2,0]}_\G$ is the infinite geometric series in $\alpha\alpha + \alpha\beta\alpha$; that is, the sum of all walks that can be formed by concatenating zero or more copies of $\alpha\alpha$ and $\alpha\beta\alpha$.}
  \label{fig:dressed_vertex}
\end{figure}

\noindent Recall from Definitions \ref{defn:KStructuredCycles} and \ref{defn:CkCycleSets} that a $\dsig$-structured cycle is defined recursively as having a base cycle not longer than $k_0$ and zero or more $[k_1,\ldots,k_{D-1},0]$-structured child cycles nested off each internal vertex. This leads to the following recursive definition for $(\alpha)^\dsig_\G$.

\begin{prop}[A closed-form expression for a $\dsig$-dressed vertex]\label{prop:ClosedFormKDressedVertex}
For a given dressing signature $\dsig = [k_0,\ldots,k_{D-1},0]$, the $\dsig$-dressed vertex $(\alpha)^\dsig_{\G}$ can be written as
\begin{align}
         (\alpha)^\dsig_{\G}
         = \left[1 -
         \sum_{L=1}^{k_0}
         \sum_{S_{\G;\alpha;L}} \alpha\mu_2\ldots\mu_L\alpha \nest (\mu_L)_{\G\backslash\alpha\cdots\mu_{L-1}}^{[k_1,\ldots,k_{d-1},0]} \nest \cdots \nest (\mu_2)_{\G\backslash\alpha}^{[k_1,\ldots,k_{d-1},0]}\right]^{-1}\label{eqn:KDressedVertexClosedForm}
\end{align}
where $\alpha\mu_2\cdots\mu_L\alpha \in S_{\G;\alpha;L}$ is a simple cycle of length $L \le k_0$ and $[k_1,\ldots,k_{d-1},0]$ is the dressing signature obtained by deleting the first element of $\dsig$.
\end{prop}
\begin{proof}
The proof follows directly from combining the definitions of the $\dsig$-dressed vertex (Definition \ref{defn:KDressedVertex}) and the $\dsig$-structured cycles (Definition \ref{defn:CkCycleSets}).
\end{proof}

We are now in a position to define an operator $\sumDress{\dsig}{\G}$ that maps a given $\dsig$-irreducible walk to the sum of all walks in the set $\Dress{\dsig}{\G}[i]$. As for the definition of $\Dress{\dsig}{\G}[i]$ in Section \ref{sec:S4:WalkReductionWalkDressing}, it is convenient to first introduce a family of auxiliary operators $\sumdress{\dsig;l}{\G}$ which act on $(\dsig,l)$-irreducible cycles. These are defined so that $\sumdress{\dsig;l}{\G}$ maps a given $(\dsig,l)$-irreducible cycle $q$ to the sum of all walks in the set $\dress{\dsig;l}{\G}[q]$.

\begin{definition}[The operators $\sumdress{\dsig;l}{\G}$]\label{defn:sumdressOperator}
Let $\G$ be a digraph, $\dsig$ be a dressing signature of depth $D$, and $0 \le l \le D$ be a parameter. Then $\sumdress{\dsig;l}{\G} : Q^{\dsig;l}_\G \rightarrow \fps{\mathbb{Z}}{W_\G}$ is a map from the set of $(\dsig,l)$-irreducible cycles on $\G$ to the space of formal power series in $W_\G$ over $\mathbb{Z}$. Its action is to map a given $(\dsig,l)$-irreducible walk $q$ to the characteristic series of the set $\dress{\dsig;l}{\G}[q]$; that is:
\begin{align}
  \sumdress{\dsig;l}{\G}\big[q\big] = \sum_{c\,\in\,\dress{\dsig;l}{\G}[q]} c.
\end{align}
Explicitly, let $q$ have decomposition
\begin{align}
  q = \alpha\mu_2\cdots\mu_L \nest \Bigg[\Nest_{j=1}^{N_L} c_{L,j}\Bigg] \nest \cdots \nest \Bigg[\Nest_{j=1}^{N_2} c_{2,j}\Bigg],
\end{align}
where $\alpha\mu_2\cdots\mu_L\alpha$ is a simple cycle of length $L \ge 1$, and $N_i \ge 0$. Then $\sumdress{\dsig;l}{\G}$ maps $q$ to the formal power series given by (cf.~Definition \ref{defn:CycleDressingOperators})
\begin{align*}
\sumdress{\dsig;l}{\G}\big[q\big] = \alpha\mu_2\cdots\mu_L\alpha
&\nest \Nest_{j=1}^{N_L} \left(
             (\mu_L)^{[k_{l_{L,j}},\ldots,k_{D-1},0]}_{\G\backslash\alpha\cdots\mu_{L-1}} \nest\sumdress{\dsig;l_{L,j}}{\G\backslash\alpha\cdots\mu_{L-1}}\big[c_{L,j}\big]\right)
\nest (\mu_L)^{\dsig_{\,L}}_{\G\backslash\alpha\cdots\mu_{L-1}}\\
&\nest \cdots \\
&\nest \Nest_{j=1}^{N_2} \left(
            (\mu_2)^{[k_{l_{2,j}},\ldots,k_{D-1},0]}_{\G\backslash\alpha} \nest\sumdress{\dsig;l_{2,j}}{\G\backslash\alpha}\big[c_{2,j}\big] \right)
\nest (\mu_2)^{\dsig_{\,2}}_{\G\backslash\alpha}
\end{align*}
where
\begin{align*}
\dsig_{\,\,i} = \begin{cases}
  [k_{l+1},\ldots,k_{D-1},0] & \text{if $L \le k_l$ and $N_i = 0$,}\\
  [k_0,\ldots,k_{D-1},0] & \text{otherwise,}
\end{cases}
\end{align*}
and the parameters $l_{i,j}$ for $2\le i \le L$ and $1 \le j \le N_i$ are given by
\begin{align*}
l_{i,j} = \begin{cases}
l+1& \text{if $L\le k_l$ and $j = 1$,}\\
0& \text{otherwise.}
\end{cases}
\end{align*}
Note that the special treatment of the first child in each hedge when $L \le k_l$ is due to the fact that this child has a local depth of $l+1$.
\end{definition}

We are now in a position to define $\sumDress{\dsig}{\G}$, which maps a given $\dsig$-irreducible walk $i$ to the sum of all walks in the set $\Dress{\dsig}{\G}[i]$.

\begin{definition}[The operator $\sumDress{\dsig}{\G}$]\label{defn:sumDressOperator}
Let $\G$ be a digraph and $\dsig$ be a dressing signature. Then $\sumDress{\dsig}{\G} : I^{\dsig}_\G \rightarrow \fps{\mathbb{Z}}{W_\G}$ is a map from the set of $\dsig$-irreducible walks on $\G$ to the space of formal power series in $W_\G$ over $\mathbb{Z}$. Its action is to map a given $\dsig$-irreducible walk $i$ to the characteristic series of the set $\Dress{\dsig}{\G}[i]$; that is:
\begin{align}
  \sumDress{\dsig}{\G}\big[i\big] = \sum_{w\,\in\,\Dress{\dsig}{\G}[i]} w.
\end{align}
Explicitly, let $i$ be a $\dsig$-irreducible walk on $\G$ with decomposition
\begin{align*}
i = \alpha\mu_2\cdots\mu_L\omega
\nest \Bigg[\Nest_{j=1}^{N_{L+1}}c_{L+1,j}\Bigg]
\nest \cdots
\nest \Bigg[\Nest_{j=1}^{N_2}c_{2,j}\Bigg]
\nest \Bigg[\Nest_{j=1}^{N_1}c_{1,j}\Bigg].
\end{align*}
Then $\sumDress{\dsig}{\G}$ maps $i$ to the formal power series (cf.~Definition \ref{defn:WalkDressingOperator})
\begin{align*}
\sumDress{\dsig}{\G}\big[i\big] = \alpha\mu_2\cdots\mu_L\omega
&\nest \Nest_{j=1}^{N_{L+1}}\left(
         (\omega)_{\G\backslash\alpha\cdots\mu_L}^{\dsig} \nest\sumdress{\dsig;0}{\G\backslash\alpha\cdots\mu_L}\big[c_{L+1,j}\big]\right)\nest \,(\omega)_{\G\backslash\alpha\cdots\mu_L}^{\dsig}\\
&\nest \:\cdots \\
&\nest \Nest_{j=1}^{N_1}\left(
         (\alpha)_{\G}^{\dsig}
         \nest \sumdress{\dsig;0}{\G}\big[c_{1,j}\big]\right)
         \nest \,(\alpha)_{\G;}^{\dsig},
\end{align*}
where $\sumdress{\dsig;0}{\G}$ is given in Definition \ref{defn:sumdressOperator}.
\end{definition}

\begin{theorem}\label{thm:CharSeriesIsDressedKIrred}
For a digraph $\G$ and an arbitrary dressing signature $\dsig$, the characteristic series $\Sigma_{\G;\alpha\omega}$ can be written as
\begin{align*}
  \Sigma_{\G;\alpha\omega} = \sum_{i \,\in \,I^{\dsig}_{\G;\alpha\omega}} \sumDress{\dsig}{\G}[i]
\end{align*}
where $I^{\dsig}_{\G;\alpha\omega}$ is the set of $\dsig$-irreducible walks from $\alpha$ to $\omega$ on $\G$, as defined in Theorem \ref{thm:StructureOfIK}, and $\sumDress{\dsig}{\G}$ is the dressing operator given in Definition \ref{defn:sumDressOperator}.
\end{theorem}
\begin{proof}
Starting from the definition of $\Sigma_{\G;\alpha\omega}$, we have
\begin{align}
  \Sigma_{\G;\alpha\omega}
  = \sum_{w \,\in \,W_{\G;\alpha\omega}} \! w
  = \sum_{i \,\in\, I^\dsig_{\G;\alpha\omega}} \sum_{w \,\in\, \Dress{\dsig}{\G}[i]} w
  =  \sum_{i \,\in\, I^\dsig_{\G;\alpha\omega}} \sumDress{\dsig}{\G}[i],
\end{align}
where the second equality follows from the fact that the family of sets $\big\{\Dress{\dsig}{\G}[i] : i \in I^\dsig_{\G;\alpha\omega}\big\}$ form a partition of $W_{\G;\alpha\omega}$ (see Theorem \ref{thm:FamilyOfPartitions}), and the third uses the definition of $\sumDress{\dsig}{\G}[i]$ (Definition \ref{defn:sumDressOperator}).
\end{proof}
Theorem \ref{thm:CharSeriesIsDressedKIrred} represents the result discussed in the introduction to this article: a representation of the sum of all walks on an digraph $\G$ as a sum over $\dsig$-irreducible walks, with each $\dsig$-irreducible walk being dressed by all possible configurations of resummable cycles.

\subsection{The sum of all walk weights on a weighted digraph}\label{sec:52:ResummedWeightSeries}
In this section we extend the results of \S\ref{sec:51:ResummedCharSeries} to the case of a weighted graph, with a view to providing a starting point for applications of walk-sums to the computation of matrix functions.
\begin{definition}[Weighted digraph]
A weighted digraph $(\G, \wt)$ is a digraph $\G$ paired with a weight function $\wt$ that assigns a weight $\wt[e]$ to each edge $e$ of $\G$. To make the discussion concrete while retaining some generality, we consider here a weight function that assigns a complex matrix of size $d_\beta \times d_\alpha$ to the edge $(\alpha\beta)$. We denote the weight of the edge $(\alpha\beta)$ by $\ew_{\beta\alpha} = \wt[(\alpha\beta)]$.
\end{definition}

\begin{definition}[Weight of a walk]
Given a weighted digraph $(\G, \wt)$, we extend the domain of the weight function $\wt$ from edges to walks by
defining the weight of a trivial walk $(\alpha) \in T_\G$ to be the $d_\alpha \times d_\alpha$ identity matrix $\mat{I}_\alpha$, and the weight of a walk $w = \alpha\mu_2\cdots\mu_L\omega$ of length $L \ge 1$ to be the right-to-left product of the weights of the edges traversed by $w$:
\begin{align*}
  \wt\big[\alpha\mu_2\cdots\mu_L\omega\big] &= \wt[(\mu_L\omega)]\cdots\wt\big[(\mu_2\mu_3)\big]\wt\big[(\alpha\mu_2)\big]\\
  &= \ew_{\omega\mu_L}\cdots\ew_{\mu_3\mu_2}\ew_{\mu_2\alpha}.
\end{align*}
Note that the ordering of the edge weights is suitable for the matrix multiplication to be carried out, the end result being a matrix of size $d_\omega \times d_\alpha$.
\end{definition}

\begin{definition}[Weight of a formal power series]\label{def:WeightOfFPS}
Given a formal power series $f$, it is natural to interpret the weight of $f$ as the value obtained on replacing each walk $w$ in $f$ by its weight $\wt[w]$; that is
\begin{align}
  \wt[f\big] = \wt\bigg[\sum_{w \,\in \,W_{\G;\alpha\omega}} \!(f,w) w\bigg] = \sum_{w \,\in \,W_{\G;\alpha\omega}} \!(f,w) \,\wt[w].
\end{align}
Note that this series does not necessarily converge: its behaviour depends on the details both of both the weight function $\wt$ and the formal power series $f$.
\end{definition}

\begin{remark}[Weight of a $\dsig$-dressed vertex]\label{rem:WeightOfKDressedVertex}
  Given a weighted digraph $(\G,\wt)$ and a dressing signature $\dsig = [k_0,\ldots,k_{D-1},0]$, it follows from Proposition \ref{prop:ClosedFormKDressedVertex} and Definition \ref{def:WeightOfFPS} that the weight of the $\dsig$-dressed vertex $(\alpha)^\dsig_\G$ -- assuming it converges -- is given by
\begin{align*}
  \wt\big[(\alpha)^\dsig_{\G}\big]
  = \left[1 - \sum_{L=1}^{k_0} \sum_{S_{\G;\alpha;L}}
  \ew_{\alpha\mu_L}
  \wt\big[(\mu_L)^{[k_1,\ldots,k_{D-1},0]}_{\G\backslash\alpha\cdots\mu_{L-1}}\big]
  \cdots
  \ew_{\mu_3\mu_2}
  \wt\big[(\mu_2)^{[k_1,\ldots,k_{D-1},0]}_{\G\backslash\alpha}\big]
  \ew_{\mu_2\alpha}
  \right]^{-1}.
\end{align*}
Note that since $\wt\big[(\mu)^{[0]}_\G\big] = \wt\big[(\mu)\big] = \mat{I}_\mu$, the recursion terminates after $D$ iterations. A sufficient condition for $\wt\big[(\alpha)^\dsig_{\G}\big]$ to converge is that
\begin{align*}
  \bigg\vert\sum_{L=1}^{k_0} \sum_{S_{\G;\alpha;L}}
  \ew_{\alpha\mu_L}
  \wt\big[(\mu_L)^{[k_1,\ldots,k_{D-1},0]}_{\G\backslash\alpha\cdots\mu_{L-1}}\big]
  \cdots
  \ew_{\mu_3\mu_2}
  \wt\big[(\mu_2)^{[k_1,\ldots,k_{D-1},0]}_{\G\backslash\alpha}\big]
  \ew_{\mu_2\alpha} \bigg\vert < 1,
\end{align*}
that is, that the sum of the weights of all $\dsig$-structured cycles off $\alpha$ on $\G$ has norm (or modulus, for scalar weights) less than 1.
\end{remark}

As the final result of this section, we present an expression for the sum of the weights of all walks from $\alpha$ to $\omega$ on $\G$ -- assuming this sum converges -- as a sum over $\dsig$-irreducible walks.

Note that it follows from Definition \ref{def:WeightOfFPS} that the sum of the weight of all walks from $\alpha$ to $\omega$ on a $\G$ is the weight of the characteristic series $\Sigma_{\G;\alpha\omega}$ of Definition \ref{defn:CharSeriesSigmaAW}:
\begin{align}
\wt\big[\Sigma_{\G;\alpha\omega}\big] = \sum_{w\,\in\,W_{\G;\alpha\omega}} \wt\big[w\big].\label{eqn:CharSeriesWeight}
\end{align}
The results of Sections \ref{sec:S3:FamilyOfPartitions} and \ref{sec:S4:WalkReductionWalkDressing} can now be exploited to identify and resum infinite geometric series appearing in the right-hand side of Eq.~\eqref{eqn:CharSeriesWeight}.

\begin{cor}\label{cor:SumOfAllWalkWeights}
Let $(\G,\wt)$ be a weighted directed graph, $\alpha$ and $\omega$ be two vertices in $\G$, and $\dsig$ be an arbitrary dressing signature, which has depth $D$. Then the sum of the weights of all walks from $\alpha$ to $\omega$ -- assuming it converges -- is given by
\begin{align}
\wt\big[\Sigma_{\G;\alpha\omega}\big] = \sum_{i\,\in\, I^\dsig_{\G;\alpha\omega}} \wt\big[\sumDress{\dsig}{\G}[i]\big],
\end{align}
where $I^\dsig_{\G;\alpha\omega}$ is the set of $\dsig$-irreducible walks from $\alpha$ to $\omega$ on $\G$, and $\wt\big[\sumDress{\dsig}{\G}[i]\big]$ is defined as follows. Let $i$ have decomposition
\begin{align*}
i = \alpha\mu_2\cdots\mu_L\omega
\nest \bigg[\Nest_{j=1}^{N_{L+1}}q_{L+1,j}\bigg]
\nest \cdots
\nest \bigg[\Nest_{j=1}^{N_2}q_{2,j}\bigg]
\nest \bigg[\Nest_{j=1}^{N_1}q_{1,j}\bigg],
\end{align*}
where $q_{i,j}$ is a $(\dsig,0)$-irreducible cycle. Then $\wt\big[\sumDress{\dsig}{\G}[i]\big]$ is given by
\begin{align*}
  \wt\big[\sumDress{\dsig}{\G}[i]\big] =
  &\hphantom{\times}\,\, \wt\big[(\omega)^\dsig_{\G\backslash\alpha\cdots\mu_{L}}\big]
  \sideset{}{^\prime}\prod_{j=1}^{N_{L+1}}
  \bigg(\wt\big[\sumdress{\dsig;0}{\G\backslash\alpha\cdots\mu_{L}}[q_{L+1,j}]\big]
  \wt\big[(\omega)^\dsig_{\G\backslash\alpha\cdots\mu_{L}}\big] \bigg)\, \ew_{\omega\mu_L}\\
  & \times\, \cdots\\
  &\times \wt\big[(\mu_2)^\dsig_{\G\backslash\alpha}\big]
  \sideset{}{^\prime}\prod_{j=1}^{N_{2}}
  \bigg(\wt\big[\sumdress{\dsig;0}{\G\backslash\alpha}[q_{2,j}]\big]
  \wt\big[(\mu_2)^\dsig_{\G\backslash\alpha}\big] \bigg)\,\ew_{\mu_2\alpha}\\
  &\times \wt\big[(\alpha)^\dsig_{\G}\big]
  \sideset{}{^\prime}\prod_{j=1}^{N_{1}}
  \bigg(\wt\big[\sumdress{\dsig;0}{\G}[q_{1,j}]\big]
  \wt\big[(\alpha)^\dsig_{\G}\big] \bigg)
\end{align*}
where the primes on the products indicate that they are to be taken right-to-left, the weights of the $\dsig$-dressed vertices are evaluated according to Remark \ref{rem:WeightOfKDressedVertex}, and $\wt\big[\sumdress{\dsig;0}{\G}[q_{i,j}]\big]$ is defined as follows. Let $q$ be a $(\dsig,l)$-irreducible cycle with decomposition
\begin{align}
q = \alpha\mu_2\ldots\mu_L\alpha
\nest \bigg[\Nest_{j=1}^{N_L}c_{L,j}\bigg]
\nest \cdots
\nest \bigg[\Nest_{j=1}^{N_2}c_{2,j}\bigg].
\end{align}
Then $\wt\big[\sumdress{\dsig;l}{\G}[q]\big]$ for $0 \le l \le D$ is given by
\begin{align*}
\wt\big[\sumdress{\dsig;l}{\G}[q]\big] =& \hphantom{\times}\:\:
\ew_{\alpha\mu_L}\wt\big[(\mu_L)^{\dsig_L}_{\G\backslash\alpha\cdots\mu_{L-1}}\big]
  \sideset{}{^\prime}\prod_{j=1}^{N_{L}}
  \bigg(\wt\big[\sumdress{\dsig;l_{L,j}}{\G\backslash\alpha\cdots\mu_{L-1}}[c_{L,j}]\big]\,
  \wt\big[(\mu_L)^{[k_{l_{L,j}},\ldots,k_{D-1},0]}_{\G\backslash\alpha\cdots\mu_{L-1}}\big] \bigg)\\
  &\times\,\cdots\\
  &\times \ew_{\mu_3\mu_2}\wt\big[(\mu_2)^{\dsig_2}_{\G\backslash\alpha}\big]
  \sideset{}{^\prime}\prod_{j=1}^{N_{2}}
  \bigg(\wt\big[\sumdress{\dsig;l_{2,j}}{\G\backslash\alpha}[c_{2,j}]\big]\,
  \wt\big[(\mu_2)^{[k_{l_{2,j}},\ldots,k_{D-1},0]}_{\G\backslash\alpha}\big] \bigg)\\
  &\times \ew_{\mu_2\alpha}
\end{align*}
where
\begin{equation*}
  \dsig_i = \begin{cases}
    [k_{l+1},\ldots,k_{D-1},0] & \text{if $L \le k_l$ and $N_i = 0$,}\\
    [k_0,\ldots,k_{D-1},0] & \text{otherwise,}
  \end{cases}
\end{equation*}
and
\begin{equation*}
  l_{i,j} = \begin{cases}
    l+1 & \text{if $L \le k_l$ and $j = 1$,}\\
    0 & \text{otherwise.}
  \end{cases}
\end{equation*}
\end{cor}

\begin{proof}
  This result is an immediate consequence of the resummed form of the characteristic series $\sigma_{\G;\alpha\omega}$ (see Theorem \ref{thm:CharSeriesIsDressedKIrred}) and the convention that the weight of a formal power series is the sum of the weights of the individual terms (Definition \ref{def:WeightOfFPS}). The definitions of $\wt\big[\sumdress{\dsig;l}{\G}[q]\big]$ and $\wt\big[\sumDress{\dsig}{\G}[i]\big]$ follow from applying the weight function $\wt$ to the formal series $\sumdress{\dsig;l}{\G}[q]$ (see Definition \ref{defn:sumdressOperator}) and $\sumDress{\dsig}{\G}[i]$ (see Definition \ref{defn:sumDressOperator}).
\end{proof}

An important consequence of Corollary \ref{cor:SumOfAllWalkWeights} is that any problem that can be written as a sum of weights of all walks on a weighted directed graph can be recast as a sum of dressed weights of $\dsig$-irreducible walks, for any dressing signature $\dsig$. Such an expression could be termed an \key{irreducible walk sum} formulation of the original problem. Since the results of Sections \ref{sec:S3:FamilyOfPartitions} and \ref{sec:S4:WalkReductionWalkDressing} hold for arbitrary values of $\dsig$, many different irreducible walk sum expressions for the same series can be found. If the original series is absolutely convergent, then the irreducible walk sum expression for any value of $\dsig$ is also absolutely convergent.

Since the dressing signature can be freely chosen, the most appropriate irreducible walk sum expression can be selected and applied to a given problem. This choice should be made on a case-by-case basis subject to the following trade-off: as $\dsig$ is progressively increased from $[0]$ (corresponding to the original series) a wider and wider variety of terms (those that correspond to $\dsig$-structured cycles) are exactly resummed into closed form. Then on one hand, there remain correspondingly fewer terms to be explicitly summed over (the $\dsig$-irreducible walks), and the irreducible walk sum expression might be expected to converge more quickly than the original series due to the partial series resummation that has already been carried out; but on the other, the weights of the $\dsig$-dressed vertices become increasingly complex to evaluate, taking the form of branched continued fractions of increasing depth and breadth. Upon reaching the maximum possible dressing signature $\dsig_{\,\mathrm{max}}$, there remain only a \key{finite} number of terms to be explicitly summed (i.e.~the simple paths), and the walk sum reduces to a path sum. The applications of such path sum expressions to matrix functions \cite{Giscard2013,Giscard2014} and statistical inference \cite{Giscard2014a} have recently been investigated.
\section{Conclusion and Outlook}\label{sec:S6:Conclusion}
In this article we defined a family of partitions of $W_\G$, the set of all walks on an arbitrary directed graph $\G$. Each partition in this family is indexed by a dressing signature $\dsig$: an integer sequence that identifies a set of cycles on $\G$ with a common well-defined structure. We term these cycles resummable cycles, and a walk that does not traverse any such cycles for a given value of $\dsig$ a $\dsig$-irreducible walk. The cells in the corresponding partition of $W_\G$ then each contain a $\dsig$-irreducible walk $i$ plus all walks that can be formed from $i$ by adding one or more resummable cycles to $i$.

The central objects in constructing this family of partitions are the walk reduction operator $\Reduc{\dsig}$, which serves to map any walk $w$ to its $\dsig$-irreducible `core walk' $\Reduc{\dsig}[w]$ by deleting all resummable cycles from it, and the walk dressing operator $\Dress{\dsig}{\G}$, which maps a $\dsig$-irreducible walk $i$ to the set of all walks that can be formed by adding a (possibly empty) collection of resummable cycles to $i$. We provided explicit definitions of these operators, and showed that these definitions satisfy the properties required in order that the family of sets obtained by applying $\Dress{\dsig}{\G}$ to every $\dsig$-irreducible walk forms a partition of $W_\G$. We further gave an explicit expression for $I^\dsig_\G$, the set of all $\dsig$-irreducible walks on $\G$. This expression is valid for an arbitrary dressing signature $\dsig$.

A direct consequence of these results is that the sum of all walks on an arbitrary digraph $\G$ can be recast as a sum only over $\dsig$-irreducible walks, where each $\dsig$-irreducible walk has been dressed by all possible configurations of resummable cycles. This reformulation corresponds to the exact resummation of infinite geometric series of cycles appearing within the sum over all walks. By choosing different dressing signatures, many different such `partially-resummed' walk series can be obtained. The set of $\dsig$-irreducible walks that appear in the sum, and the details of how each walk is dressed, are determined by the value of $\dsig$ in each case.

The ability to carry out exact resummations of infinite geometric series appearing within the sum over all walks on an arbitrary directed graph promises to have applications to problems that are naturally formulated as a sum of such walks. Examples of such problems include the computation of matrix functions by series summation \cite{Giscard2013}, and the problem of statistical inference in Gaussian graphical models \cite{Malioutov2006,Chandrasekaran2008}. By applying the resummations outlined here, a variety of resummed series -- termed irreducible walk sum expressions -- for the original problem can be derived. For a given value of $\dsig$, the corresponding irreducible walk sum expression is the result of exactly resumming the contributions of all terms in the original series that correspond to $\dsig$-structured cycles on the underlying graph. The convergence behaviour, numerical stability, and mathematical properties of the irreducible walk sum expressions remain interesting open questions.

\section*{Acknowledgments}
This work was conducted under the auspices of a Theodor Heuss Research Fellowship at the Ludwig Maximilian University of Munich. The author would like to thank the Alexander von Humboldt Foundation for financial support, the groups of Ulrich Schollw\"{o}ck and Lode Pollet for their hospitality, and Pierre-Louis Giscard for helpful comments on the manuscript.

\bibliographystyle{siam}

\end{document}